\documentclass{amsart}
\usepackage{amsmath} 
\usepackage{amssymb}
\usepackage{mathtools}
\usepackage{centernot}
\usepackage{stmaryrd}
\usepackage{esvect}
\usepackage{amsthm}
\usepackage{upgreek}
\usepackage{comment}
\usepackage{amsfonts}
\usepackage{physics}
\usepackage{tikz-cd}
\usepackage[utf8]{inputenc}
\usepackage[english]{babel}
\usepackage{graphicx}
\usepackage[margin=1in]{geometry}
\usepackage{mathtools}

\usepackage{float}

\usepackage{hyperref}
\usepackage{microtype}

\graphicspath{{images/}}

\renewcommand{\part}[1]{\noindent\textbf{Part #1)}}

\newcommand{\ba}{\begin{align*}}

\newcommand{\ea}{\end{align*}}

\newcommand{\Int}[1]{%
  {\kern0pt#1}^{\mathrm{o}}%
}

\renewcommand{\S}{\textbf{S}}

\newcommand{\bp}{\begin{pmatrix}}
\newcommand{\ep}{\end{pmatrix}}

\newcommand{\interior}[1]{%
  {\kern0pt#1}^{\mathrm{o}}%
}

\newcommand{\pmc}{\backslash \backslash}
\newcommand{\gnorm}[1]{\lvert \lvert {#1} \rvert \rvert} 

\newtheorem{theorem}{Theorem}[section]
\newtheorem*{theorem*}{Theorem}
\newtheorem{conjecture}[theorem]{Conjecture}

\newtheorem*{claim*}{Claim}
\newtheorem{corollary}[theorem]{Corollary}
\newtheorem{lemma}[theorem]{Lemma}

\theoremstyle{definition}
\newtheorem{definition}{Definition}[section]
\newtheorem{example}{Example}[section]

\makeatletter
\newtheorem*{rep@theorem}{\rep@title}
\newcommand{\newreptheorem}[2]{%
\newenvironment{rep#1}[1]{%
 \def\rep@title{#2 \ref{##1}}%
 \begin{rep@theorem}}%
 {\end{rep@theorem}}}
\makeatother

\newreptheorem{theorem}{Theorem}
\newreptheorem{lemma}{Lemma}
\newreptheorem{conjecture}{Conjecture}

\usepackage{xargs}                      
\usepackage{xcolor}
\usepackage[colorinlistoftodos,prependcaption,textsize=tiny]{todonotes}
\begin{document}

\title{TG-hyperbolicity of Composition of Virtual Knots}

\author[C. Adams]{Colin Adams}
\address{Department of Mathematics and Statistics, Williams College, USA}
\email{cadams@williams.edu}

\author[A. Simons]{Alexander Simons}
\address{Mathematical Sciences Building,  One Shields Ave., University of California, Davis, CA 95616}
\email{asimons@math.ucdavis.edu}

\begin{abstract}
The composition of any two nontrivial classical knots is a satellite knot, and thus, by work of Thurston,  is not hyperbolic. In this paper, we explore the composition of virtual knots, which are an extension of classical knots that generalize the idea of knots in $S^3$ to knots in $S \times I$ where $S$ is a closed orientable surface. We prove that for any two hyperbolic virtual knots, there is a composition that is hyperbolic. We then obtain strong lower bounds on the volume of the composition using information from the original knots. 
\end{abstract}

\maketitle











\noindent {\textsc{Acknowledgements}}
We want to thank the other members of the 2020 SMALL Knot Theory group: Michele Capovilla-Searle, Jacob McErlean, Darin Li, Lily Li, Natalie Stewart, and Miranda Wang. Many of the ideas in this paper were instigated by conversations we had during SMALL.


\setcounter{tocdepth}{5}

\section{Introduction}\label{section}


Given two knots $K_1$ and $K_2$, in the 3-sphere, we can take their \emph{composition} by choosing a projection of each, with a choice of the outer region, removing a trivial arc from an edge of the outer region, and gluing each endpoint from $K_1$ to an  endpoint of $K_2$ while preserving the projections. This operation generates a new knot denoted $K_1 \# K_2$. If either knot is invertible, the resulting composite knot is unique. If neither knot is invertible,  two distinct knots  can result, which, once we fix orientations on the knots, correspond to either composing so the orientations match to yield an orientation on the composition or composing so they do not match. 



This idea of composition of knots extends to virtual knots, which were introduced in  \cite{Kauffman virtual knot intro}. 
For a projection of a virtual knot, we allow a new third type of crossing called a \emph{virtual crossing} and denoted by a circle around the crossing.
Any knot diagram with a virtual crossing is called a \emph{virtual knot diagram}. The other crossings are called \emph{classical crossings}, and any knot diagram with only classical crossings is referred to as a \emph{classical knot diagram}, but we include it as a type of virtual knot diagram with zero virtual crossings. Just as Reidemeister moves provide equivalence between classical knot diagrams, there are additional virtual Reidemeister moves that provide equivalence between virtual knot diagrams \cite{Kauffman virtual knot intro}.




In fact, as in \cite{Stable equivalence} and \cite{KK}, the study of virtual knots is equivalent to the study of knots in thickened closed orientable surfaces modulo stabilizations and destabilizations.  The \emph{genus} of a virtual knot refers to the genus $g$ of the minimal genus thickened surface that contains a representative of $K$. When we say a virtual knot is contained in a thickened surface $S \times I$, this means the knot can be realized by a knot projection on $S$ with only classical crossings on $S$. In \cite{Kuperberg} it was proved that a virtual knot or link has a unique realization as a knot or link in a minimal genus thickened surface up to homeomorphism sending knot to knot.

We are particularly interested in hyperbolic knots. Given a knot $K$ in the 3-sphere, if its complement can be endowed with a hyperbolic metric, then we say that $K$ is \emph{hyperbolic}. The Mostow-Prasad Rigidity Theorem \cite{Mostow}, \cite{Prasad} says that if a 3-manifold has a finite volume hyperbolic metric, the metric, and therefore the corresponding volume, is unique.  If two knots $K_1$ and $K_2$ have hyperbolic complements in $S^3$ with different volumes, then the manifolds defined by their complements must be distinct, and therefore $K_1$ and $K_2$ are distinct as well.
 
For manifolds that have higher genus boundary, the volume is either infinite or not well-defined. We circumvent this problem by restricting to tg-hyperbolicity. 
A compact orientable 3-manifold $M$ is \emph{tg-hyperbolic} if, after capping off all sphere boundaries with balls, and removing all torus boundaries, the resulting manifold $M'$ has a finite volume hyperbolic metric such that all remaining boundary components are totally geodesic. We also say $M'$ is a tg-hyperbolic manifold.
Mostow-Prasad Rigidity holds for all tg-hyperbolic 3-manifolds, and thus each has a well-defined volume. 

A surface properly embedded in a 3-manifold is \emph{essential} if it is incompressible, boundary incompressible, and not boundary parallel. Thuston proved that a compact orientable 3-manifold is tg-hyperbolic if and only if it contains no essential spheres, annuli or tori.

This paper focuses on hyperbolic virtual knots. Since virtual knots generalize the idea of knots in the 3-sphere to knots in thickened surfaces, hyperbolicity of virtual knots generalizes in the same way. Given a virtual knot $K$, we consider the manifold obtained by taking the complement of $K$ in the minimal genus thickened surface containing it. If the resulting manifold is tg-hyperbolic, then we say that $K$ is tg-hyperbolic, and can calculate a hyperbolic volume for $K$. 

See \cite{tg hyperbolicity in S cross I} for a more complete  introduction to hyperbolicity for virtual knots. That paper includes a table of the volumes of the 116 nontrivial virtual knots of four or fewer classical crossings calculated via the computer program SnapPy \cite{CDG}, all of which, with the exception of the trefoil knot, turn out to be tg-hyperbolic. Note that if a virtual knot $K$ is realized in $S \times I$ as a tg-hyperbolic knot, then the genus of $K$ is the genus of $S$.  This is true because a tg-hyperbolic manifold contains no essential annuli, and therefore it cannot contain a reducing annulus, the removal of a regular neighborhood of which would allow (after capping holes) for a decrease in genus.

In \cite{Thurston}, Thurston proved that all classical knots fall into three disjoint categories: torus knots, satellite knots, and hyperbolic knots. All compositions of classical knots fall into the category of satellite knots and thus none are hyperbolic. So hyperbolic invariants are useless for studying composition of classical knots. However, virtual knots behave very differently. As we will see, any two non-classical tg-hyperbolic virtual knots have a composition that is a tg-hyperbolic virtual knot. 

 Just as we define composition of classical knots, we can define composition of virtual knots in terms of projections. We choose a region of each projection to be the exterior region, remove an arc from an edge of that region,  and then glue the two projections together at their endpoints, obtaining a projection of the composition. However, in the category of virtual knots, two virtual knots can have infinitely many distict compositions coming from complicating the original projections before composing. 
 
 Composition of virtual knots was first discussed in \cite{KM}. We include the details in Section \ref{sec: virt comp}. The method of composition described above translates into two different possibilities when considered for the corresponding knot in a thickened surface $S \times I$. For each knot that is to be composed, the operation depends on a choice of a disk $D$ in $S$ such that $D \times I$, called the cork $C$, 
 intersects the knot in a single unknotted arc. A virtual knot $K$ in a thickened surface $S \times I$ with cork $C$ is denoted $(S \times I, K, C).$ We will often refer to this as a \emph{triple}. 

In the case that there is a nontrivial simple closed curve on the surface that intersects the disk $D$ corresponding to the cork in a single connected arc and intersects the knot projection only once (possibly after isotopy of the knot), which will be inside the disk, we say that the cork is \emph{singular}. If no such curve exists, the cork is \emph{nonsingular}.

Composition of two knots with singular corks, composition of two knots with nonsingular corks, and composition of a knot with a singular cork with a knot with a nonsingular cork, must all be dealt with separately. 
In the most common case, with two nonsingular corks, the corks are removed from the two thickened surfaces and the remaining boundaries of the removed corks in the thickened surfaces are glued to one another. 

Throughout what follows, all results apply to both knots and links, although we often use $K$ to represent either a knot or link, and in discussion we sometimes refer to a knot when it could also be a link.

The main results of this paper are stated below. We prove that  two non-classical tg-hyperbolic virtual knots can always be composed to obtain a tg-hyperbolic virtual  knot, and then bound the volume of this new manifold obtained from taking the knot complement from below. Note that the composition of a virtual knot with a classical knot will never be hyperbolic as there will always be an essential annulus present.


\begin{reptheorem}{thm: singular tg-hyperbolicity}
Given two non-classical virtual knots or links  that are tg-hyperbolic, with triples\\
$(S_1 \times I, K_1, C_1), (S_2 \times I, K_2, C_2)$ such that both $C_1$ and $C_2$ are singular, the singular composition \newline $(S_1 \times I, K_1, C_1) \#_s (S_2 \times I, K_2, C_2)$ is tg-hyperbolic. Further, $g(K_1 \# K_2) = g(K_1) + g(K_2) -1$.
\end{reptheorem}

\medskip



Although Theorem \ref{thm: singular tg-hyperbolicity} applies for any pair of singular corks, a corresponding result for nonsingular corks does not hold. In fact, every tg-hyperbolic knot in a thickened surface has a choice of nonsingular cork that will prevent its composition with any other tg-hyperbolic knot from being tg-hyperbolic. For instance, we can tie a slip knot in our knot and then insert the cork, as in Figure \ref{slipknot}(a) or (b). This will cause there to be an essential torus in the resulting composite knot as in Figure \ref{slipknot}(c), precluding tg-hyperbolicity. So hypotheses in the following theorems about nonsingular corks have been included to avoid this possibility.

\begin{figure}[htbp]
    \centering
    \includegraphics[width=.9\textwidth]{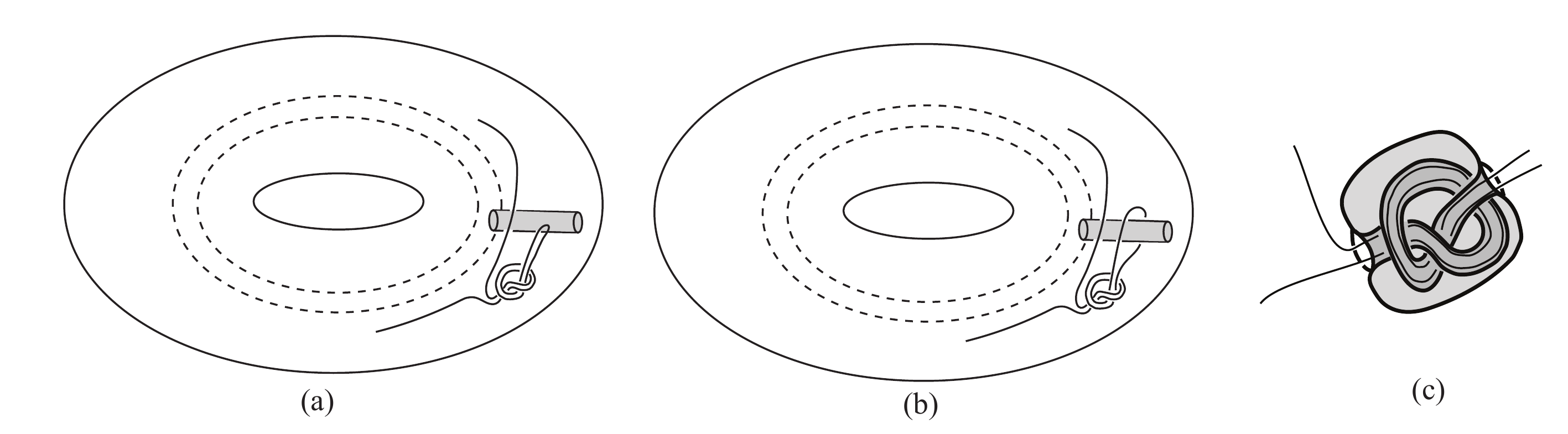}
    \caption{For the corks depicted, the composition with any other knot in a thickened surface cannot be tg-hyperbolic because of the presence of an essential torus.}
    \label{slipknot}
\end{figure}

\begin{definition} A tg-hyperbolic triple $(S \times I, K, C)$ with nonsingular cork $C$ is \emph{hyperbolically composable} if there exists another triple $(S' \times I, K', C')$ with nonsingular cork $C'$ such that their composition is tg-hyperbolic.
\end{definition}



\begin{reptheorem}{thm: hyperbolicallycomposable} The nonsingular composition of two hyperbolically composable  triples $(S_1 \times I, K_1, C_1)$ and $(S_2 \times I, K_2, C_2)$ is tg-hyperbolic, with genus $g(K_1 \# K_2) = g(K_1) + g(K_2)$.
\end{reptheorem}

Note that in \cite{KM}, Kauffman and Manturov proved that $g(K_1) + g(K_2) \geq g(K_1 \# K_2) \geq g(K_1) + g(K_2) - 1$. Since singular composition yields a thickened surface of genus $g(K_1) + g(K_2) - 1$, this must be the genus. But in the case of nonsingular composition, it is the fact that the resulting knot complement is tg-hyperbolic that ensures the resulting genus is $g(K_1) + g(K_2)$.

In this paper, we sometimes need to create manifolds by removing one properly embedded submanifold $C$ from another $M$. The formal way to do this is to take the \emph{path-metric completion of $M \setminus S$}, denoted $M \pmc S$. 

\begin{definition}
Given a submanifold $C$ properly embedded in a 3-manifold $M$, $M \pmc C$ is obtained by removing a regular neighborhood of $C$ from $M$ and then taking the path-metric completion. 
\end{definition}

\begin{definition} The \emph{nonsingular cork double} $D_{ns}(S \times I, K, C)$ of a triple $(S \times I, K, C)$ with nonsingular cork $C$ represents 
the nonsingular composition of a virtual knot $(S \times I, K , C)$ with its image $(S \times I, K , C)^R$ reflected across a plane perpendicular to the projection plane. If we have a singular cork, the \emph{singular cork double} $D_s(S \times I, K,C)$ represents the singular composition of $(S \times I, K , C)$ with its image reflected across a plane perpendicular to the projection plane.  
\end{definition}

  \begin{reptheorem}{thm: double tg-hyp implies comp tg-hyp}  Let $(S_1 \times I, K_1, C_1)$ and $(S_2 \times I, K_2, C_2)$ be two non-classical triples such that their doubles are tg-hyperbolic. Then, the composition, singular if both corks are singular and nonsingular otherwise,  is tg-hyperbolic. If both corks are singular, the genus is $g(S_1) + g(S_2) -1$ and otherwise the genus is $g(S_1) + g(S_2)$.
\end{reptheorem}


Thus, even if a virtual knot  is not tg-hyperbolic, but, for a given choice of cork in some realization of the knot in a thickened surface, its corresponding double is tg-hyperbolic, then for any two such cases, the composition is again tg-hyperbolic. %

The 4-crossing virtual knots 4.4, 4.5, 4.54. 4.55, 4.56, 4.74, 4.76 and 4.77 are all composite knots that are examples of this, where at least one of the two factor knots is in fact a trivial knot, but with cork chosen in a genus one realization of the trivial knot so that the double is tg-hyperbolic.

In general, it is not obvious if a given cork is singular or nonsingular, and if it is nonsingular, whether or not the corresponding double is tg-hyperbolic. So we include the following.

\begin{reptheorem}{thm: composition tg-hyperbolicity}
Given two non-classical virtual knots or links that are tg-hyperbolic,  there is a choice of corks so that the composition, singular if both corks are singular and nonsingular otherwise, is tg-hyperbolic. In the case of singular composition, the resulting genus is $g(K_1) + g(K_2) -1$ and otherwise it is $g(K_1) + g(K_2)$.
\end{reptheorem}

Doubling allows us to define a notion of volume that will be useful for computing volume bounds when composing different knots in thickened surfaces. 

\begin{definition}
The \emph{nonsingular composite volume}, denoted $vol_{ns}(S \times I, K, C)$, of a knot $K$ with nonsingular cork $C$ is obtained by taking half of the volume of the cork double of $K$:  $$vol_{ns}(S \times I, K, C) = \frac{1}{2}vol(D_{ns}(S \times I, K, C))$$
\end{definition}

Note that this volume is only defined when $C$ is a nonsingular cork and the cork double is tg-hyperbolic.
When $C$ is singular, we have two different ways to calculate volume bounds based on the type of composition. The first is using an analogue of nonsingular composite volume for a singular composition. (The second will occur when we are composing triples, one with a singular cork and one with a nonsingular cork.) 

\begin{definition}
The \emph{singular composite volume}, denoted $vol_s(S \times I, K, C)$, of a knot $K$ with a singular cork $C$ is obtained by taking half of the volume of the singular cork double of $K$:  $$vol_s(S \times I, K, C) = \frac{1}{2} vol(D_s(S \times I, K, C))$$
\end{definition}

This now allows us to use results from \cite{AST} (as generalized in \cite{CFW}) to say something interesting about the volume of the composition of tg-hyperbolic knots  living in thickened surfaces when both corks are singular or both corks are nonsingular. 

\begin{reptheorem}{thm: nonsing compose nonsing vol bound}
Given two non-classical hyperbolically composable triples $(S_1 \times I, K_1, C_1)$ and $(S_2 \times I, K_2, C_2)$, the volume of the composition $M = (S_1 \times I, K_1, C_1) \#_{ns} (S_2 \times I, K_2, C_2)$ is bounded below by the nonsingular composite volumes of $(S_1 \times I, K_1, C_1)$ and $(S_2 \times I, K_2, C_2)$:
$$vol (M) \geq vol_{ns}(S_1 \times I, K_1, C_1) + vol_{ns}(S_2 \times I, K_2, C_2)  $$

\end{reptheorem}


\begin{reptheorem}{thm: sing compose sing vol bound}
 Let $(S_1 \times I, K_1, C_1)$ and $(S_2 \times I, K_2, C_2)$ be two tg-hyperbolic triples such that both $C_1$ and $C_2$ are singular. If both have genus one or both have genus at least two, then the volume of the composition $M = (S_1 \times I, K_1, C_1) \#_s (S_2 \times I, K_2, C_2) $ is bounded below by the singular composite volumes of $(S_1 \times I, K_1, C_1)$ and $(S_2 \times I, K_2, C_2)$:
$$vol (M) \geq vol_s(S_1 \times I, K_1, C_1) + vol_s(S_2 \times I, K_2, C_2)  $$
\end{reptheorem}


The most interesting case for bounding volume occurs when one cork is singular and the other is nonsingular. A full discussion of how to bound the volume of the composition of this case is found in Section \ref{section: volume bounds}. \\

\begin{reptheorem}{thm: nonsing compose sing vol bound}
    Let $(S_1 \times I, K_1, C_1)$ and $(S_2 \times I, K_2, C_2)$ be two non-classical tg-hyperbolic triples such that $C_1$ is singular and $(S_2 \times I, K_2, C_2)$ is hyperbolically composable. Then the volume of the composition $M = (S_1 \times I, K_1, C_1) \#_{ns} (S_2 \times I, K_2, C_2)$ is bounded below: 
    $$vol(M) \geq \frac{1}{4}vol(D(D_{ns}(S_1 \times I, K_1, C_1)) \pmc T) + \frac{1}{2}vol(D_{ns}(S_2 \times I, K_2, C_2))$$
\end{reptheorem}

In Section 5, we provide evidence that the choice of projections and corks is crucial when composing virtual knots.

\begin{reptheorem}{thm: vol limits to infinity} 
Given two tg-hyperbolic non-classical virtual knots  $K_1, K_2$  such that both $K_1$ and $K_2$ are alternating, there exists a sequence of tg-hyperbolic compositions $W_i = (S_1 \times I, K_1, C_1^i) \#_{ns} (S_2 \times I, K_2, C_2)$ such that $$\lim_{i \to \infty}(vol(W_i)) = \infty.$$
\end{reptheorem}

We conjecture that the restriction to virtual knots that are alternating is unnecessary.

 Fixing a particular diagram of a knot $K$ with $c$ classical crossings on a surface $S$, the crossings cut the projection into $2c$ arcs. Each is a candidate for placement of a cork upon it. So there are $2c$ corks corresponding to a given diagram that are distinct. Since we can alter projections, there are an infinite number of corks corresponding to a given knot $K$ in a surface. 
      
\begin{definition} If a knot $K$ is in a reduced alternating diagram on $S$, we call any cork that corresponds to one of the $2c$ possibilities for that diagram an {\it alternating cork}. 
    
\end{definition}

We show that for genus at least 1, all alternating corks are nonsingular. Then we show:



\begin{reptheorem}{thm: alternatingcork} Let $K$ be a non-classical alternating virtual knot that has 
 a weakly prime projection on $S$. Then for any choice of alternating cork, $(S \times I, K, C)$ is hyperbolically composable.
    \end{reptheorem}

\bigskip  \noindent


Section \ref{sect: explicit examples}  includes explicit examples, as well as a table of volume contributions to allow the reader to use these ideas without having to perform time-intensive volume computations.
\newline

Before moving on to prove these theorems, we turn for a moment to boundaries of 3-manifolds. Since tg-hyperbolic manifolds are atoroidal, it follows that any tg-hyperbolic manifold obtained by taking a knot complement in a thickened torus does not include the inner and outer boundaries $S \times \{0, 1\}$. However, when discussing knot complements in thickened surfaces with genus $g \geq 2$, we assume that their boundaries $S \times \{0, 1\}$ are included. These boundaries are taken to be totally geodesic. This is because when the genus of $S$ is greater than one, Mostow-Prasad Rigidity only holds for thickened surfaces $S \times I$ with totally geodesic boundary.

At first, this seems to present an issue with composing genus one and genus two (or more) knots. However, we get around this issue by doubling. For any genus one knot in a nonsingular composition, we first remove its cork, and then double it over what remains of the cork boundary to obtain a genus two surface that includes its boundary. Then we take half of this double to compose, and therefore everything being composed is \emph{boundary compatible} in that both pieces include their boundary. Since we can also consider singular composition as doing nonsingular composition and then reducing along a reducing annulus and capping off (see Section \ref{sec: virt comp}), this approach also works to prove tg-hyperbolicity in the case of singular composition when both factor knots have genus two or greater. This approach, however, fails to work when generating volume bounds for singular composition of a genus one knot with a genus two (or more) knot. This is why  Theorem \ref{thm: sing compose sing vol bound} is stated to exclude the composition of a genus one knot with a genus at least two knot. The case of singular composition when both factor knots have genus one is simple because the result is a genus one knot, so the boundary is not included.

If instead of knots in thickened surfaces, we consider knots in handlebodies, results on compositions appear in \cite{AS}.

\section{Defining Composition of Virtual Knots}\label{sec: virt comp}
 In \cite{KM}, Kauffman and Manturov prove that there are two types of composition for virtual knots. Start with two virtual knots $K_1$ and $K_2$, embedded in minimal genus thickened surfaces $S_1 \times I$ and $S_2 \times I$. The first type of connected sum is constructed in the following way:

In each $S_i \times I$, choose a disk $D_i$ with $D_i \subset \partial (S_i \times I)$ such that $(D_i \times I) \cap K_i$ is an unknotted arc, and call each $D_i \times I$ a \emph{cork}. Denote the cork $C_i$. Then remove each $C_i$ and glue the resulting manifolds together by identifying the boundaries of the removed corks in $S_i \times I \setminus C_i$ so that the endpoints of the knots are identified. This yields $M = (S_1 \# S_2) \times I$ and contains $K_1 \# K_2$.  We refer to this construction as \emph{nonsingular composition}, and denote it $(S_1 \times I, K_1, C_1) \#_{ns} (S_2 \times I, K_2, C_2)$. We include Figure \ref{fig: nonsingular composition example} here, which shows an example of nonsingular composition. More examples are included in Section \ref{sect: explicit examples}.\\ 

\begin{figure}[htbp]
    \centering
    \includegraphics[width=0.4\textwidth]{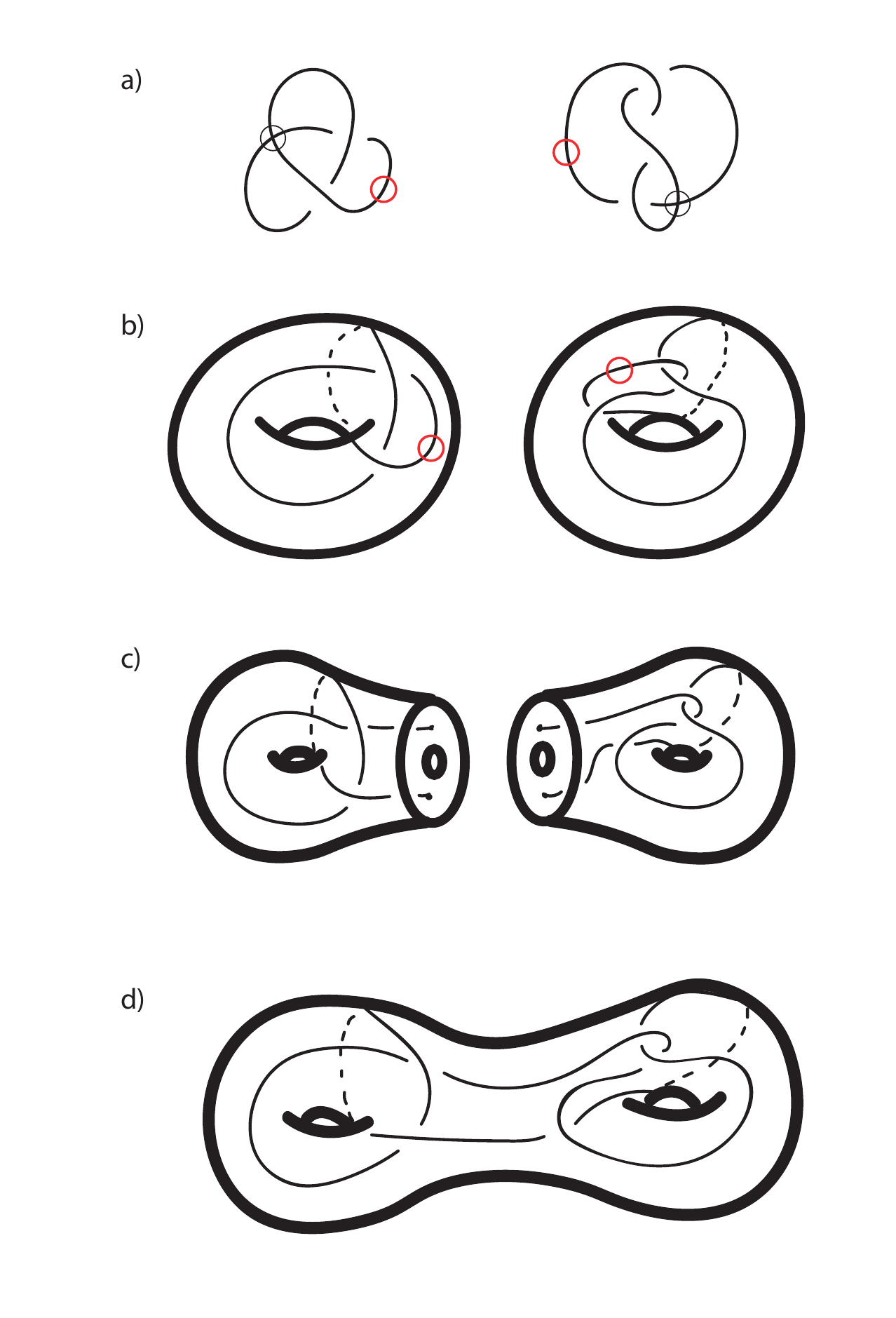}
    \caption{An example of nonsingular composition. a) and b) show the diagram and thickened surface representation of the two knots. c) shows each $S_i \times I$ after the cork has been removed, and d) shows the final composition. For simplicity, the inner surface boundary $S_i \times \{0\}$ is not drawn in.}
    \label{fig: nonsingular composition example}
\end{figure}

In the literature, virtual knots are often referred to by their \emph{surface-link pair} realizations, and denoted $(S, K)$ where $K$ is the virtual knot embedded in a minimal genus thickened surface $S \times I$. We adapt this idea, and use a triple notation throughout this paper to keep track of the thickened surface $S \times I$ containing the knot $K$, as well as the choice of cork $C$ (which defines the composition). We denote these triples $(S \times I, K, C)$. We use $S\times I$ and not $S$ because $S \times I$ is the manifold we are concerned with, and $S\times I$ contains $K$ and $C$ as submanifolds. We define $M$ to be the knot complement: $M = (S \times I) \setminus K$ when $S$ has genus at least 2 and $M = (S \times (0,1)) \setminus K$ when $S$ has genus one. We say that $(S \times I, K, C)$ is a \emph{tg-hyperbolic triple} when the manifold $M$ is tg-hyperbolic.

The second type of composition is possible only in certain cases. Before defining it, we must first define a \emph{singular curve}.

\begin{definition} \label{def: gen sing curve}
Let $K$ be a knot in a thickened surface $S \times I$ with genus $g > 0$, and let $I = [0,1]$. A \emph{singular curve} $\gamma$ is a nontrivial curve on $S$ such that $\gamma \times \{0, 1\}$ bounds an annulus $A$ in $S \times I$ punctured exactly once by $K$. 
\end{definition}

If there exists a projection $P$ of a virtual knot $K$ onto a projection surface $S$ that intersects a nontrivial curve exactly once, that is a singular curve. But in general, given a projection, it may be difficult to determine if a singular curve exists. Figure \ref{fig: virt tref singular curve} gives an example of a singular curve for a  projection of the virtual trefoil.

\begin{figure}[htbp]
    \centering
    \includegraphics[width=0.35\textwidth]{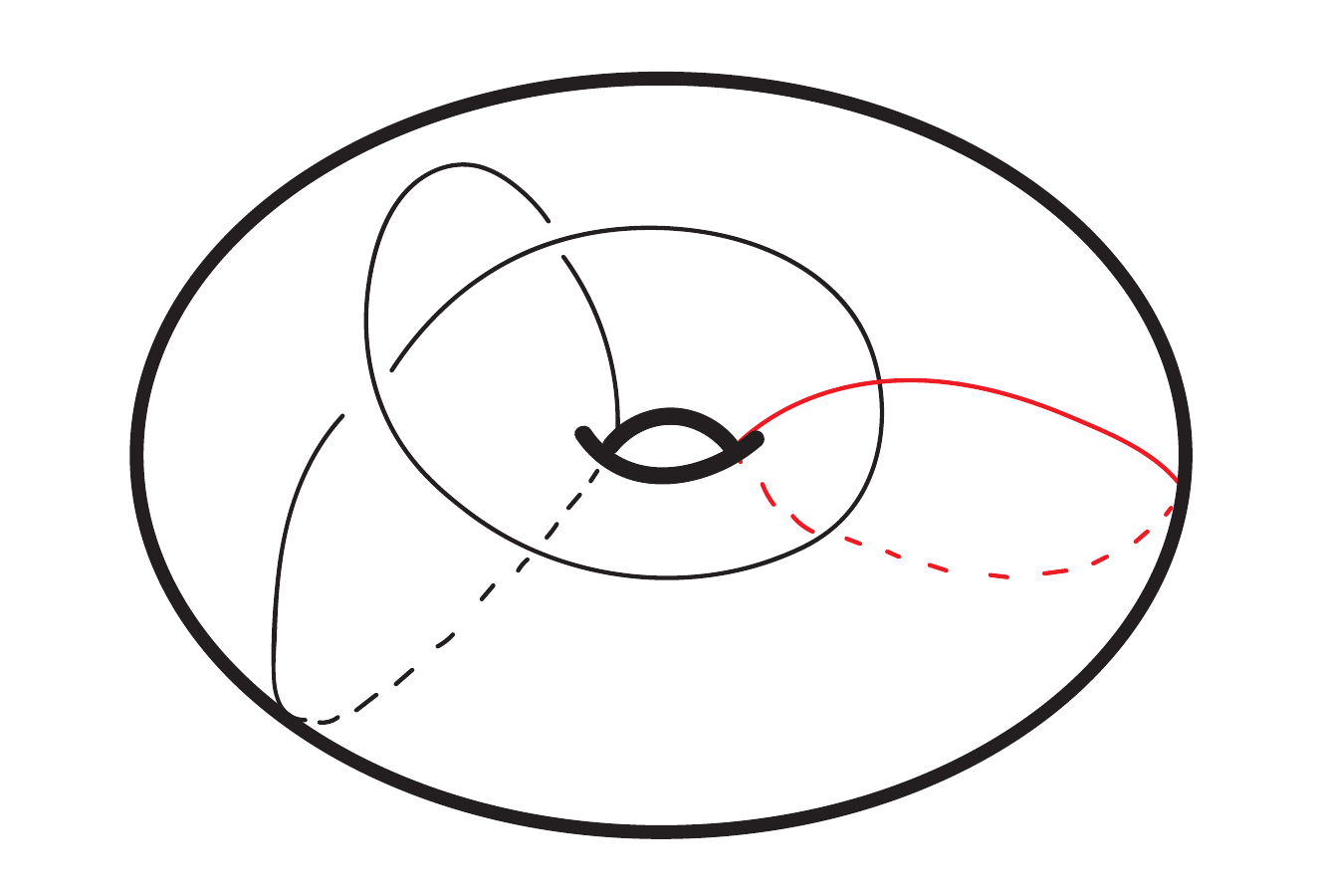}
    \caption{The virtual trefoil drawn on a torus. A singular curve is drawn in red.}
    \label{fig: virt tref singular curve}
\end{figure}

We next define \emph{singular composition}. 
If at least one singular curve exists for a minimal genus surface associated with a virtual knot $K$, then we refer to that virtual knot as \emph{singular}. 

Let $K_1, K_2$ be knots living in thickened surfaces $S_i \times I$ for $i = \{1, 2\}$ such that both $K_1, K_2$ are singular. Then each $S_i \times I$ has a singular curve $\gamma_i$ defining a once-punctured annulus $A_i$. We then compose $K_1$ and $K_2$ by cutting each $S_i \times I$ along $A_i$ defined by $\gamma_i$, and identifying the boundaries $\partial A_i$ to obtain $S \times I = (S_1 \# S_2) \times I$ containing $K_1 \# K_2$. Furthermore, $g(S) = g(S_1) + g(S_2) - 1$. We refer to this composition of two singular knots by cutting and pasting along singular curves as a \emph{singular composition}, and denote it with the symbol $\#_s$. Figure \ref{fig: ex cutting open singular knot} shows an example of cutting open $M_1$ along $A_1$. 

We note that there is some ambiguity in this definition. After cutting open each $S_i \times I$, there are two gluing boundaries for each $S_i \times I$. This implies there is a choice about which boundary glues to which between the two manifolds. In general, both of the possible choices are valid singular compositions.


\begin{figure}[htbp]
    \centering
    \includegraphics[width=0.3\textwidth]{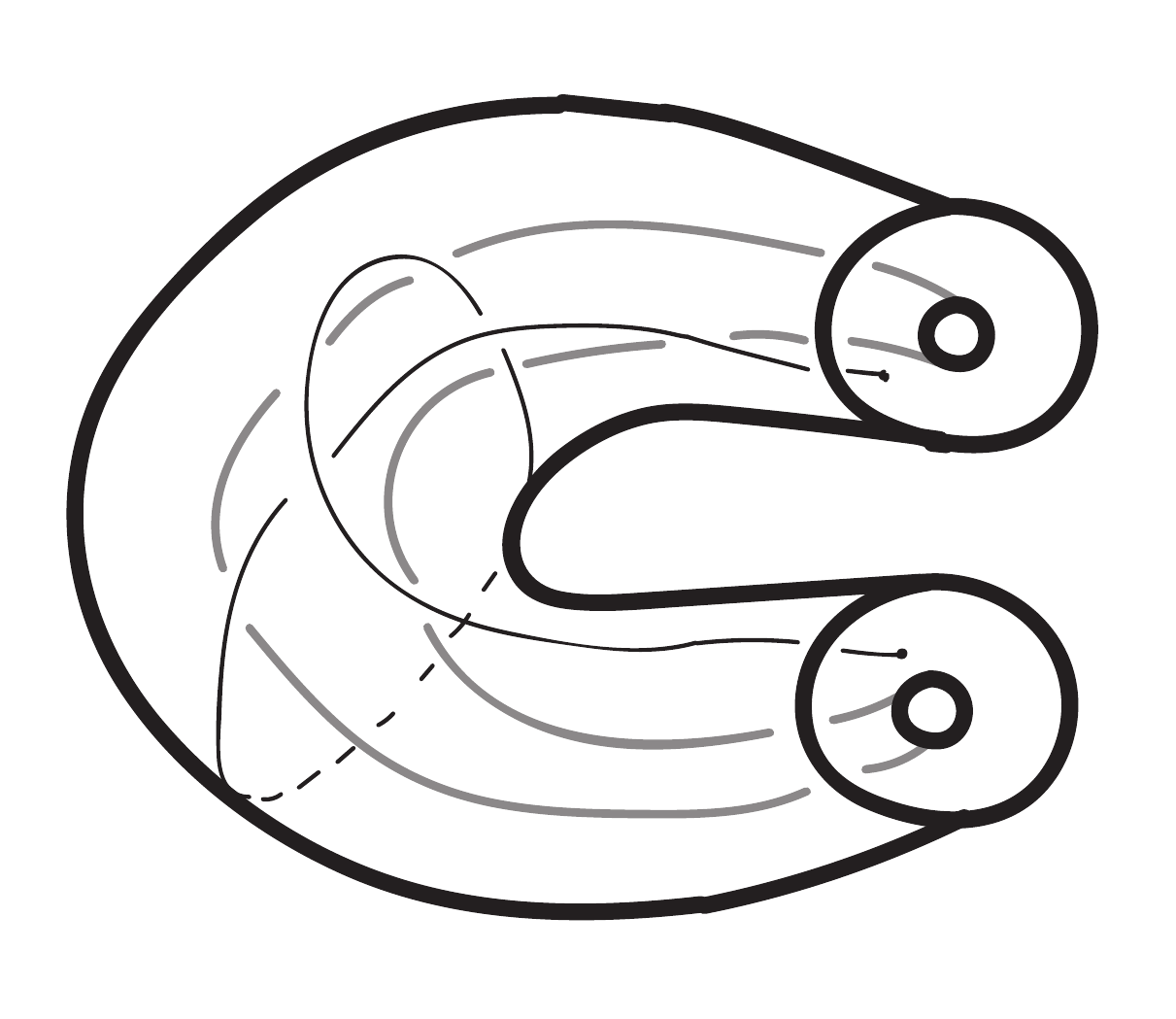}
    \caption{The result of cutting open a genus one thickened surface containing a singular knot. We cut along the once-punctured annulus $A$ defined by the singular curve $\gamma$ shown in Figure \ref{fig: virt tref singular curve}.}
    \label{fig: ex cutting open singular knot}
\end{figure}

Also, we can alternatively define singular composition via the previously defined nonsingular composition with corks as appears below. The definition of a singular cork will be crucial for thinking about singular composition using corks. 

\begin{definition} \label{def: singular cork}
Let $K$ be a knot living in a thickened surface $S \times I$. Recall that a cork $C$ is a thickened disk $C = D \times I$ such that $C \cap K$ is just an unknotted arc. We say that a cork is \emph{singular} if there exists a disk $E$ in $M \pmc C$ such that $\partial E$ consists of four arcs,  two arcs of which are on $\partial C$ that together separate the two endpoints of $K$ on $\partial C$, one arc of which is on the inner boundary of $M$, and one arc of which is on the outer boundary of $M$. Further, $E$ cannot be isotoped 
into $\partial C$. If no such disk exists, then we call $C$ a \emph{nonsingular cork}. 
\end{definition} 

We call such corks singular because in the presence of such a cork $C$, there is an annulus punctured once by $K$ in $C \cup E$, a boundary of which corresponds to a singular curve. Note that the issue with slipknots as depicted in Figure \ref{slipknot} does not occur for singular corks.

 We are now ready to define singular composition using corks. As in the nonsingular case, we can compose two triples $(S_1 \times I, K_1, C_1) \#_s (S_2 \times I, K_2, C_2)$ both with singular corks $C_1$ and $C_2$ by removing the corks and pasting along boundaries of the now-removed corks. Since both corks were singular, each $M_i \pmc C_i$ contains a disk $E_i$ with two arcs on $C_i$. When composing the two manifolds, we can glue these two disks together along their arcs on the boundaries of $C_i$. Then $E_1 \cup E_2$ will be an annulus $A$ in the resulting manifold $M$. (An annulus of this type is called a reducing annulus in the literature.) We can cut along $A$, and cap off each resulting copy of $A$ with a thickened disk  to yield a manifold $M'$. Importantly, $M'$ is homeomorphic to the manifold obtained by performing singular composition by cutting and pasting along once-punctured annuli defined by singular curves.

 \medskip

There are three important notes to make about the two types of composition detailed above. First, these do in fact correspond to all possible compositions of diagrams of virtual knots \cite{KM}. For examples of how composition with manifolds relates to the diagrammatic composition of virtual knots, the reader should refer to Section \ref{sect: explicit examples}. 
Second, the definition of singular cork relies heavily on the definition of a singular curve. This means that corks that visually appear to be nonsingular may actually be singular. For an example of this, see Figure \ref{fig: reduced necessary example}.

\begin{figure}[htbp]
    \centering
    \includegraphics[width=0.5\textwidth]{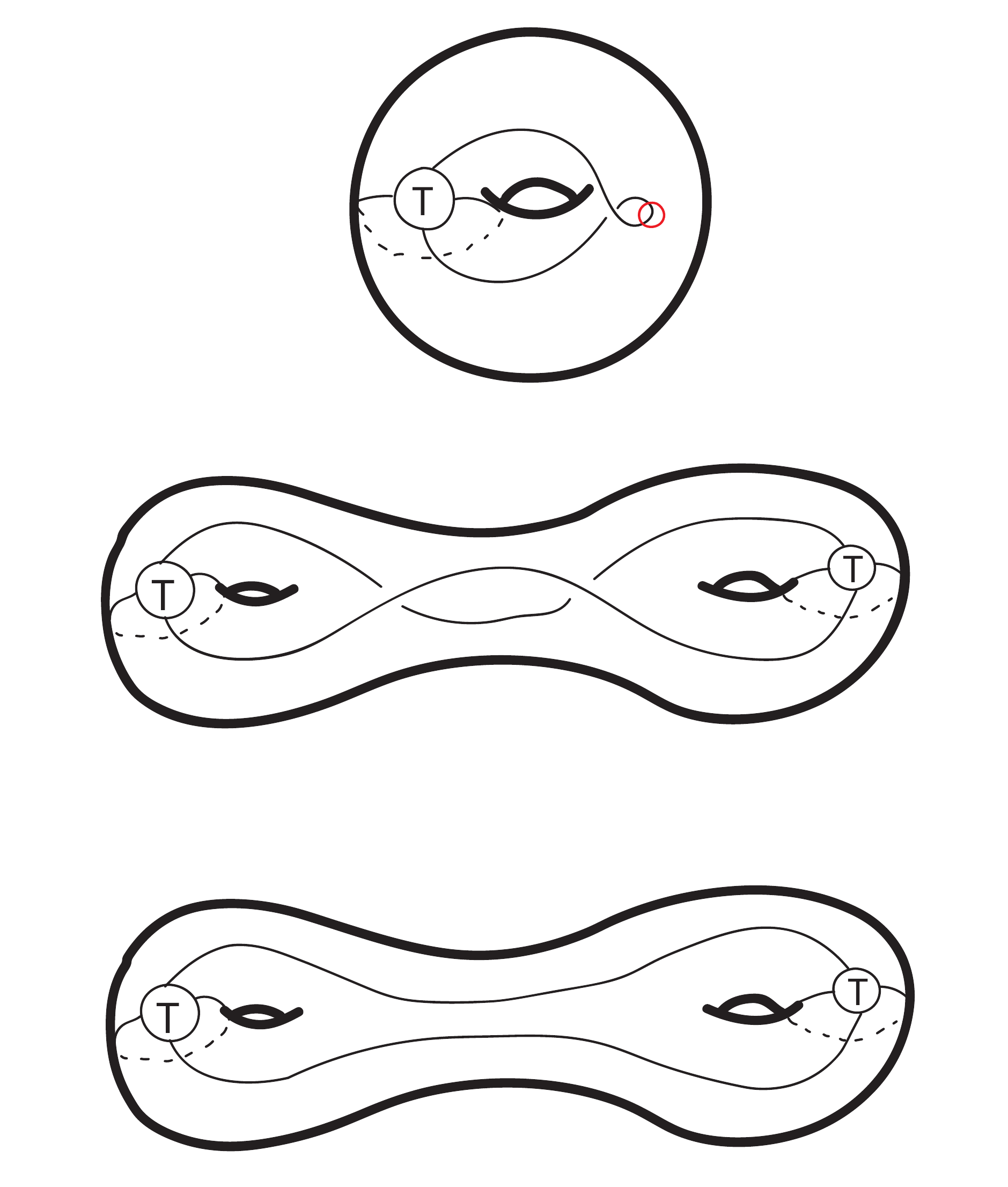}
    \caption{An example of a choice of cork (shown in red) that visually appears to be nonsingular, but is actually a singular cork. The nonsingular composition of the knot with itself has an obvious reducing annulus.}
    \label{fig: reduced necessary example}
\end{figure}

Third, it is important to clarify that a singular knot can be involved in a nonsingular composition. For example, if $K_1$ is singular and $K_2$ is nonsingular, then all possible compositions $K_1 \# K_2$ will be nonsingular compositions. Furthermore, even in a singular knot it is possible to choose a nonsingular cork. It is the singularity of the corks that determines the composition; the fact that a knot is singular simply means there are possible choices of singular corks.

\section{Proving TG-Hyperbolicity} \label{sect: tg-hyperbolicity}

In this section, we prove Theorems \ref{thm: singular tg-hyperbolicity}, \ref{thm: hyperbolicallycomposable} and \ref{thm: composition tg-hyperbolicity}. While these theorems use similar proof techniques, we prove them in separate subsections for readability. Before beginning the proofs, we include a little bit of background and terminology that we will use throughout. 

Throughout this section, we let $M_i = S_i \times I \setminus K_i$, $K = K_1 \# K_2$ and let $M$ be the complement of the composite knot in the resulting thickened surface of genus $g(S_1) + g(S_2) -1$ when the composition is singular, and $g(S_1) + g(S_2)$ when the composition is nonisngular.

First, we show that any essential disk $D$ in either a singular or nonsingular composition must imply the existence of an essential sphere $F$. This is important because then eliminating essential spheres will also eliminate essential disks. 

\begin{lemma} \label{lemma: disk generates sphere}
Given two triples $(S_1 \times I, K_1, C_1), (S_2 \times I, K_2, C_2)$,  any essential disk $D$ in either a singular or nonsingular composition $M$ must generate an essential sphere $F$ in $M$.
\end{lemma}
\begin{proof}

We first consider the case in which $D$ has boundary on a component $K'$ of $K_1 \# K_2$. Because the thickened surface can be placed in $S^3$, we know $\partial D' = K'$. Then we can take the boundary of a regular neighborhood of $D$ to yield an essential sphere $F$ containing $K'$ in $M$.

Next, consider $D$ with boundary on one of the boundary surfaces of $M$. Since $\partial M$ is incompressible, we  know $\partial D$ must bound a disk $E$ on that surface. Then for $D$ to be essential,  $D \cup E$ must be an essential sphere in $M$. 
\end{proof}



 \subsection{Singular Composition is TG-Hyperbolic} \label{subsect: sing comp tg-hyperbolicity} 

In this subsection, we prove Theorem \ref{thm: singular tg-hyperbolicity}.  Our goal is to eliminate all possible essential disks, spheres, annuli, and tori through casework on the possible intersections between these surfaces and the doubling surface. 

In considering these arguments, it is easiest to think about singular composition as defined by cutting and pasting along thickened singular curves (instead of removing corks, gluing, cutting along an annulus and then adding thickened disks). Recall that given a knot $K$ living in a thickened surface $M = S \times I$, a singular curve $\gamma$ defines a once-punctured annulus $A$, with the single puncture coming from $K$. A singular cork $C$ also defines a once-punctured annulus $A$. Briefly, we perform singular composition $(S_1 \times I, K_1, C_1) \#_s (S_2 \times I, K_2, C_2)$ by cutting each $M_i$ along the once-punctured annulus $A_i$ defined by the singular curve, and then gluing them together along these new boundaries. We call the union of these two new once-punctured annuli $X$, and denote each $X_1$ and $X_2$. See Figure \ref{fig: sing comp ex labeled X} for an example.   

\begin{figure}[htbp]
    \centering
    \includegraphics[width=0.8\textwidth]{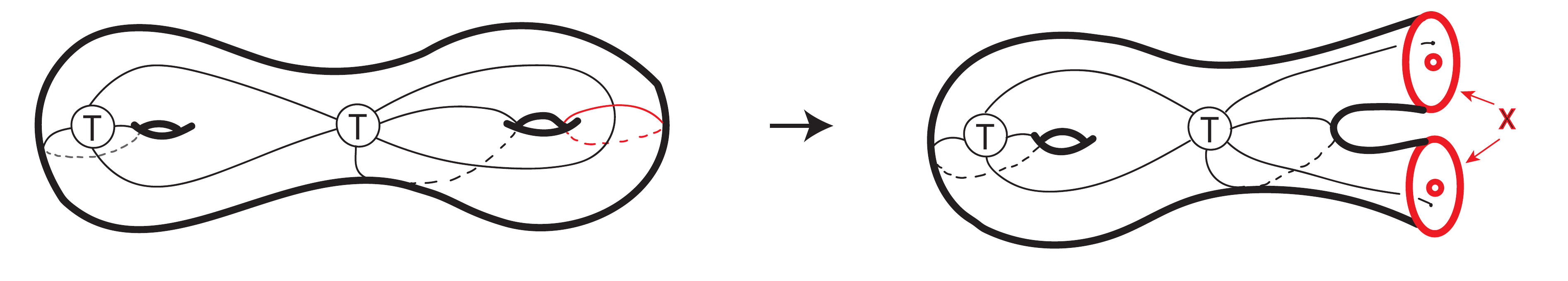}
    \caption{An example of cutting a genus 2 knot along a singular curve, and the resulting surfaces $X_1$ and $X_2$ used in composition.}
    \label{fig: sing comp ex labeled X}
\end{figure}

First, we prove that $X$ is both incompressible and boundary incompressible. 
\begin{lemma}
Given two tg-hyperbolic triples $(S_1 \times I, K_1, C_1), (S_2 \times I, K_2, C_2)$ such that both $C_1, C_2$ are singular corks,  the gluing surfaces $X_1$ and $X_2$ are incompressible and boundary incompressible in $M = (S_1 \times I, K_1, C_1) \#_s (S_2 \times I, K_2, C_2)$.
\end{lemma}
\begin{proof}
For contradiction, suppose that $X_1$ is compressible. Then it has a compressing disk $D$ bounded by a circle $\alpha$ on $X_1$. Since $\alpha$ cannot bound a disk on $X_1$, it must bound either the hole from the inner surface, the puncture from $K$, or both. Obviously, $\alpha$ cannot encircle the hole from the inner surface, since $\alpha$ bounds the disk $D$ in $M$, and so must be trivial with respect to $\partial M$. So $\alpha$ must encircle just the puncture from $K$. But then we can take the once punctured disk $A$ defined by $\alpha$ on $X_1$, and glue this to $D$ along $\alpha$. This yields a sphere punctured exactly once by $K$, which is not possilb e in a thickened surface since it lives in $S^3$. Thus the compressing disk $D$ could not have existed, and so $X_1$ must be incompressible. The same argument implies $X_2$ is incompressible.

It is straightforward to see that each $X_i$ is boundary incompressible. For contradiction, suppose that some arc $\alpha$ that cannot be isotoped to the boundary $\partial M$ also bounds a disk $D$ with a single arc $\beta$ on $\partial M$. Obviously, there are only two choices for $\alpha$, up to isotopy. These two possibilities are shown in Figure \ref{fig: bdry incompressible}. Both cases can be eliminated with the same arguments, just with respect to the inner versus the outer boundary of $M$. Since $\alpha \cup \beta$ bounds a disk, we know that $\alpha \cup \beta$ is trivial with respect to $\partial M$. Therefore we can isotope $\beta$ to be entirely on $\partial X_i$, and then isotoping $\beta$ slightly off of the boundary and onto $X_i$ we get a compressing disk $E$ for $X_i$, which is impossible by the arguments above. Therefore each $X_i$ is boundary incompressible. 
\end{proof}

Note that the incompressibility of $X_1$ and $X_2$ in $M$ also implies incompressibility of the corresponding once-punctured annuli in each of $(S_1 \times I, K_1, C_1)$ and  $(S_2 \times I, K_2, C_2)$. 

\begin{figure}[htbp]
    \centering
    \includegraphics[width=0.4\textwidth]{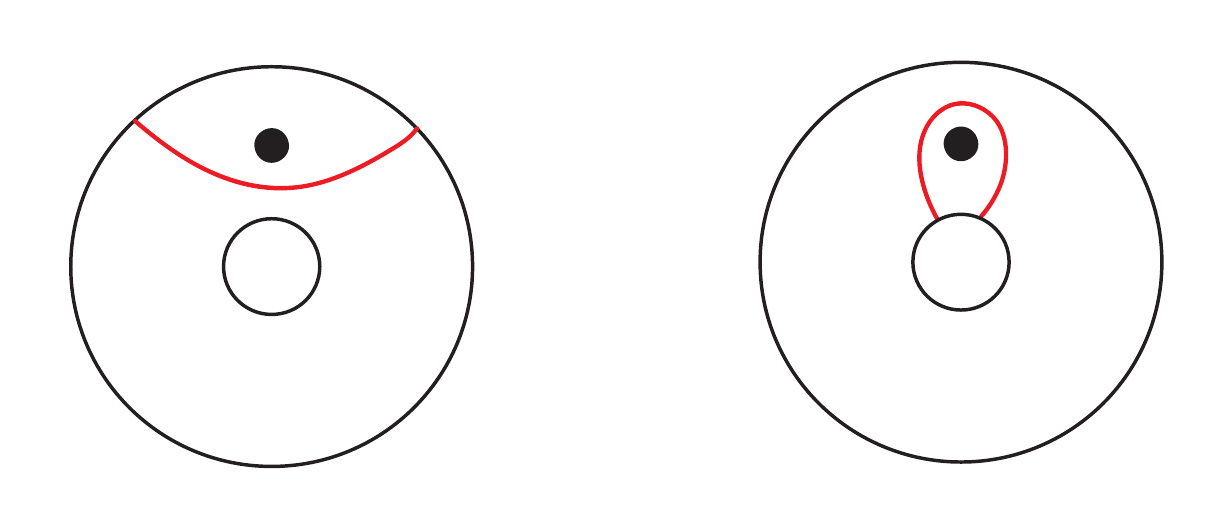}
    \caption{The two possible choices for $\alpha$ are shown in red.}
    \label{fig: bdry incompressible}
\end{figure}

We want to eliminate all essential spheres, annuli and tori, as a theorem of Thurston \cite{Thurston hyperbolization} proves this is sufficient for the existence of a hyperbolic metric. In order to do so, we assume these surfaces exist, consider the intersection curves between these surfaces and $X$, and then find a contradiction. The following lemma allows us to do just that: 

\begin{lemma} \label{must intersect (sing)}
Let $(S_1 \times I, K_1, C_1), (S_2 \times I, K_2, C_2)$ be two tg-hyperbolic triples with singular corks. Any essential sphere, annulus, or torus $F$ that exists in the composition $M = (S_1 \times I, K_1, C_1) \#_s (S_2 \times I, K_2, C_2)$ must intersect the gluing surface $X$.
\end{lemma}

\begin{proof}
 For contradiction, suppose there exists an essential surface $F$ in the composition that does not intersect $X$. Without loss of generality, assume that $F$ lives to the $M_1$ side of $X$. Obviously, $F$ also lives in $M_1$, a tg-hyperbolic manifold. Therefore $F$ must be compressible or boundary parallel in $M_1$. 
 
 If $F$ is compressible, then it has a compressing disk $D$ in $M_1$, which can intersect $X$. But by the incompressibility of $X$ in $M_1$, we can replace it with a compressing disk that avoids $X$, and therefore the disk appears in $M$, contradicting the essentiality of $F$ in $M$.
 
 If $F$ is boundary parallel, then it must be boundary parallel to a component of $K_1$. Furthermore, that component does not intersect the cutting surface $X$. Therefore $F$ must still be boundary parallel to that same component of $K_1$ in $M$, and therefore cannot be essential.  
 
 Therefore no essential sphere, annulus, or torus exists in $M$ without intersecting $X$. 
\end{proof}

We now proceed by eliminating all possible intersection curves in $F \cap X$. 

\begin{lemma}\label{lemma: sing hyp, ess surface cannot intersect}
Let $(S_1 \times I, K_1, C_1), (S_2 \times I, K_2, C_2)$ be two tg-hyperbolic triples with singular corks. Any essential sphere, annulus, or torus $F$ that exists in the composition $M = (S_1 \times I, K_1, C_1) \#_s (S_2 \times I, K_2, C_2)$ has an embedding in $M$ that does not intersect the gluing surface $X$.
\end{lemma}

\begin{proof}
 First, we show that we can assume $F$ does not have boundary on $K$. If $F$ has boundary on $K$, it must be an annulus. The other boundary of $F$ can therefore be on the inner surface, the outer surface, or also on $K$. So $F$ can have boundary on one or two components of $K$; call these components $J_1$ and $J_2$. 
 Note that $J_2$ might be the same component as $J_1$, or $J_2$ might not exist. Then we can take the boundary of a regular neighborhood of $F \cup J_1 \cup J_2$ to obtain a torus or annulus $F'$. (In the case $J_1 = J_2$, two tori result. Take the outer one.) If $F'$ is also essential, then since $F'$ does not have boundary on $K$, we can proceed by only considering possibilities for essential surfaces that do not have boundary on $K$. 
 So assume that $F'$ is not essential. Clearly, if $F'$ is an annulus, any compressing disk for $F'$ would imply a compressing disk for $F$. If $F'$ is a torus, compressing it generates a sphere, which has components of $K$ to one side and boundaries of $S \times I$ to the other, so it is essential. 
 
  So $F'$ must be boundary parallel. This implies both boundaries of $F$ are on $K$ and $J_1 = J_2$ and that $F$ is boundary compressible. Therefore we proceed by only considering surfaces that do not have boundary on $K$. 
      \newline

  In the rest of the proof, we will often refer to \emph{trivial} circles and \emph{parallel} circles. A circle is trivial on a surface if it bounds a disk on that surface, and two circles embedded in a surface are parallel if they cobound an annulus on the surface. 

    Furthermore, in each of the following cases we can first assume that we have chosen $F$ such that $F$ minimizes the number of intersection curves in $F \cap X$. Note that this minimality is only over surfaces that do not have boundary on $K$. We can now proceed with eliminating intersection curves.
 \newline

 



\noindent \textbf{Trivial circles:} Suppose there is an intersection curve  that is a trivial circle on $F$. Choose an innermost such called $\alpha$. Since it is a trivial circle on $F$, it must bound a disk $D$ with $D \subset F$. Since $X$ is incompressible, $\alpha$ must also bound a disk $E$ on $X$. Now we can consider the sphere $F'$ created by gluing these two disks together along $\alpha$. If $F'$ bounds a ball, then we can push the entire sphere to either side of $X$. This means that we could have pushed the disk $D$ on $F$ to the other side of $X$, thereby eliminating the intersection curve $\alpha$. This is a contradiction to the minimality of intersection curves. If the sphere does not bound a ball, then it is an essential sphere that can be pushed off $X$ and therefore also lives in $M_1$, which is a contradiction to the tg-hyperbolicity of $M_1$. Thus, we can eliminate the case in which $\alpha$ is a trivial circle on $F$. 

Note that by eliminating intersection circles that are trivial on $F$, we also completely eliminate spheres because any circle on a sphere is trivial. 

Any intersection curve that is nontrivial on $F$ but trivial on $X$ would bound a compressing disk for $F$, implying $F$ is not essential. So we have eliminated all trivial circle intersections on $F$ and on $X$.
\newline

\noindent \textbf{Identical intersection curves in $X_i$:} Suppose $X_1$ and $X_2$ have identical sets of intersection curves (up to isotopy). By identical, we mean that we could glue $X_1$ to $X_2$ with an orientation preserving homeomorphism $\phi$ such that $\phi$ identifies intersection curves $\alpha_i$ on $X_1$ with intersection curves $\beta_i$ on $X_2$. This means that we can undo the composition and glue each $M_i$ back up while preserving the intersection curves. This process yields two separate surfaces, $F_1$ living in $M_1$ and $F_2$ living in $M_2$. Neither $F_i$ can be essential because they both live in tg-hyperbolic manifolds. Also note that each $F_i$ must be either an annulus or a torus. This is because for either to be a sphere,  $F$ must have had a disk to one side of $X$, which is impossible since we eliminated trivial intersection circles on $F$. 

If either is compressible, then it must have a compressing disk $D$. Since $X$ is incompressible, we can isotope $D$ such that $D \cap X$ is empty. Therefore $D$ still exists in $M$ for $F$, which implies that $F$ was compressible and therefore not essential.

So each $F_i$ must be boundary parallel in $M_i$. Also, they both must be boundary parallel to the same thing: either the inner surface, the outer surface, or a component of $K_i$ since both $F_1$ and $F_2$ have the same set of intersection curves with $X$. 

If both $F_i$s are boundary parallel to the inner surface of $M_i$, then $F$ must also be boundary parallel to the inner surface of $M$. Similarly, if both $F_i$s are boundary parallel to the outer surface of $M_i$, then $F$ is boundary parallel to the outer surface of $M$. 
If both $F_i$s are boundary parallel to a component of $K_i$, then $F$ must also be boundary parallel to a component of $K_1 \# K_2$. This is because all the intersection curves on $X$ must be circles boundary parallel to the object that each $F_i$ is boundary parallel to. 

\medskip

\noindent \textbf{Nontrivial intersection circles:} Next, we consider sets of intersection circles that are nontrivial on each of $F$ and $X$. Possibilities are shown in Figure \ref{fig:single parallel curve}. Note that in all three cases, the curves are parallel on $X_i$ to the inner surface, the outer surface or to the knot. 

\begin{figure}[htbp]
    \centering
    \includegraphics[width=0.5\textwidth]{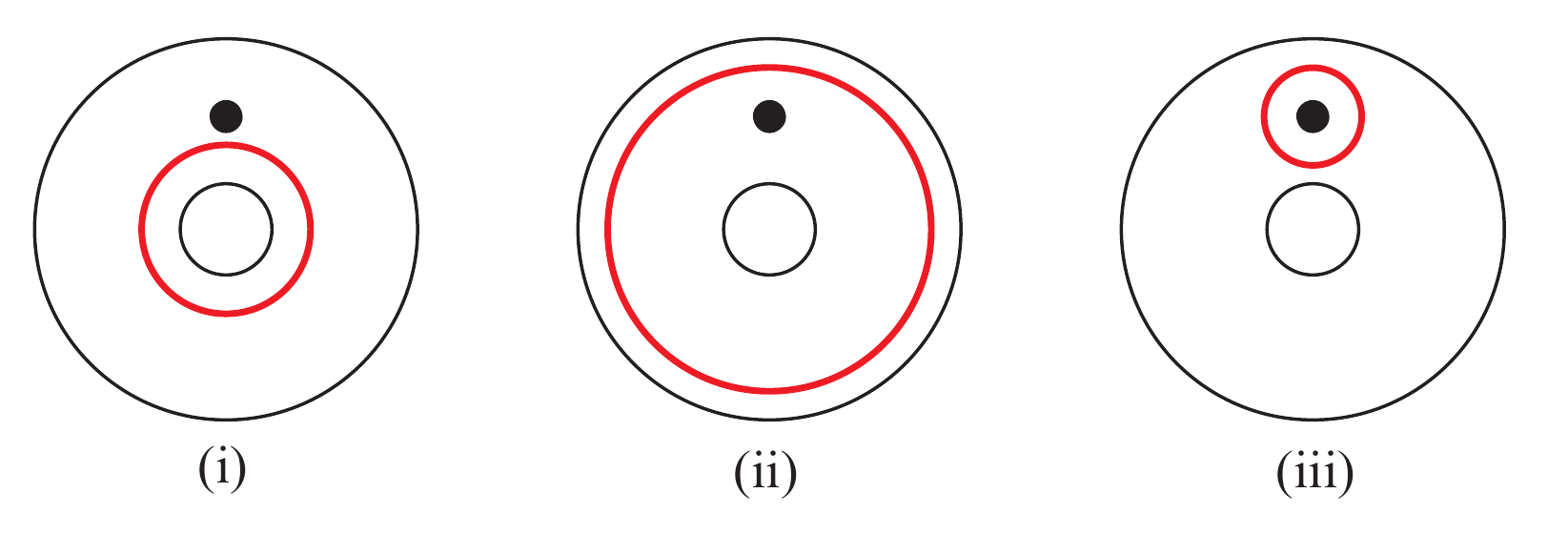}
    \caption{Intersection circle parallel to the inner surface (i), the outer surface (ii), and the puncture from $K_1 \# K_2$ (iii).}
    \label{fig:single parallel curve}
\end{figure}

Suppose $F$ has an outermost intersection circle $\alpha$ that is parallel to the outer surface, as shown in Case (ii). This means $\alpha \cup \partial_{\text{outer}} M$ bound an annulus $A$, and $A$ does not have any intersection curves. Then we can cut $F$ along $\alpha$, and glue on two copies of $A$, isotoping each slightly off of $X$. If $F$ is a torus, then this yields an annulus $F'$. If $F$ is an annulus, then this yields two annuli $F_1$ and $F_2$. We first consider $F$ a torus, with $F'$ an annulus. If $F'$ is essential, then this contradicts minimality of intersection curves of $F$. So $F'$ must be compressible or boundary parallel. If $F'$ is compressible, then it has a compressing disk $D$. We can always isotope $\partial D$ to be entirely on a part of $F'$ that does not intersect $\alpha$, and therefore $D$ must still exist for $F$. Thus $F$ was not essential. If $F'$ is boundary parallel, then it must be boundary parallel to the outer surface (since it has both boundary circles on the outer surface). But this implies that $F$ was boundary parallel to the outer surface in $M$, and thus $F$ could not have been essential in $M$. 

Now consider $F$ an annulus, and $F_1$, $F_2$ also annuli. If either $F_i$ are essential, this contradicts minimality of intersection curves of $F$. So both must be either boundary parallel or compressible. If either is compressible, then its compressing disk $D$ must still exist for $F$, and thus $F$ could not have been essential. Therefore both $F_1$ and $F_2$ must be boundary parallel. Furthermore, they both must be boundary parallel to the outer surface, since they both have curves isotopic to a nontrivial circle on the outer surface. But this implies that $F$ was also boundary parallel to the outer surface, and thus not essential. 

Therefore we can eliminate the case in which $F$ has an intersection circle $\alpha$ that is parallel to the outer surface on $X$. Replacing the word `outer' with `inner' in the above proof also proves the case for $\alpha$ parallel to the inner surface. So we just consider $\alpha$ parallel to a puncture from $K$. Again, the same set of arguments apply to eliminate this case if we just consider the cusp boundary of $K$ instead of the outer surface.

Since we have eliminated all possible intersection circles, we are just left with the possibility of intersection arcs on $X_i$. Note that  this implies $F$ has boundary on $\partial M$, and therefore $F$ must be an annulus for the rest of these arguments eliminating arcs.



\medskip 

\noindent \textbf{Intersection arcs that are trivial on $X$:} First, it is relatively easy to eliminate trivial intersection arcs on $X$. Assume $F$ has the minimum number of intersection curves, and has an intersection arc $\alpha$ that is trivial on $X$. By definition, $\alpha$ bounds a disk $D$ with one arc on $\partial (S \times I)$. Since $F$ is essential, $D$ cannot be a boundary compression disk, and therefore $F$ has some disk $E \subseteq F$ such that $E$ has boundary defined by $\alpha$ and a single arc on $\partial F$. But together with a disk on $\partial M$, these disks bound a sphere $F'$ completely to one side of $X$, say to the $M_1$ side. We can therefore consider it in  $M_1$. Then $F'$ is a sphere in a tg-hyperbolic manifold, and so must bound a ball. This implies that we could have pushed $E$ through to the other side of $X$ in $M$, thereby eliminating the intersection curve $\alpha$. This is a contradiction to the minimality of intersection curves, and therefore we can eliminate trivial arcs on $X$.

\bigskip  

\noindent \textbf{Intersection arcs that are trivial on $F$:}  Next, we eliminate any trivial intersection arc on $F$. We first show this is equivalent to proving that any intersection arc $\alpha$ must have endpoints on different boundary circles of $F$. For contradiction, assume not. Then $\alpha$ is an intersection arc with both endpoints on the same boundary circle of $F$. This also implies that both endpoints of $\alpha$ are on the same boundary of $\partial (S \times I)$. 
The only option is an arc $\alpha$ on $X$ that separates the puncture from the boundary that it does not intersect.
Then together with a single arc  $\alpha'$ on $\partial F$, $\alpha$ bounds a disk $D$ on $F$. We also know that $\alpha$ together with a single arc $\gamma$ on $\partial (S \times I)$ bounds a once-punctured disk $D'$ on $X$
. Since $\partial F$ must be on the same boundary of $S \times I$ as $\gamma$, we consider the disk $E$ on $\partial (S \times I)$ bounded by $\gamma \cup \alpha'$. Taking the union $D \cup D' \cup E$ then yields a once punctured sphere $F'$. But this is impossible, since the one puncture comes from $K$ and each component of $K$ is a connected 1-manifold. 


The above arguments imply that any arc $\alpha$ cannot be trivial on $F$, and so must have endpoints on different boundary circles of $F$. 

\medskip

\noindent \textbf{Parallel arcs:} Suppose $F$ has the minimum number of intersection curves with $X$, and there are parallel intersection arcs on $X$. 
Then  we can pick parallel arcs $\alpha$ and $\beta$ on $X_i$ such that, together with two arcs on $\partial X_i$, bound a disk $D$ on $X_i$, and $D$ has empty intersection with any other intersection curve. Then we cut $F$ along $\alpha$ and $\beta$ into two disks $D_1$ and $D_2$, glue on a copy of $D$ to each and isotope slightly so the copies of $D$ no longer intersect $X_i$. This yields two annuli $F_1$ and $F_2$ whose union has fewer intersection curves with $X$ than $F$. If either is essential, this contradicts minimality of intersection curves. So both $F_1$ and $F_2$ must be boundary parallel or compressible.

Suppose $F_i$ is compressible. Then it has a compressing disk $D'$. So compress along $D'$, and cap off the trivial boundary circles of $F_i$ to obtain two spheres. Since we have already eliminated the existence of essential spheres, both must bound a ball. But this implies that $F_i$ bounds a thickened disk, and therefore $D_i$ can be isotoped through that thickened disk to $D$ and then a little beyond, which is an isotopy of $F$ that lowers its number of intersections with $X$, a contradiction to the minimality of intersection curves. 

This means that both $F_1$ and $F_2$ must be boundary parallel. We know neither can be boundary parallel to a component of the link, since the boundaries of $F$ are on the boundary surfaces of $S \times I$. 
Also, we know that $F_1$ and $F_2$ must be boundary parallel to the same surface (either inner or outer). Now we consider the three possibilities for parallel intersection arcs, and eliminate each. 

\begin{figure}[htbp]
    \centering
    \includegraphics[width=0.5\textwidth]{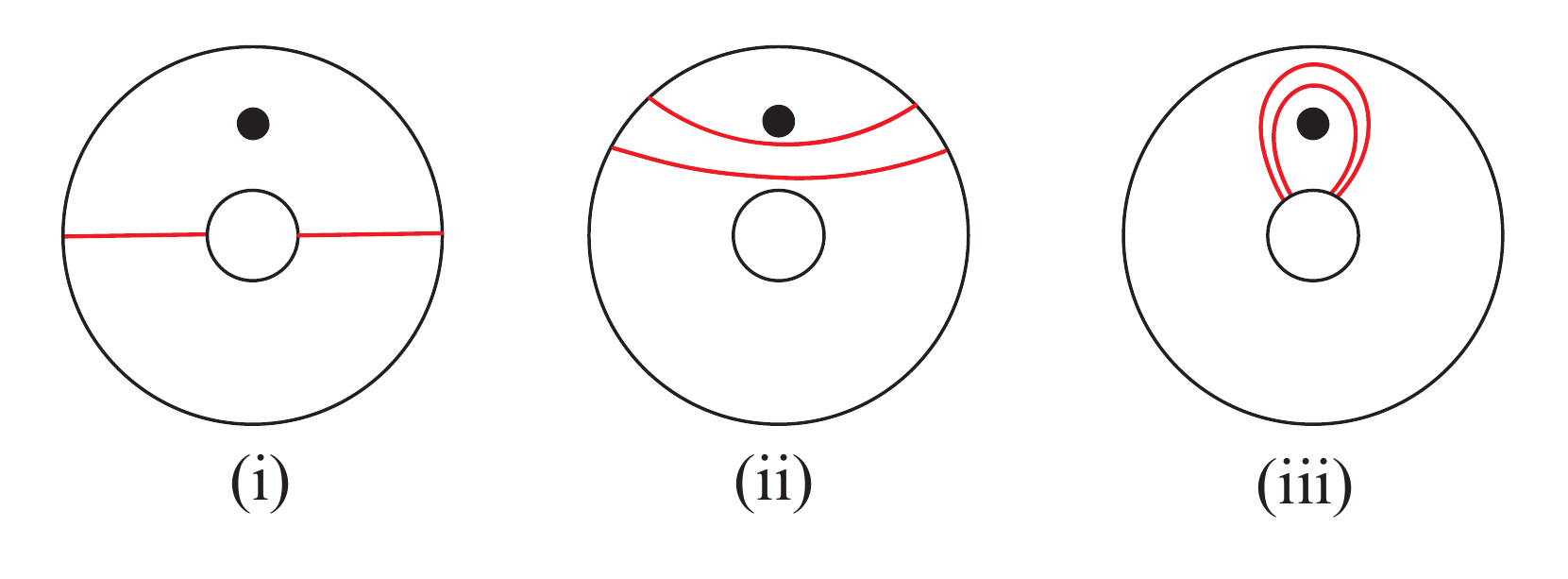}
    \caption{The three possibilities for parallel intersection arcs $\alpha$ and $\beta$.}
    \label{fig: parallel int arc cases}
\end{figure}
Case (i) of Figure \ref{fig: parallel int arc cases} is not possible, since it requires that for each $F_i$, one boundary circle is on the inner surface and one is on the outer surface, making it impossible for either $F_i$ to be boundary-parallel. 

Consider case (ii) of Figure \ref{fig: parallel int arc cases}. Since both arcs have endpoints on the outer surface, we know that $F_1$ and $F_2$ must have boundary circles on the outer surface and therefore can only be boundary parallel to the outer surface. We also know that $F_1$ includes the arc $\beta$, and since $F_1$ is boundary parallel to the outer surface, $\beta$ must be boundary parallel as well. This means that together with an arc on the outer boundary, $\beta$ bounds a disk. But this directly contradicts the boundary incompressibility of $X$. Therefore $F_1$ and $F_2$ cannot be boundary parallel to the outer surface (and so must be essential, a contradiction to minimality of intersection curves). 
Case (iii) follows from the proof of case (ii) above and the fact that arguments about the inner and outer surfaces are symmetrical. 

This means that no parallel intersection arcs can exist on $X$. At this point, we have eliminated all intersection circles, trivial intersection arcs, and parallel intersection arcs. This leaves us with only three possible sets of intersection curves, shown in Figure \ref{fig: last sing int arc cases}. 

\begin{figure}[htbp]
    \centering
    \includegraphics[width=0.5\textwidth]{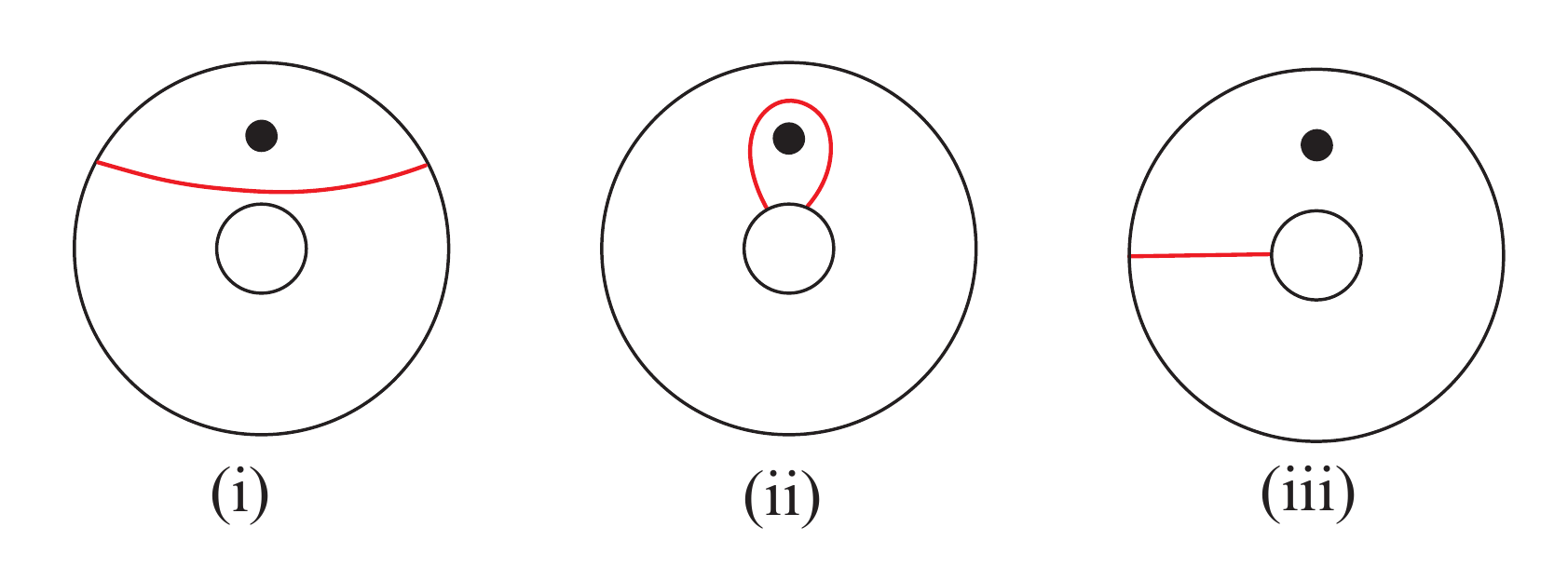}
    \caption{The last remaining possibilities for intersection arcs on $X_1$ and $X_2$.}
    \label{fig: last sing int arc cases}
\end{figure}

\bigskip

\noindent \textbf{Nontrivial, nonparallel arcs:} We proved above that each $X_i$ must have a different set of intersection curves. Therefore there are three possible pairs of intersection curves left when considering both $X_1$ and $X_2$. While previous arguments looked at just a set of intersection curves on a single $X_i$, these last six cases require us to think about the set of intersection curves on $X = X_1 \cup X_2$. 

Consider the boundary circles of $F$. We know from previous arguments that each intersection arc must have endpoints on different boundary circles of $F$. Taking both $X_1$ and $X_2$ into account, we also know that each boundary circle must cross $X$ an even number of times (but this could happen from crossing $X_1$ and odd number of times and $X_2$ an odd number of times). However, this implies that the cases in which one of $X_1$ and $X_2$ has a curve that looks like Case (iii) in Figure \ref{fig: last sing int arc cases} and the other has a different intersection curve are impossible.  Also, when $X_1$ has an arc with both boundary circles on the same boundary surface, then we have a contradiction unless $X_2$ also has an arc with both boundary circles on that same surface. But this is also impossible, since each $X_i$ must have a different set of intersection curves. 

The above arguments show that $X_1$ and $X_2$ cannot have different sets of intersection curves, contradicting the fact we proved earlier that each $X_i$ must have a different set of intersection curves. Therefore, $F \cap X$ is empty for any essential surface $F$ embedded in a way that minimizes the number of intersection curves in $F \cap X$. 
\end{proof}

We can now prove the main result of this subsection. 

\begin{theorem}\label{thm: singular tg-hyperbolicity}
Given two non-classical virtual knots  that are tg-hyperbolic, with triples\\
$(S_1 \times I, K_1, C_1), (S_2 \times I, K_2, C_2)$ such that both $C_1$ and $C_2$ are singular, the singular composition \newline $(S_1 \times I, K_1, C_1) \#_s (S_2 \times I, K_2, C_2)$ is tg-hyperbolic. Further, $g(K_1 \# K_2) = g(K_1) + g(K_2) -1$.
\end{theorem}

\begin{proof}
We know from Lemma \ref{lemma: disk generates sphere} that an essential disk can only exist if there is an essential sphere. By Lemma \ref{must intersect (sing)}, any essential sphere, torus, or annulus in $M$ must be embedded in such a way that it intersects $X$. However, we eliminated all possible sets of intersection curves for an essential sphere, torus, or annulus in the singular composition of two tg-hyperbolic triples using Lemma \ref{lemma: sing hyp, ess surface cannot intersect}. This implies that no essential disk, sphere, annulus, or torus exists in the manifold $M$. By the work of Thurston \cite{Thurston hyperbolization}, this implies that $M$ is a tg-hyperbolic manifold. 

 Since $K_1$ and $K_2$ are tg-hyperbolic, $S_1$ and $S_2$ must be minimal genus for those knots. Since $(S_1 \times I) \#_s (S_2 \times I$) has genus $g(K_1) + g(K_2) -1$ and  $M$ is tg-hyperbolic and therefore minimal genus, we have that $g(K_1 \# K_2) = g(K_1) + g(K_2) -1$. 
\end{proof}

\subsection{TG-Hyperbolicity for Hyperbolically Composable Composition}\label{subsect: comp is tg-hyperbolic}

We prove a similar sequence of lemmas as in the previous subsection in order to prove the following.

\begin{theorem} \label{thm: hyperbolicallycomposable}The nonsingular composition of two hyperbolically composable  triples $(S_1 \times I, K_1, C_1)$ and $(S_2 \times I, K_2, C_2)$ is tg-hyperbolic, with genus $g(K_1 \# K_2) = g(K_1) + g(K_2)$.
\end{theorem}

The fact the triples are both hyperbolically composable implies there exist triples $(S_1' \times I, K_1', C_1')$ and $(S_2' \times I, K_2', C_2')$ such that 
$(S_1 \times I, K_1, C_1) \#_{ns} (S_1' \times I, K_1', C_1')$ and $(S_2 \times I, K_2, C_2) \#_{ns} (S_2' \times I, K_2', C_2')$ are tg-hyperbolic. Let $S = S_1 \# S_2$ and let $K = K_1 \# K_2$. Let $M_i = S_i \times I \setminus K_i$.




Lemma \ref{lemma: disk generates sphere} implies that if we eliminate all essential spheres in $(S_1 \times I, K_1, C_1) \#_{ns} (S_2 \times I, K_2, C_2)$, then we also eliminate all essential disks. 

\begin{lemma}\label{lemma: doublingsurfaceinc}
If $(S_1 \times I, K_1, C_1)$ and $(S_2 \times I, K_2, C_2)$ are tg-hyperbolic, the surface $X = \partial C_1 \cap (S \times I) = \partial C_2 \cap (S \times I)$ is incompressible and boundary incompressible in $M = (S_1 \times I, K_1, C_1) \#_{ns} (S_2 \times I, K_2, C_2)$.
\end{lemma}

We include Figure \ref{fig: removing C yields X} to help to visualize the surface $X$. 

\begin{figure}[htbp]
    \centering
    \includegraphics[width=0.5\textwidth]{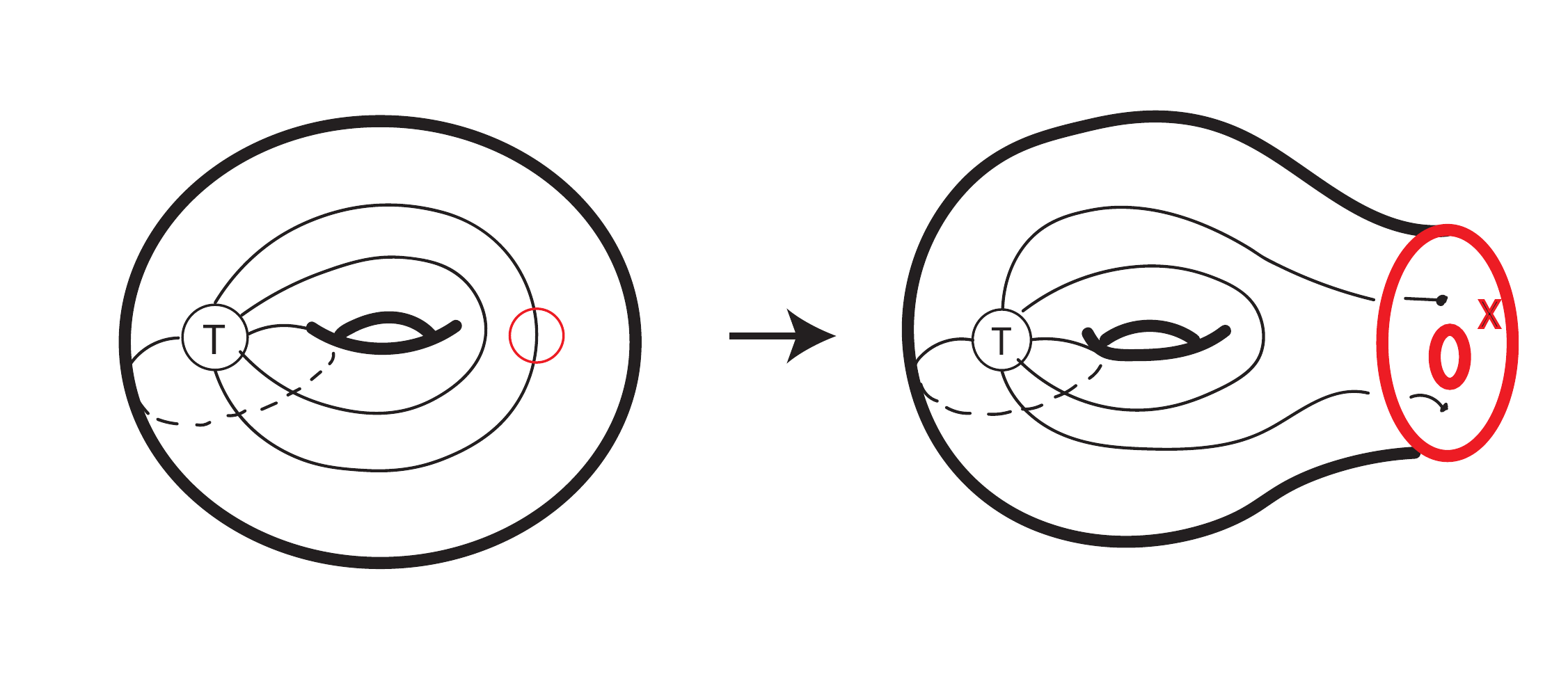}
    \caption{When $C_1$ is removed from $M_1$, the twice punctured annulus is called $X$.}
    \label{fig: removing C yields X}
\end{figure}

\begin{proof}
Suppose $X$ is compressible. Then there exists a compressing disk $D$ with $X \cap \partial D = \beta,$ where $\beta$ is a circle. Since $D$ is a compressing disk, we know that $D$ is completely to one side of $X$. For convenience, we assume $D$ is to the $M_1$ side. We consider the region that $\beta$ encloses on $X$. We refer to the two punctures from $K$ as punctures, and the hole from the inner surface of $M_1$ as a hole. Since $D$ is a disk, it is obvious that $\beta$ cannot enclose the hole because the inner surface has positive genus. So $\beta$ encloses 0, 1, or 2 of the punctures from $K$ and does not enclose the hole.
It is impossible for $\beta$ to enclose only a single puncture from $K$; if a strand of the knot enters the region inside of $D$ it must also leave that region.
If $\beta$ encloses no punctures, then it is trivial on $X$, a contradiction to $D$ being a compressing disk for $X$.

If $\beta$ encloses two punctures, then we can fill the cork $C_1$ back in to obtain $M_1$. Then there is a sphere in $M_1$ obtained by taking $D \cup D'$ where $D'$ is a disk in $C_1$ with boundary $\beta$. This sphere must contain the component of $K_1$ corresponding to the two punctures, but then it is an essential sphere in $M_1$, a contradiction to its tg-hyperboicity. Thus $X$ is incompressible. 


 To prove boundary incompressibility, suppose first  that some arc $\alpha$ properly embedded in  $X$ that cannot be isotoped through $X$ to the boundary $\partial (S \times I)$ bounds a disk $D$ in $M$ with a single arc $\beta$ on $\partial (S \times I)$. We assume $D$ is to the $M_1$ side.
 
 Since $\alpha \cup \beta$ bounds a disk, we know that $\alpha \cup \beta$ is trivial with respect to $M_1 \pmc C_1$. Therefore we can isotope $\beta$ to be entirely on $\partial X$, and then isotoping $\beta$ slightly off of the boundary and onto $X$ we get a compressing disk $E$ for $X$, which has been precluded by the arguments above. 


 Suppose now that there is an arc $\alpha$ in $X$ that begins and ends at the punctures on $X$ and that together with a nontrivial arc in $\partial N(K_1)$ bounds a disk $D$ in $M_1$. If $\alpha$ begins and ends on the same puncture, then it must bound a disk on $X$, contradicting it being a boundary compression. If it begins and ends on different punctures, then the boundary of a neighborhood of $D \cup K_1$ contains a disk that yields a compression of $X$, a contradiction. Therefore $X$ is boundary incompressible. 
\end{proof}


Next, we eliminate all possible sets of intersection curves between an essential surface $F$ and $X$.
\begin{lemma} \label{no intersections}
 If $(S_1 \times I, K_1, C_1)$ and $(S_2 \times I, K_2, C_2)$ are tg-hyperbolic, any essential sphere, annulus, or torus $F$ that exists in $M = (S_1 \times I, K_1, C_1) \#_{ns} (S_2 \times I, K_2, C_2)$ and minimizes the number of intersection curves in $F \cap X$ cannot intersect $X$. 
\end{lemma}
\begin{proof}
 Exactly as in the proof of Lemma \ref{lemma: sing hyp, ess surface cannot intersect} we can assume $F$ does not have boundary on $K$. Furthermore, in each of the following cases we can first assume that $F$ has been put in a position that minimizes the number of intersection curves in $F \cap X$ and $F$ has the minimal number of intersection curves among essential spheres, annuli or tori, depending on which $F$ is.

\bigskip

\noindent \textbf{Trivial circles:} Intersection circles that are trivial on either $F$ or $X$ are eliminated exactly as in the corresponding argument in the proof of Lemma \ref{lemma: sing hyp, ess surface cannot intersect}.
\bigskip  

\noindent \textbf{Intersection circles that are nontrivial on both $F$ and $X$:} 
First, we eliminate intersection curves that are parallel to the inner surface, the outer surface, or punctures from $K$. Examples of these three are shown in Cases (i)-(iii) of Figure \ref{fig: types of nontrivial int curves ns}.

\begin{figure}[htbp]
    \centering
    \includegraphics[width=0.5\textwidth]{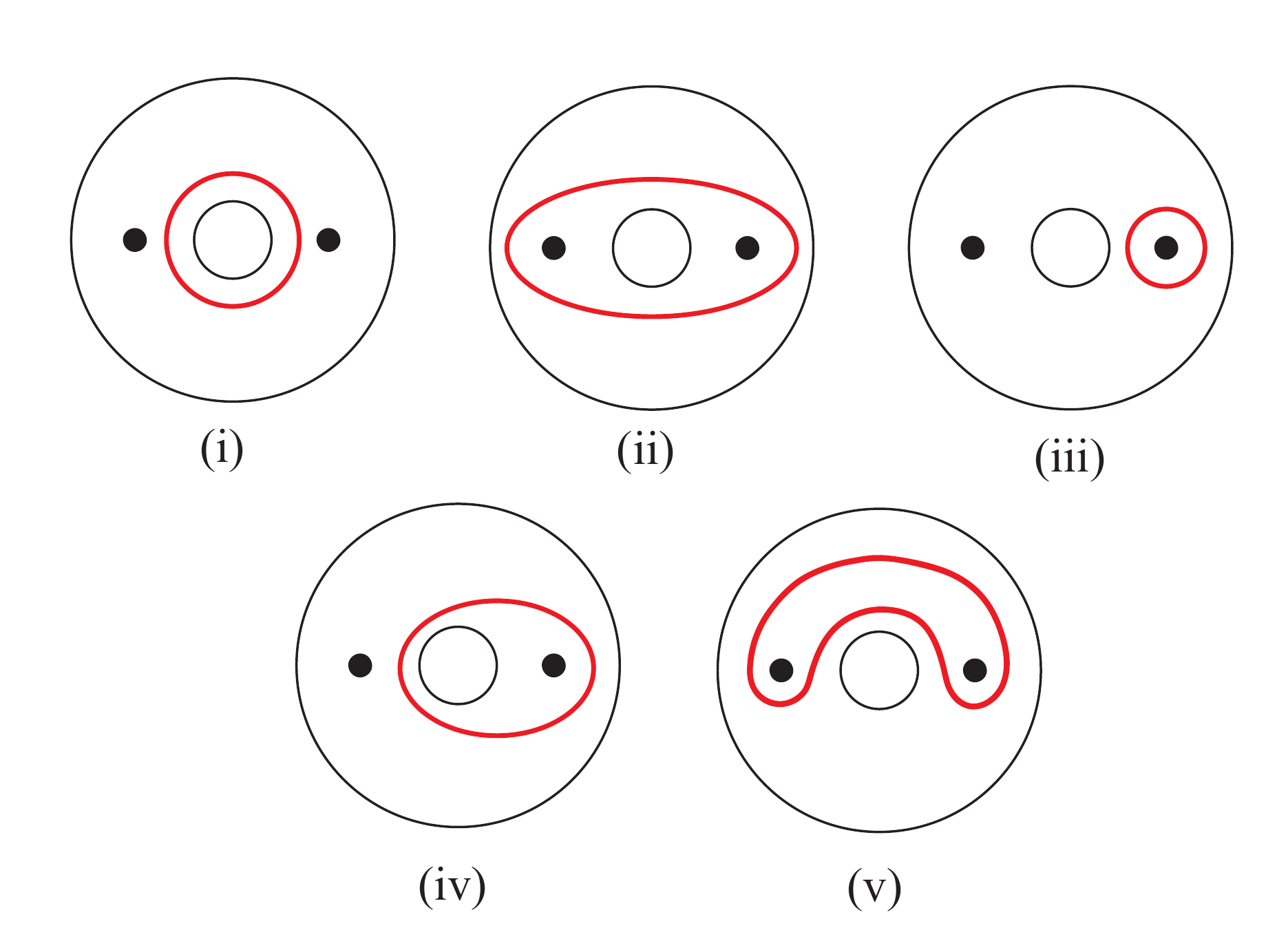}
    \caption{Types of nontrivial intersection curves. Cases (i) and (ii) are parallel to the inner and outer surfaces, respectively. Case (iii) is parallel to a puncture from $K$.}
    \label{fig: types of nontrivial int curves ns}
\end{figure}

Suppose $F$ has an outermost intersection circle $\alpha$ that is parallel to the outer surface, as shown in Case (ii). This means $\alpha \cup \partial_{\text{outer}} M$ bound an annulus $A$, and $A$ does not have any intersection curves. Then we can cut $F$ along $\alpha$, and glue on two copies of $A$, isotoping each slightly off of $X$. If $F$ is a torus, then this yields an annulus $F'$. If $F$ is an annulus, then this yields two annuli $F_1$ and $F_2$. We first consider $F$ a torus, with $F'$ an annulus. If $F'$ is essential, then this contradicts minimality of intersection curves of $F$. So $F'$ must be compressible or boundary parallel. If $F'$ is compressible, then it has a compressing disk $D$. We can always isotope $\partial D$ to be entirely on a part of $F'$ that does not intersect $\alpha$, and therefore $D$ must still exist for $F$. Thus $F$ was not essential. If $F'$ is boundary parallel, then it must be boundary parallel to the outer surface (since it has both boundary circles on the outer surface). But this implies that $F$ was boundary parallel to the outer surface in $M$, and thus $F$ could not have been essential in $M$. 

Now consider $F$ an annulus, and $F_1$, $F_2$ also annuli. If either $F_i$ are essential, this contradicts minimality of intersection curves of $F$. So both must be either boundary parallel or compressible. If either is compressible, then its compressing disk $D$ must still exist for $F$, and thus $F$ could not have been essential. Therefore both $F_1$ and $F_2$ must be boundary parallel. Furthermore, they both must be boundary parallel to the outer surface, since they both have curves isotopic to a nontrivial circle on the outer surface. But this implies that $F$ was also boundary parallel to the outer surface, and thus not essential. 

Therefore we can eliminate the case in which $F$ has an intersection circle $\alpha$ that is parallel to the outer surface on $X$. Replacing the word `outer' with `inner' in the above proof also proves the case for $\alpha$ parallel to the inner surface. So we just consider $\alpha$ parallel to a puncture from $K$. Again, the same set of arguments apply to eliminate this case if we just consider the cusp boundary of $K$ instead of the outer surface. 

Therefore the only possible intersection circles left are nontrivial circles that enclose either two punctures or one puncture and the hole from the inner surface, as in Figure \ref{fig: types of nontrivial int curves ns} (iv) and (v).  Note that for each type, there could be multiple other intersection circles, as long as all are isotopic to each other on $X$. 


\bigskip  

\noindent \textbf{Parallel intersection circles:} We next eliminate parallel circles of Type (iv) or (v). Suppose $F$ had parallel intersection circles. We know that together, any two parallel intersection circles bound an annulus on $F$. Since they are parallel, they also bound an annulus on $X$. Since we have eliminated all other possibilities, we know that every intersection circle must be a circle parallel to any other intersection circle on both $F$ and $X$. Because all annuli corresponding to $F \cap M_1 \pmc C_1$ separate $M_1 \pmc C_1$ we can pick some `innermost' $\alpha_1$ and $\alpha_2$ such that they bound an annulus $F_1$ on $F$ and an annulus $F_2$ on $X$ with each $F_1 \cap \alpha_i = \emptyset$ and $F_2 \cap \alpha_i = \emptyset$ for all other intersection curves $\alpha_i, i \neq 1, 2$. We can therefore consider the torus $F' = F_1 \cup F_2$ obtained by gluing along $\alpha_1$ and $\alpha_2$. We know that $F'$ cannot be essential because we can isotope it slightly to remove all intersections with $X$, a contradiction since we assumed $F$ had a minimal number of intersection curves among essential tori and annuli. Therefore $F'$ must be either boundary parallel or compressible in $M$. 

Suppose $F'$ is boundary parallel. In both case (iv) and (v), we know that neither intersection circle is parallel with the inner or outer surface, which implies that $F'$ cannot be boundary parallel with the inner or outer surface. So $F'$ must be boundary parallel to a component $K'$ of $K$. Since $F_2 \subset X$, $K'$ cannot be knotted, so $F'$ is an unknotted torus, and $K'$ is isotopic to $\alpha_1$ through $M_1 \pmc C_1$. If $\alpha_1$ is isotopic to case (iv) of Figure \ref{fig: types of nontrivial int curves ns}, then when we fill $C_1$ back in to obtain $M_1$, this component $K'$ bounds a disk $D$ punctured once by another component of $K$, which is the one that passes through $C_1$. The boundary of a regular neighborhood of $D \cup K'$  is an essential annulus in $M_1$,  a contradiction to the tg-hyperbolicity of $M_1$. 

If $\alpha_1$ isotopic to case (v) of Figure \ref{fig: types of nontrivial int curves ns}, then we again consider $M_1$ containing $F'$, and close it back up with $C_1$. Then we see that $\alpha_1$ bounds a disk in $C_1$ that is a compressing disk $D$ for $F'$, implying that we could compress along $D$ to get a sphere that does not bound a ball in $M_1$. Since $M_1$ is tg-hyperbolic, this is impossible. Therefore $F'$ cannot be boundary parallel.

So $F'$ must be compressible in $M$. Call its compressing disk $E$. Compression yields a sphere $F''$. Since we eliminated essential spheres, $F''$ must bound a ball, in which case $F'$ must bound a solid torus. But this implies that we could have pushed the annulus $F_1 \subseteq F$ through $X$, eliminating the intersection curves $\alpha_1$ and $\alpha_2$. This is a contradiction to the minimality of intersection curves of $F$. 
Therefore we have reached a contradiction for every possible set of parallel intersection circles, and thus can eliminate parallel circles from consideration.

\medskip
\noindent {\bf One intersection curve of type (iv) or (v):} Since we have eliminated all possibilities for parallel circles, the last two cases for intersection circles  occur for just a single intersection circle that appears as in Figure \ref{fig: types of nontrivial int curves ns} (iv) or (v). Also, since there is just a single intersection circle, we know that $F$ must be an annulus. 

Since we are only concerned with annuli right now, we note that the boundary circles of the essential annulus $F$ must be nontrivial on the outer or inner surface of $S \times I$. Call these boundary circles $\gamma_1$ and $\gamma_2$. 
The fact that the boundary circles of $F$ must be nontrivial on $\partial (S \times I)$ immediately eliminates both cases. This is because all nontrivial circles on an annulus must be isotopic, and thus $\alpha$ must be isotopic to a nontrivial circle on $\partial (S \times I)$ to each side of $X$. Ignoring $K$, it is obvious in case (v) that $\alpha$ is trivial with respect to $\partial (S \times I)$ and so cannot exist. In case (iv), ignoring $K$, $\alpha$ is isotopic to each of the two boundaries of $X$, one on the inner surface and one on the outer surface. So the boundaries of $F$ must be isotopic to one, the other or both of these curves on the boundary of $S \times I$. But then $F$ must be a separating annulus, and therefore there cannot be an odd number of punctures to either side of $\alpha$ in $X$. This is a contradiction, and so we eliminate this case. 

We have now eliminated all possible intersection circles in $F \cap X$. 


\bigskip  

\noindent \textbf{Intersection arcs:}
The endpoints of the arcs must be on either the inner or outer surface (and cannot be on the punctures from $K$) as noted at the beginning of the proof.  

\begin{figure}[htbp]
    \centering
    \includegraphics[width=0.5\textwidth]{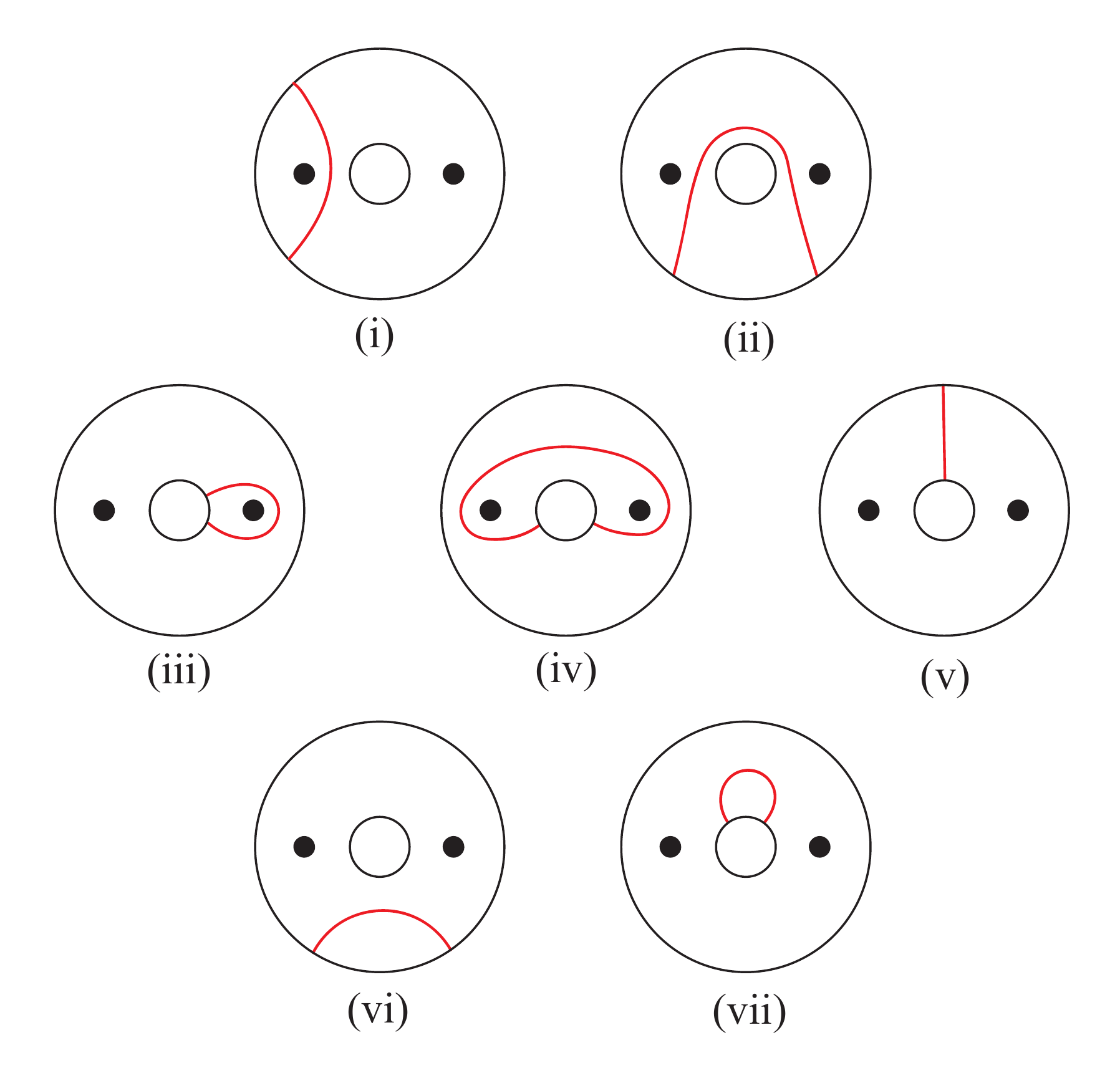}
    \caption{Types of intersection arcs. Cases (vi) and (vii) are trivial on $X$. Cases (i)-(v) are all nontrivial on $X$.}
    \label{fig: types of int arcs nonsingular}
\end{figure}

We also note that since tori and spheres cannot have boundary on the boundary of the manifold $S \times I$, the presence of an intersection arc implies that $F$ must be an annulus. All possible types of intersection arcs up to isotopy are shown in Figure \ref{fig: types of int arcs nonsingular}.

\bigskip  \noindent

\noindent \textbf{Trivial arcs on $X$:} These are eliminated exactly as in the corresponding part of the proof of Lemma \ref{lemma: sing hyp, ess surface cannot intersect}. 
\bigskip  

\noindent \textbf{Trivial arcs on $F$:}  Next, we eliminate any trivial intersection arc on $F$. Note that this is equivalent to proving that any intersection arc $\alpha$ must have endpoints on different boundary circles of $F$. For contradiction, assume not. Then $\alpha$ is an intersection arc with both endpoints on the same boundary circle of $F$. This also implies that both endpoints of $\alpha$ are on the same boundary component of $\partial M$.  Therefore $\alpha$ must look like (i), (ii), (iii), or (iv). First, we consider $\alpha$ to be in either case (i) or (iii). Since we assumed that $\alpha$ has endpoints on the same boundary of $F$, then together with a single arc  $\alpha'$ on $\partial F$, $\alpha$ bounds a disk $D$ on $F$. We also know that $\alpha$ bounds a once punctured disk $D'$ with a single arc $\gamma$ on $\partial M$. Since $\partial F$ must be on the same boundary of $M$ as $\gamma$, we consider the disk $E$ on $\partial M$ bounded by $\gamma \cup \alpha'$. Taking the union $D \cup D' \cup E$ then yields a once punctured sphere $F'$. But this is obviously impossible, since the one puncture comes from $K$ and each component of $K$ is a connected 1-manifold. 

Now suppose  $\alpha$ appears as in Cases (ii) or (iv), and has endpoints on the same boundary of $F$. Without loss of generality, we will consider Case (ii) with respect to the outer boundary; the argument for Case (iv) follows by replacing 'inner boundary' with 'outer boundary.' We know that together with a single arc $\alpha'$ on $\partial F \cap \partial_{\text{outer}}M$, $\alpha$ bounds a disk $D$ on $F$. Now we close up one copy of $M$ by adding back the cork $C$. We can then glue  another disk $D'$ to $D$, with $D' \subset C$ and $D'$ bounded by $\alpha$ and a single arc on $\partial_{\text{outer}}M$. Then $D \cup D' = E$ is a disk with boundary on the outer boundary of one copy of $M$. So it must bound a disk $D''$ on $\partial M$. So $E \cup D''$ is a sphere $F'$, which by construction must have a ball containing a component of $K$ to one side, and thus does not bound a ball. But $M$ is a tg-hyperbolic manifold, which yields a contradiction. Therefore our assumption must have been false, and $\alpha$ cannot be isotopic to the arcs shown in Cases (ii) or (iv) of Figure \ref{fig: types of int arcs nonsingular} if it has endpoints on the same boundary of $F$. 

The above arguments imply that any arc $\alpha$ cannot be trivial on $F$, and so must have endpoints on different boundary circles of $F$. 

\bigskip

\noindent \textbf{Parallel arcs on $X$:} Next, we consider parallel arcs. Suppose $F$ has the minimum number of intersection curves with $X$, and also has two parallel intersection arcs $\alpha$ and $\beta$. From the above arguments, we know that each of $\alpha$ and $\beta$ must have endpoints on different boundary circles of $F$. This implies that $\alpha$ and $\beta$, together with two disjoint arcs on $\partial F$, bound a disk $D$ on $F$. And by definition of parallel arcs, together with two arcs on $\partial X$ they bound a disk $E$ on $X$. Now we cut $F$ at $\alpha$ and $\beta$, and glue on two copies of $E$. This results in two new annuli $F_1$ and $F_2$, with  $F_1 = E \cup D$. Furthermore, we can isotope each $F_i$ slightly off of $X$ so that their union has fewer intersection curves than $F$. If either $F_1$ or $F_2$ is essential, then this contradicts the minimality of number of intersection curves. Therefore each  must be either compressible or boundary parallel in $M$. Suppose $F_i$ is compressible. Then since all nontrivial curves on an annulus are isotopic, $F_i$ must bound a thickened disk, which implies we could isotope $F_i$ through $X$ and lower the number of intersection arcs for $F$, a contradiction. 

Thus both $F_1$ and $F_2$ must be boundary parallel. Again, we know that neither can be boundary parallel to a component of the link, since $F$ must have had boundary circles that were nontrivial with respect to the inner and outer surfaces and all nontrivial circles on an annulus are isotopic. Also, we know that $F_1$ and $F_2$ must be boundary parallel to the same surface (either inner or outer). We eliminate each of the five remaining cases, shown in Figure \ref{fig: nonsingular parallel arcs}.

\begin{figure}[htbp]
    \centering
    \includegraphics[width=0.5\textwidth]{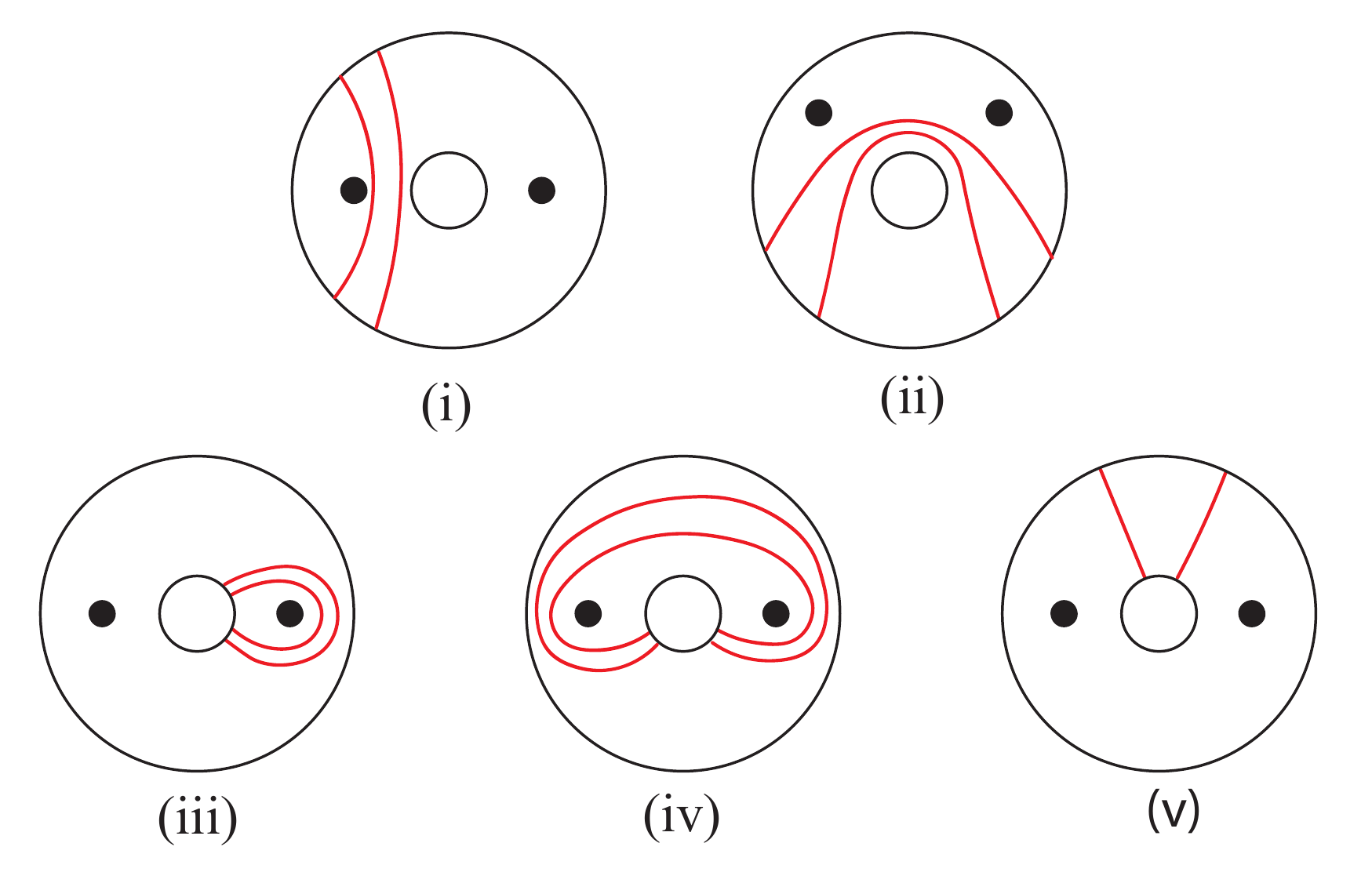}
    \caption{The five possibilities for parallel arcs in nonsingular composition.}
    \label{fig: nonsingular parallel arcs}
\end{figure}

First, suppose $F$ was an essential annulus with intersection curves isotopic to cases (i), (ii) , (iii) or (iv) of Figure \ref{fig: nonsingular parallel arcs}. Since both arcs have endpoints on the same boundary surface, we know that $F_1$ and $F_2$ must have boundary circles on that surface. Furthermore, we know that the disk $D$ bounded by $\alpha$ and $\beta$ on $X$ is a subsurface of each $F_i$, and thus $D$ must be boundary parallel to that surface as well. So both $\alpha$ and $\beta$ can be isotoped relative their boundaries into that surface. But this directly contradicts the boundary incompressibility of $X$. So case (i)-(iv) are not possible. 




In case (v), the two boundary circles of $F_i$ are on distinct boundary surfaces, so it cannot be the case $F_1$ or $F_2$ is boundary-parallel. 

\bigskip  \noindent

\textbf{Nontrivial, nonparallel intersection arcs:} Now that we have eliminated trivial arcs and parallel arcs, we consider all other possibilities for intersection arcs on $X$. In determining possibilities, recall that any intersection arc must have endpoints on different boundary components of the annulus $F$. Since any nontrivial circle on $\partial (S \times I)$ must cross $X$ an even number of times, there are an either an even number of endpoints of intersection arcs or no endpoints on each of the inner and outer boundaries of $X$. Furthermore, there must be at least two intersection arcs. Also, each pair of intersection arcs must have endpoints on the same boundary surface of $M$, so we cannot have one intersection arc with both endpoints on the inner surface and another with both endpoints on the outer surface. Therefore there are only five cases left to consider, shown in Figure \ref{fig: last nonsingular arcs cases}.

\begin{figure}[htbp]
    \centering
    \includegraphics[width=0.5\textwidth]{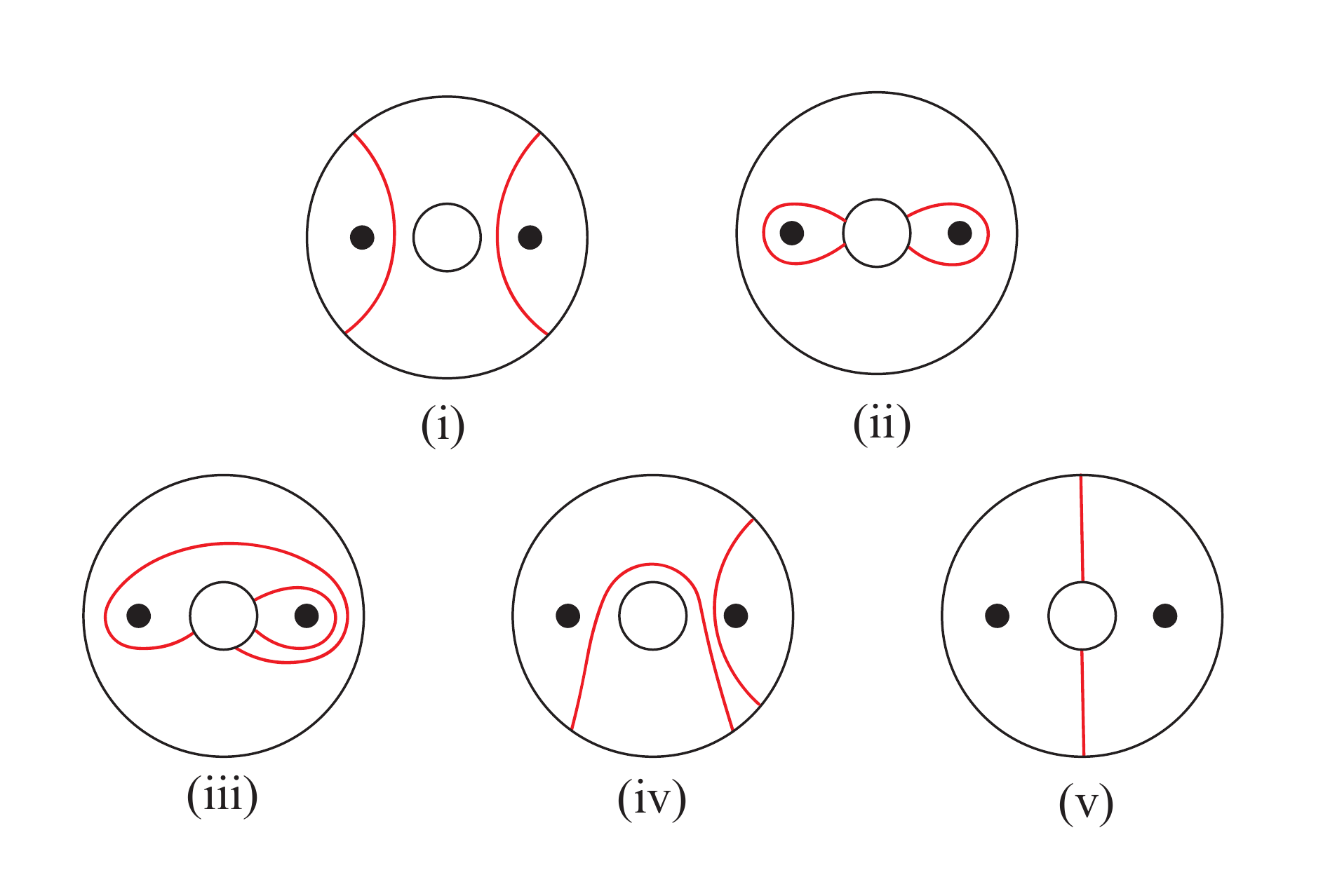}
    \caption{The remaining five possibilities for intersection arcs in nonsingular composition.}
    \label{fig: last nonsingular arcs cases}
\end{figure}

We go through and eliminate cases (i)-(v) of Figure \ref{fig: last nonsingular arcs cases} individually. Suppose that an essential annulus $F$ had intersection curves isotopic to case (i) or (ii) of Figure \ref{fig: last nonsingular arcs cases}. Call these intersection curves $\alpha$ and $\beta$. Then we consider  $M_1$, and reinsert the cork $C_1$. In the process, we cut $F$ along $\alpha$ and $\beta$ and glue these two cuts together along a disk $D$ (note that $D$ has two arcs of its boundary on the outer boundary $\S \times \{1\} $ and an arc on each of $\alpha$ and $\beta$). This yields an annulus $F'$ in $M_1$. Since $M_1$ is tg-hyperbolic, $F'$ must be boundary parallel or compressible.

Suppose $F'$ is boundary parallel. Then since both its boundary circles are nontrivial circles on one boundary of $S \times I $, $F'$ must be boundary parallel with that boundary of $S \times I$. But this is obviously impossible, because the strand of $K$ that went through the cork $C_1$ is between $F'$ and that boundary. Therefore $F'$ cannot be boundary parallel.

Then $F'$ must be compressible, with compressing disk $D$. If $D$ can be isotoped so that $D \cap X = \emptyset$, then $D$ also exists for $F$ in $M$, making $F$ not essential. Now consider $D$ such that it is impossible to isotope $D$ to get an empty intersection with $C$.
Then we compress $F$ along $D$. This yields two disks: $F_1''$ and $F_2''$. Each disk must have trivial boundary circles on $\partial (S \times I)$, which implies the annulus $F$ had trivial boundary circles on $\partial (S \times I)$. Since $F$ must have nontrivial boundary circles to be essential, it must not be essential in $M$. Therefore if $F'$ has a compressing disk in $M_1$, $F$ is not essential in $M$, a contradiction. 


Suppose that an essential annulus $F$ had intersection curves isotopic to case (iii) or (iv) of Figure \ref{fig: last nonsingular arcs cases}. We know that the two intersection curves (call them $\alpha$ and $\beta$) must together bound a disk $D$ on $F$ to one side of $X$. Importantly, $D$ must be a separating disk in  $M_1 \pmc C_1$.  But this is impossible, since $\alpha$ and $\beta$ are separated by a single puncture from $K$ on $X$, implying that a strand of $K$ goes into the region separated by $D$ but never leaves. 


Now suppose that an essential annulus $F$ had intersection curves isotopic to Case (v) of Figure \ref{fig: last nonsingular arcs cases}. Then to one side of $X$, $F$ looks like a disk $D$. For convenience, assuming the disk lies in $M_1$ (with the cork removed), we see that this disk $D$ has two arcs on the boundary of the cork $C_1$, one arc on the inner boundary of $M_1$, and one arc on the outer boundary of $M_1$. But this is exactly the definition of a singular cork (Definition \ref{def: singular cork}), a contradiction to the supposition that $C_1$ was chosen to be nonsingular. Therefore $F$ cannot exist with this set of intersection curves. 


The above arguments have eliminated all sets of just circles and all sets of just arcs in the intersection $F \cap X$. Also note that there cannot be any combinations of intersection arcs and intersection circles. This is because any circle must be nontrivial on $X$ and not parallel to a puncture from $K$, the inner surface, or the outer surface. Thus any intersection circle must enclose either two punctures or a puncture and the hole, implying there cannot be more than one nontrivial, nonparallel intersection arc. By the above arguments, this is impossible. 

Since we have eliminated all possible intersection curves for an essential sphere, annulus, or torus $F$ that minimizes the number of intersection curves on $X$, this implies that any essential $F$ in $M$ must have a conformation in which it is entirely to one side of $X$. This completes the proof.
\end{proof}

\begin{proof} [Proof of Theorem \ref{thm: hyperbolicallycomposable}] Note that if $F$ were an essential sphere, torus or annulus, then $F$ would exist to one side of $X$ by Lemma \ref{no intersections}. For convenience, suppose it is  on the side corresponding to $(S_1 \times I, K_1, C_1).$ Then the surface $F$ exists in the composition $(S_1, K_1, C_1) \#_{ns} (S_1' \times I, K_1', C_1')$. But since $(S_1 \times I, K_1, C_1) \#_{ns} (S_1' \times I, K_1', C_1')$ is tg-hyperbolic, it must either be boundary-parallel or compressible there. 

Suppose $F$ compresses in $(S_1 \times I, K_1, C_1) \#_{ns} (S_1' \times I, K_1', C_1')$ by a disk $D$. Then because $X$ is incompressible, the disk can be made not to intersect $X$. But then $D$ exists in $(S_1 \times I, K_1, C_1) \#_{ns} (S_2 \times I, K_2, C_2)$, a contradiction to $F$ being incompressible there.

Suppose that $F$ is boundary parallel in $(S_1 \times I, K_1, C_1) \#_{ns} (S_1' \times I, K_1', C_1')$. If $F$ is a torus, then because $F$ does not intersect $X$, $F$ must cut a solid torus containing one component of $K$ from $M_1$. Hence, it is also boundary parallel in $(S_1 \times I, K_1, C_1) \#_{ns} (S_2 \times I, K_2, C_2)$, a contradiction. If $F$ is an annulus with boundary in $\partial N(K)$, then together with an annulus in $\partial N(K)$, it cuts a solid torus from $M_1$, making it boundary-parallel in $(S_1 \times I, K_1, C_1) \#_{ns} (S_2 \times I, K_2, C_2)$. And if $F$ is an annulus boundary-parallel to $\partial (S \times I)$ it also must cut a solid torus from $M_1$ making it boundary-parallel in $(S_1 \times I, K_1, C_1) \#_{ns} (S_2 \times I, K_2, C_2)$, again a contradiction. Thus no essential disk annulus, torus, or sphere exists in $M$. By \cite{Thurston hyperbolization}, $M$ is tg-hyperbolic. 

 Since $K_1$ and $K_2$ are tg-hyperbolic, $S_1$ and $S_2$ must be minimal genus for those knots. Since $(S_1 \times I) \#_{ns} (S_2 \times I$) has genus $g(K_1) + g(K_2)$ and  $M$ is tg-hyperbolic and therefore minimal genus, we have that $g(K_1 \# K_2) = g(K_1) + g(K_2)$. 
    
\end{proof}

\subsection{Doubles TG-Hyperbolic Implies Composition is TG-Hyperbolic}\label{subsect: Doubles tg-hyp implies Comp tg-hyp}

In this subsection we prove the following.

\begin{theorem}\label{thm: double tg-hyp implies comp tg-hyp}Let $(S_1 \times I, K_1, C_1)$ and $(S_2 \times I, K_2, C_2)$ be two non-classical triples such that their doubles are tg-hyperbolic. Then, the composition, singular if both corks are singular and nonsingular otherwise,  is tg-hyperbolic. If both corks are singular, the genus is $g(S_1) + g(S_2) -1$ and otherwise the genus is $g(S_1) + g(S_2)$.
\end{theorem}


\begin{proof}  From Theorem 5.5 of \cite{CFW}, it follows  that if we have two tg-hyperbolic manifolds $M_1$ and $M_2$, each containing a separating totally geodesic surface and the two surfaces are homeomorphic, then we can cut $M_1$ and $M_2$ open along the surfaces and glue each piece of $M_1$ to a piece of $M_2$ along copies of the surfaces to obtain two manifolds that are also tg-hyperbolic. In the case of nonsingular composition, $M_1$ and $M_2$ are $D_{ns}(S_1 \times I, K_1, C_1)$ and $D_{ns}(S_2 \times I, K_2, C_2)$ and the separating totally geodesic surfaces are the boundaries of the corks through which we reflected to obtain $D_{ns}(S_1 \times I, K_1, C_1)$ and $D_{ns}(S_2 \times I, K_2, C_2)$. Cutting and pasting along these surfaces yields two copies of the composition, which then must be tg-hyperbolic.

In the case of singular composition, $M_1$ and $M_2$ are again $D_s(S_1 \times I, K_1, C_1)$ and $D_s(S_2 \times I, K_2, C_2)$, but now each separating totally geodesic surface consists of two punctured annuli in each manifold. Theorem 5.5 of \cite{CFW} still applies, and cutting and pasting along these surfaces yields two copies of the composition, which then must be tg-hyperbolic. 

Thus, the only case that remains is the case of composition when $C_1$ is nonsingular and $C_2$ is singular. Note that then the composition is nonsingular composition.

Considering $D_s(S_2 \times I, K_2, C_2)$, we can cut along the two once-punctured totally geodesic annuli to obtain two pieces, each having two once-punctured annuli on it. Then gluing the two copies of the once-punctured annuli to each other on each piece yields two copies of the original $S_2 \times I \setminus K_2$. Hence, $S_2 \times I \setminus K_2$ must itself be tg-hyperbolic by Theorem 5.5 of \cite{CFW}.



Incompressibility and boundary incompressibility of $X$ holds as follows.  To the $M_2$ side, it follows immediately from Lemma \ref{lemma: doublingsurfaceinc} since we made no assumption that the corresponding cork was nonsingular. We just assumed the composition was nonsingular. To the $M_1$ side, they hold, as if not, the corresponding disk for the compression or boundary compression would double to an essential sphere or disk in $D_{ns}(S_1 \times I, K_1, C_1)$, contradicting its tg-hyperbolicity.

Rather than repeat a complete proof that essential surfaces with a minimum number of intersection curves with $X$ must in fact miss $X$, as we did in Lemma \ref{no intersections}, we will highlight the changes necessary in that proof caused by the fact we no longer know $(S_1 \times I, K_1, C_1)$ is tg-hyperbolic. 
(Again, throughout that proof, we did not use the fact $C_2$ was nonsingular.) Note that all of the arguments given in the proof of that lemma still apply to the $M_2$ side of $X$. So we primarily consider what happens to the $M_1$ side. 

\medskip

\noindent {\bf Trivial circles:} As in the proof of Lemma \ref{no intersections}, we create a sphere, which either bounds a ball and is therefore a contradiction to the minimal number of intersection curves of $F$ with $X$ or it does not bound a ball, which will remain true in $D_{ns}(S_1 \times I, K_1, C_1)$, contradicting its tg-hyperbolicity.

\medskip

\noindent {\bf Intersection circles that are nontrivial on both $F$ and $X$:} The argument in the proof of Lemma \ref{no intersections} goes through as is, leaving only the possibility of intersection circles of type (iv) or (v) from Figure \ref{fig: types of nontrivial int curves ns}.

\medskip

\noindent {\bf Parallel intersection circles:} If there are parallel intersection curves, they must all be of type (iv) or all of type (v). As in the proof of Lemma \ref{no intersections}, we generate a torus $F'$ that must be boundary-parallel or compressible in $M$. If $F'$ is boundary-parallel, then $F'$ must still be boundary-parallel to a component $K'$ of $K$, which is isotopic to any one of the intersection curves on $X$. If $K'$ is to the $M_1 \pmc C_1$ side of $X$, then in $D_{ns}(S_1 \times I, K_1, C_1)$ it doubles to a link, the two components of which bound an essential annulus, a contradiction to tg-hyperbolicity of $D_{ns}(S_1 \times I, K_1, C_1)$. And if it is to the $M_2 \pmc C_2$ side of $X$, the argument in the proof of Lemma \ref{no intersections} yields a contradiction.

So $F'$ must be compressible and the proof of Lemma \ref{no intersections} yields a contradiction.

\medskip

\noindent {\bf One intersection curve of type (iv) or (v):} The proof given in Lemma \ref{no intersections} goes through to eliminate this case. So the only remaining possibility is when $F$ is an annulus and the intersection curves are arcs.

\medskip

\noindent  {\bf Intersection arcs that are trivial on $X$:}  As in the proof of  Lemma \ref{no intersections}, we generate a sphere $F'$ that does not intersect $X$. If it is to the $M_1 \pmc C_1$ side, then it appears in $D_{ns}(S_1 \times I, K_1, C_1)$, which is tg-hyperbolic, so the sphere must bound a ball. The argument for when it is to the other side already appears in the proof of  Lemma \ref{no intersections}.

\medskip

\noindent {\bf Intersection curves that are trivial on F:} the argument given in the proof of  Lemma \ref{no intersections} goes through with the only change being that when the sphere $F'$ is to the $M_1 \pmc C_1$ side, the contradiction comes from considering the sphere as a subset of $D_{ns}(S_1 \times I, K_1, C_1)$ as in the previous argument.

\medskip

\noindent {\bf Parallel arcs on $X$:} The argument in Lemma \ref{no intersections} goes through as is.

\medskip
\noindent {\bf Nontrivial, nonparallel intersection arcs:} As in the proof of in Lemma \ref{no intersections}, there are five cases left to consider as in Figure \ref{fig: last nonsingular arcs cases}. In all cases, the two intersection arcs cut $F$ into two disks $F_1$ and $F_2$, with $F_1$ in $M_1 \pmc C_1$ and $F_2$ in $M_2 \pmc C_2$. We eliminate cases (i) and (ii) by considering the disk $F_2$ and applying the argument in the proof of Lemma \ref{no intersections} to $M_2$. The argument to eliminate cases (iii) and (iv) goes through as is.

The last case is case (v). But then each disk $F_i$ has boundary with two arcs on the boundary of the cork $C_i$, one arc on the inner boundary of $M_i$, and one arc on the outer boundary of $M_i$. But this is exactly the definition of a singular cork (Definition \ref{def: singular cork}), a contradiction to the supposition that $C_1$ was chosen to be nonsingular. Therefore $F$ cannot exist with this set of intersection curves.

Since we have eliminated all possible intersection curves for an essential sphere, annulus, or torus $F$ that minimizes the number of intersection curves on $X$, this implies that any essential $F$ in $M$ must have a conformation in which it is entirely to one side of $X$.
We consider the cases when $F$ is in $M_1 \pmc C_1$ and when $F$ is in $M_2 \pmc C_2$ separately.

When $F$ lies in $M_1 \pmc C_1$, then $F$ also lies in $D_{ns}(S_1 \times I, K_1, C_1)$ and by tg-hyperbolicity of $D_{ns}(S_1 \times I, K_1, C_1)$, it must compress or be boundary-parallel there.

Suppose $F$ compresses in $D_{ns}(S_1 \times I, K_1, C_1)$ by a disk $D$. Then because $X$ is incompressible, the disk can be made not to intersect $X$. But then $D$ exists in $M$, a contradiction to $F$ being incompressible there.

Suppose that $F$ is boundary parallel in $D_{ns}(S_1 \times I, K_1, C_1)$. If $F$ is a torus, then because $F$ does not intersect $X$, $F$ must cut a solid torus containing one component of $K$ from $M_1$. Hence, it is also boundary-parallel in $M$, a contradiction. If $F$ is an annulus with boundary in $\partial N(K)$, then together with an annulus in $\partial N(K)$, it cuts a solid torus from $M_1$, making it boundary-parallel in $M$. And if $F$ is an annulus boundary-parallel to $\partial (S \times I)$ it also must cut a solid torus from $M_1$ making it boundary-parallel in $M$, again a contradiction.

When $F$ lies in $M_2 \pmc C_2$, then $F$ lies in $M_2$ which is tg-hyperbolic. So it must either compress or be boundary-parallel there.

Suppose $F$ compresses in $M_2$ by a disk $D$. Then because $X$ is incompressible, the disk can be made not to intersect $X$. But then $D$ exists in $M$, a contradiction to $F$ being incompressible there.

Suppose that $F$ is boundary parallel in $M_2$. If $F$ is a torus, then because $F$ does not intersect $X$, $F$ must cut a solid torus containing one component of $K$ from $M_2 \pmc C_2$. Hence, it is also boundary-parallel in $M$, a contradiction. If $F$ is an annulus with boundary in $\partial N(K)$, then together with an annulus in $\partial N(K)$, it cuts a solid torus from $M_2 \pmc C_2$, making it boundary-parallel in $M$. And if $F$ is an annulus boundary-parallel to $\partial (S \times I)$ it also must cut a solid torus from $M_2 \pmc C_2$ making it boundary-parallel in $M$, again a contradiction. Thus no essential disks, annuli, tori, or spheres exist in $M$, and by \cite{Thurston hyperbolization} $M$ is tg-hyperbolic.
The conclusion about the genus of the composition is immediate from the fact that the resulting manifold is tg-hyperbolic and therefore cannot possess any reducting annuli.

\end{proof}

\subsection{Finding Corks to Obtain TG-Hyperbolicity of Composition}

\begin{theorem}\label{thm: composition tg-hyperbolicity}
Given two non-classical virtual knots or links that are tg-hyperbolic,  there is a choice of corks so that the composition, singular if both corks are singular and nonsingular otherwise, is tg-hyperbolic. In the case of singular composition, the resulting genus is $g(K_1) + g(K_2) -1$ and otherwise it is $g(K_1) + g(K_2)$.
\end{theorem}


\begin{proof}
If both knots have a singular cork, we are done by Theorem \ref{thm: singular tg-hyperbolicity}. If either knot does not have a singular cork, 
we describe how to obtain the requisite cork. Let $M =S \times I \setminus K$ which we are taking to be tg-hyperbolic. Let $N(K) \setminus K$ be a choice of embedded cusp corresponding to a knot $K$. Let $\gamma_1$ be the shortest geodesic that runs from $S \times \{0\}$ to $\partial N(K)$ in the hyperbolic  metric on $S \times I \setminus K$. Then $\gamma_1$ ends on a geodesic meridian $m$ on $\partial N(K)$. Choose $\gamma_2$ to be the shortest geodesic that runs from $m$ to $S \times \{1\}$. Let the cork $C'$ be a regular neighborhood of $\gamma_1 \cup \gamma_2 \cup m$. We call such a cork a {\it 2-geodesic cork}. 

Note that this cork can be isotoped so that it does appear as $D \times I$ in $S \times I$ with $D \times I$ intersecting the knot in an unknotted arc, so the definition is compatible with the definition of a cork above. We prove that the double corresponding to a 2-geodesic cork is tg-hyperbolic and then the result follows from Theorem \ref{thm: double tg-hyp implies comp tg-hyp}. 

Since  we are assuming the knot does not have a singular cork, $C'$ must be nonsingular. By Lemmas \ref{lemma: disk generates sphere}, \ref{lemma: doublingsurfaceinc} and \ref{no intersections},  we need only eliminate essential spheres, annuli and tori that do not intersect the doubling surface $X$. 

Suppose there exists an essential surface $F$ in the double that does not intersect $X$. Then $F$ lives entirely in one copy of $M \pmc C'$. Therefore we can reinsert $C'$ to obtain a copy of $M$ that also completely contains $F$. Since $M$ is tg-hyperbolic, $F$ cannot be essential in $M$. So if $F$ is a sphere, it bounds a ball, and since $C'$ cannot intersect the ball, that ball exists in $D(K,C')$ as well, contradicting $F$'s essentiality there. Thus $F$ must be a torus or annulus and it must be compressible or boundary parallel in $M$.

If $F$ is compressible, this means there exists a compressing disk $D$ for $F$ contained in $M$. We know that $F \cap C' = \emptyset$, $D$ is a disk, and $C' \cap K$ is an unknotted arc. If $D$ can be isotoped so that $D \cap X = \emptyset$, then $D$ also exists for $F$ in $D(K, C')$, contradicting the fact $F$ is essential. Now consider $D$ such that it is impossible to isotope $D$ to get an empty intersection with $C'$. 

First, suppose $F$ is a torus. Then compressing along $D$ yields a sphere $F'$ in $M$. If $F'$ does not bound a ball, then it is an essential sphere in the tg-hyperbolic manifold $M$, a contradiction. Therefore $F'$ bounds a ball in $M$, which implies $F$ must have bounded a solid torus or a knot exterior in $M$. Since $F$ separates the manifold and neither a solid torus nor a knot exterior contain boundary surfaces, it must be that $F$ bounds a solid torus or knot exterior to the side that is away from the boundaries of $S \times I$. In particular, since the cork touches the boundaries of $S \times I$, $F$ bounds a solid torus or knot exterior in $D(K,C')$. But it cannot bound a solid torus, since then it would compress in $D(K,C')$. 

     So $F$ bounds a knot exterior $Q$ in both $D(K, C')$ and $M$. It must compress in $S \times I$, and after compression it bounds a ball $B$. So there is a ball $B$ in $M$ that contains the knot exterior. The ball avoids the knot in $M$ but must intersect the cork. We lift $M$ to hyperbolic 3-space $\mathbb{H}^3$. The ball $B$ lifts to copies of the ball, and the knot lifts to a collection of horoballs. The cork lifts to belts around the horoballs, corresponding to the neighborhood of the meridian,  together with neighborhoods of geodesics connecting horoballs to hyperbolic planes corresponding to the boundaries of $S \times I$.  See Figure \ref{fig:corklift}. (In the case $S$ is a torus, the geodesic planes are replaced by horospheres.) But because the ball $B$ intersects the cork, the lifted balls must intersect the lifts of the corks. Choosing one copy of a lifted ball $B'$, it contains a copy $Q'$ of the knot exterior $Q$. Let $P$ be the complement of the lifted corks in $B'$. This is a handlebody with free fundamental group. But then the knot exterior $Q'$ must have fundamental group appearing as a subgroup of the fundamental group of $P$. The Nielsen-Schreier Subgroup Theorem implies this subgroup must also be free, but the fundamental group of a nontrivial knot exterior is never free, a contradiction.


     \begin{figure}[htbp]
    \centering
    \includegraphics[width=0.7\textwidth]{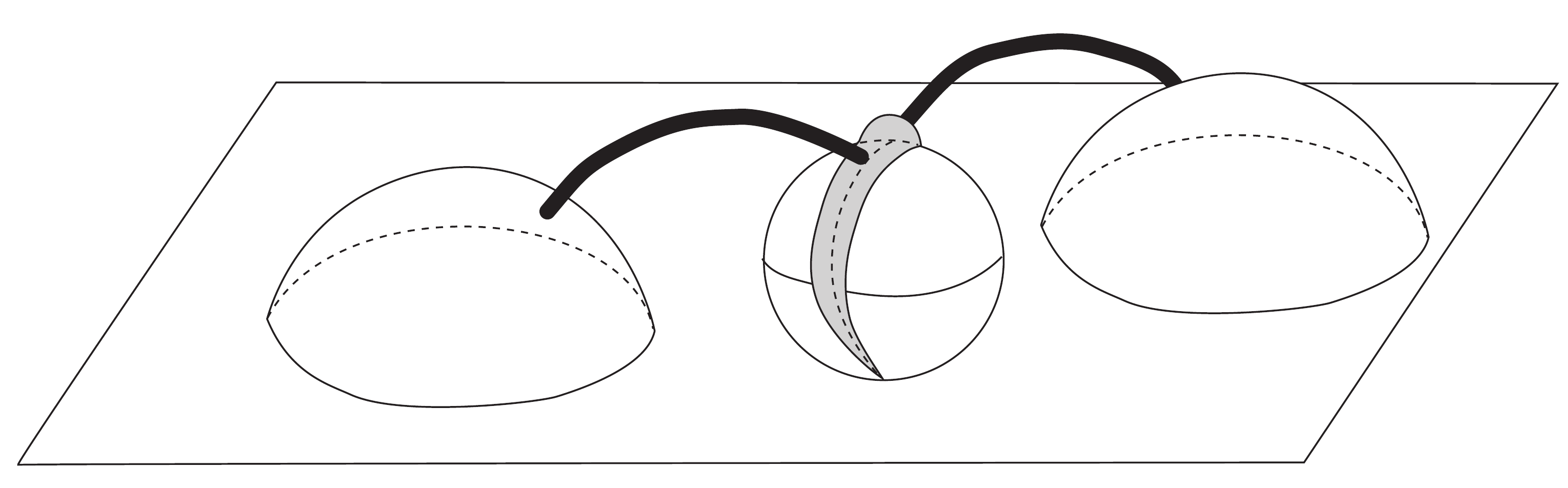}
    \caption{A lift of a 2-geodesic cork to $\mathbb{H}^3$.}
    \label{fig:corklift}
\end{figure}



This means $F$ must be boundary parallel in $M$. Since $F \cap X$ is empty, $F$ cannot be boundary parallel to either the inner or outer surface. Also, $F$ cannot be boundary parallel to the knot component that passes through $C'$ since $C'$ touches both the inner and outer surface. So $F$ must be a torus boundary parallel to a component of $K$ in $M \pmc C'$. Since $M \pmc C'$ is not changed in the doubling procedure, $F$ will still be boundary parallel and therefore inessential in $D(K, C')$.

In the case $F$ is a compressible annulus, compression yields two disks in $M$, which must have trivial boundaries since $M$ is tg-hyperbolic, contradicting the essentiality of the annulus in $D(K,C').$ So the annulus must be boundary-parallel. But if the annulus is parallel to one of the  boundaries of $M$, then it still is in  $M \pmc C'$, a contradiction. And it it is parallel into the boundary of a neighborhood of a knot component, then it still is in $M \pmc C'$, again a contradiction.

The conclusion about the  genus of the composition is immediate from the fact the resulting manifold is tg-hyperbolic and therefore cannot possess any reducing annuli.





\end{proof}

\section{Volume Bounds}\label{section: volume bounds}

This section consists of three subsections based on the type of cork chosen for each knot or link. This is because singular and nonsingular cork choices will generate different volume bounds. 

\subsection{Nonsingular Compose Nonsingular}
In this subsection we  prove the following theorem:
\begin{theorem} \label{thm: nonsing compose nonsing vol bound} 
Given two non-classical hyperbolically composable triples $(S_1 \times I, K_1, C_1)$ and $(S_2 \times I, K_2, C_2)$, the volume of the composition $M = (S_1 \times I, K_1, C_1) \#_{ns} (S_2 \times I, K_2, C_2)$ is bounded below by the nonsingular composite volumes of $(S_1 \times I, K_1, C_1)$ and $(S_2 \times I, K_2, C_2)$:
$$vol (M) \geq vol_{ns}(S_1 \times I, K_1, C_1) + vol_{ns}(S_2 \times I, K_2, C_2)  $$
\end{theorem}


\begin{proof} 
We begin with each $D(K_i, C_i)$. In each, we know that $\partial C_i$ is totally geodesic since it is the fixed point set of an orientation reversing isometry of the double. We can then cut along each $\partial C_i$ and glue each copy of $(S_1 \times I, K_1, C_1)$ to a copy of $(S_2 \times I, K_2, C_2)$, resulting in two copies of $M = (S_1 \times I, K_1, C_1) \#_{ns} (S_2 \times I, K_2, C_2)$. Applying \cite{AST} Theorem 7.1 as generalized in \cite{CFW} Theorem 5.5 to the (disjoint) union of these two manifolds and reversing the cut-and-paste process just described, and using the fact we defined $vol_{ns}(S_i \times I, K_i, C_i) = \frac{1}{2} vol(D_{ns}(S_i \times I, K_i, C_i))$, we see that 
\begin{align*}
    2 vol(M) & \geq vol(D_{ns}(S_1 \times I, K_1, C_1)) + vol(D_{ns}(S_2 \times I, K_2, C_2)) \\
    & \geq 2 vol_{ns}(S_1 \times I, K_1, C_1) + 2 vol_{ns}(S_2 \times I, K_2, C_2)
\end{align*} 
 Dividing by 2 yields the claim. 
\end{proof}

\subsection{Singular Compose Singular}
When both corks in a composition are singular, then the volume bounds follow the same idea as when both corks are nonsingular. Recalling that the singular composite volume is simply one half the volume of the singular cork double of a knot or link $K$, we have the following theorem: 
\bigskip 
\begin{theorem} \label{thm: sing compose sing vol bound}
 Let $(S_1 \times I, K_1, C_1)$ and $(S_2 \times I, K_2, C_2)$ be two tg-hyperbolic triples such that both $C_1$ and $C_2$ are singular. If both have genus one or both have genus at least two, then the volume of the composition $M = (S_1 \times I, K_1, C_1) \#_s (S_2 \times I, K_2, C_2) $ is bounded below by the singular composite volumes of $(S_1 \times I, K_1, C_1)$ and $(S_2 \times I, K_2, C_2)$:
$$vol (M) \geq vol_s(S_1 \times I, K_1, C_1) + vol_s(S_2 \times I, K_2, C_2)  $$
\end{theorem}

\begin{proof}
Since the corks are both singular, we know that the manifolds obtained by performing singular composition with a second reflected copy of $(S_i \times I, K_i, C_i)$ are tg-hyperbolic by Theorem \ref{thm: singular tg-hyperbolicity}, and the gluing surfaces are totally geodesic. We then cut along each $\gamma_i \times I$ and glue each copy of $(S_1 \times I, K_1, C_1)$ to a copy of $(S_2 \times I, K_2, C_2)$ along the new boundaries $\partial (\gamma_i \times I)$, resulting in two copies of $M = (S_1 \times I, K_1, C_1) \#_s (S_2 \times I, K_2, C_2)$. Again, applying \cite{AST} Theorem 7.1 as generalized in \cite{CFW} Theorem 5.5 to the (disjoint) union of these two manifolds, we see that 
\begin{align*}
    2 vol (M) & \geq vol(D_s(S_1 \times I, K_1, C_1)) + vol(D_s(S_2 \times I, K_2, C_2)) \\
    & \geq 2 vol_s(S_1 \times I, K_1, C_1) + 2 vol_s(S_2 \times I, K_2, C_2) 
\end{align*} 
since we defined $vol_s(S_i \times I, K_i, C_i) = \frac{1}{2} vol(D_s(S_i \times I, K_i, C_i))$. Dividing by 2 completes the proof.
\end{proof}

Note that we require both knots to have genus one or both knots to have genus greater than 1 because of issues with the boundary of manifolds. The manifold defined by a knot in a genus one thickened surface does not include its boundary, while a manifold defined by a removing a knot from a genus $g \geq 2$ thickened surface does include its boundary. In order to apply Theorem 7.1 from \cite{AST} (or Theorem 5.5 of \cite{CFW}),  either both pieces in the composition must include the boundary, or both must exclude the boundary. 


It is also worth noting that there is a particularly nice special case for singular composition. 
    \begin{corollary} \label{cor: genus 1 sing comp}
        Given two tg-hyperbolic triples $(S_1 \times I, K_1, C_1)$ and $(S_2 \times I, K_2, C_2)$ such that both $C_1$ and $C_2$ are singular and both $S_1$ and $S_2$ have genus $g = 1$,  the volume of the composition \\
        $M = (S_1 \times I, K_1, C_1) \#_s (S_2 \times I, K_2, C_2)$ is equal to the sum of the hyperbolic volumes of $K_1$ and $K_2$:
        $$vol (M) = vol(K_1) + vol(K_2) $$
    \end{corollary}
    
\begin{proof}

When the genus is one, the knots or links live in thickened tori, not including their boundaries. This means that the cutting surface $\gamma \times (0,1)$ is  a thrice-punctured sphere, which  is totally geodesic in each $M_i$ by \cite{Adams thrice punctured spheres}. As in that paper, cutting and pasting along thrice-punctured spheres preserves hyperbolic volume of the pieces. 

Note that this also follows from the special case of Theorem 5.5  of \cite{CFW} when all the surfaces are totally geodesic before and after cutting and pasting. 

Therefore, 
$$vol ((S_1 \times I, K_1, C_1) \#_s (S_2 \times I, K_2, C_2)) = vol(K_1) + vol(K_2) $$
\end{proof}

\subsection{Nonsingular Compose Singular}
To tackle this case, we make use of \cite{AST} and properties of the Gromov norm $\gnorm{M}$ found in \cite{Gromov}. Specifically, for a tg-hyperbolic manifold $M$, the volume $vol(M) = \frac{1}{2}v_3 \gnorm{DM}$ where $v_3$ is the volume of an ideal regular tetrahedron and DM is the double of the manifold over its totally geodesic boundary. Also, removing a manifold of codimension one does not change the Gromov norm. We will also use Theorem 9.1 of \cite{AST}, reproduced here: 

\begin{theorem*} \label{AST thm 9.1}
Let $\overline{N}$ be a compact manifold with interior $N$, a hyperbolic 3-manifold of finite volume. Let $\Sigma$ be an embedded incompressible ($\pi_1$-injective) surface in $N$. Then
$$ vol(N) \geq \frac{1}{2} v_3 \gnorm{D(N \pmc \Sigma)}$$
\end{theorem*}

We will ultimately prove the following volume bound: 

\begin{theorem}\label{thm: nonsing compose sing vol bound} 
    Let $(S_1 \times I, K_1, C_1)$ and $(S_2 \times I, K_2, C_2)$ be two non-classical tg-hyperbolic triples such that $C_1$ is singular and $(S_2 \times I, K_2, C_2)$ is hyperbolically composable. Then the volume of the composition $M = (S_1 \times I, K_1, C_1) \#_{ns} (S_2 \times I, K_2, C_2)$ is bounded below: 
    $$vol(M) \geq \frac{1}{4}vol(D(D_{ns}(S_1 \times I, K_1, C_1)) \pmc T) + \frac{1}{2}vol(D_{ns}(S_2 \times I, K_2, C_2))$$
\end{theorem}
\noindent Before proving the volume bounds for the composition, we describe the manifold  $D(D_{ns}(S_1 \times I, K_1, C_1)) \pmc T$ and prove that it is tg-hyperbolic.

Given a tg-hyperbolic triple $(S \times I, K, C)$ such that $C$ is singular, first double $M \pmc C$ over $\partial C$. Call this manifold $M'$. The disk in $M \pmc C$ given by singularity doubles to an annulus. Now double $M'$ over its inner and outer surface boundaries to obtain $DM'$. This manifold is now a quadruple of $M \pmc C$. The annulus  doubles  to an essential torus $T$ in $DM'$. The last step is to remove a neighborhood of this torus resulting in $DM' \pmc T$, which because of the removal of $T$, has two more cusp ends than did $DM'$. Ignoring the knot momentarily, the manifold that results is the product of a circle with a closed orientable surface of genus $g_1 + g_2 -1$ with the interior of two disks removed. Call this manifold $Q$. Then $DM' \pmc T$ is obtained from $Q$ by removing link components.

See Figure \ref{fig: constructing DM' minus T} for an example of how to construct $DM' \pmc T$. Also note that $D(D_{ns}(S_1 \times I, K_1, C_1)) \pmc T = DM' \pmc T$. We use the former definition in the theorem statement in order to highlight the dependence of this object on $K_1$, but from now on we will use $DM' \pmc T$ for readability.

\begin{figure}[htbp]
    \centering
    \includegraphics[width=0.4\textwidth]{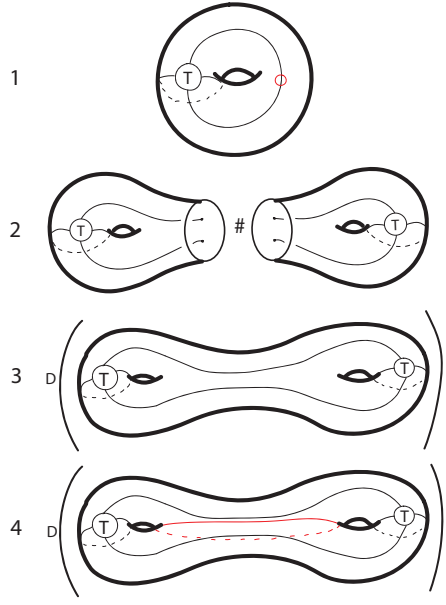}
    \caption{Constructing $DM' \pmc T$ when $K$ has genus 1. The inner boundary surface is not depicted. Here, the $D$ refers to doubling the picture over the inner and outer surface, and in step 4 the torus defined by the curve in red is removed. Although this curve defines an annulus in $M'$, doubling over the inner and outer boundaries generates the torus $T$.}
    \label{fig: constructing DM' minus T}
\end{figure}
\begin{claim*}
$DM' \pmc T$ is hyperbolic.
\end{claim*}
\begin{proof}

Like our earlier tg-hyperbolicity proofs, we eliminate the possibilities of essential disks, annuli, tori, and spheres. 


Let $T'$ and $T''$ be the two copies of $T$ in the boundary of $DM' \pmc T$ that result from the removal of $T$.
Suppose an essential disk $D$ exists in $DM' \pmc T$. If it does not touch the link components, then it exists in $Q$,  a contradiction to the fact the product of  a circle with a  non-disk surface with boundary has incompressible boundary. If $D$ has boundary in the boundary of a regular neighborhood of the link components, then it generates an essential sphere, which we will shortly eliminate. 



We want to consider just one copy of $M'$ in order to leverage our earlier proofs. In order to do this, we need to show that if $F$ is essential in $DM' \pmc T$, then $F$ also generates at least one essential surface in $M' \pmc A_T$, where $A_T$ is the essential annulus in $M'$ that doubles to $T$ in $DM'$. If $F$ lives in only one copy of $M'$, then the fact it is essential in the manifold $DM' \pmc T$ implies it is also essential in the submanifold $M' \pmc A_T$.

So suppose $F$ intersects the doubling boundary of $M'$. We assume $F$ intersects $\partial M'$ in a minimal number of intersection curves.  Then by incompressibility of $\partial M'$, we can assume the only simple closed curves are nontrivial on both $\partial M'$ and on $F$. This immediately eliminates essential spheres. If $F$ is an essential annulus or torus, then all intersections must be parallel nontrivial curves on $F$, implying that all components of $F \cap M'$ are annuli for either copy of $M'$. The annuli must be essential in each copy of $M'$ or we could lower the number of intersection curves of $F$ with $\partial M'$. Thus, the existence of $F$ implies there is an essential annulus or torus $F'$ in $M'$. 


Since $M'$ is created through nonsingular composition, all of our arguments to prove nonsingular composition is hyperbolic go through, except for where we leveraged the nonsingularity of the cork, which occured when we eliminated the case appearing in Figure \ref{fig: last nonsingular arcs cases} (v). This means that in $M'$, any essential annuli or tori must have intersection curves with $X$ isotopic to the intersection curves $A_T$ has with $X$. Thus, there are no essential tori intersecting $X$.



Now suppose $F'$ is an annulus such that $F' \cap X$ has intersection curves $\alpha_1$ and $\alpha_2$ that are isotopic to the intersection curves in $A_T \cap X$. The arcs $\alpha_1$ and $\alpha_2$ cut $F'$ into two disks $F_1'$ and $F_2'$.  The intersections of $A_T$ with $X$ also cut $A_T$ into two disks $A_T^1$ and $A_T^2$. We will show that $F'$ must be isotopic to $A_T$ by showing that $F_1'$ is isotopic to $A_T^1$ in $M_1$ and $F_2'$ is isotopic to $A_T^2$ in $M_2$, where $M_1$ and $M_2$ are the two copies of $M$ utilized to create the composition.  
Note that $F' \cap A_T = \emptyset$, since $F' \subseteq F$, and $F$ lives in $DM' \pmc T$. Both $F'$ and $A_T$ have intersection arcs on $X$ that have a single endpoint on each boundary circle of $X$. This means that when reversing the composition and reinserting copies of $C$ to obtain $M_1$ and $M_2$, we can form an annulus by taking $F_1'$, $A_T^1$ and two disks on $\partial C$ in $M_1$, and a second annulus by taking $F_2'$, $A_T^2$ and two disks on $\partial C$ in $M_2$. In each case, because $M_i$ is tg-hyperbolic and hence minimal genus, the annulus must bound a $D_i \times I$ to one side, where $D_i$ is a disk. Hence we can isotope $F_i'$ to $A_T^i$ through $D_i \times I$. But doing both, we then have that $A_T$ and $F'$ are isotopic in $M'$. This implies that $F$ must be isotopic to $T$ and thus $F$ is boundary parallel in $DM' \pmc T$ and therefore was not essential, a contradiction. 




Now suppose $F$ is an essential annulus in $DM' \pmc T$ with boundary circles on the boundary created by removing $T$. Obviously, the boundary circles $\alpha_1$ and $\alpha_2$ of $F$ must be nontrivial on $T$. Thus $\alpha_1$ and $\alpha_2$ bound an annulus $A$ on $T$. Then consider the torus $F'$ obtained by taking the union $A \cup F$ in $DM'$. By the above arguments, we know that the only essential surface (up to isotopy) in $DM'$ is $T$. Therefore $F'$ is not essential in $DM'$, and so must be boundary parallel or compressible. 

$F'$ cannot be boundary parallel, as there is no boundary in $DM'$. Thus $F'$ must be compressible, so it has a compressing disk $D$. Since $T$ is essential, it is also incompressible. Therefore $D$ cannot have boundary on $A$, and so must also be a compressing disk for $F$ in $DM' \pmc T$. This means $F$ is not essential in $DM' \pmc T$. Since there are no essential disks, spheres, annuli, or tori, $DM' \pmc T$ is hyperbolic by Thurston \cite{Thurston hyperbolization}.
\end{proof}

Now we are ready to bound the volume of a composition, proving Theorem \ref{thm: nonsing compose sing vol bound}.

\begin{proof}[Proof of Theorem \ref{thm: nonsing compose sing vol bound}.]

Start with our two tg-hyperbolic triples $(S_1 \times I, K_1, C_1)$ and $(S_2 \times I, K_2, C_2)$ such that $C_1$ is singular and $C_2$ is nonsingular and hyperbolically composable. Compose these two and double over the outer and inner boundaries to get $N = D((S_1 \times I, K_1, C_1) \#_{ns} (S_2 \times I, K_2, C_2))$. This is a double of the object whose volume we want to bound; so set $V = \frac{1}{2} vol(D((S_1 \times I, K_1, C_1) \#_{ns} (S_2 \times I, K_2, C_2)))$. We also know that this object is tg-hyperbolic from Theorem \ref{thm: hyperbolicallycomposable}.

Next, let $\Sigma$ be the double of the cork boundary in this object. Removing a neighborhood of $\Sigma$ yields a copy of $M_1 \pmc C_1$ doubled over the inner and outer boundary of $M_1$ and a copy of $M_2 \pmc C_2$ doubled over its inner and outer boundary. 
We can therefore apply Theorem 9.1 of \cite{AST} to get 
$$vol(N) > \frac{1}{2}v_3 \gnorm{D(D_{ns}(S_1 \times I, K_1, C_1)) \cup D(D_{ns}(S_2 \times I, K_2, C_2))}$$
where the second doubling of $D_{ns}(S_i \times I, K_i, C_i)$ is over $\partial C_i$. Note there is not equality because $\Sigma$ is not totally geodesic in $N$. Using properties of Gromov norm explained above, we have 
\begin{align*}
    vol(N) & > \frac{1}{2}v_3 \gnorm{D(D_{ns}(S_1 \times I, K_1, C_1)) \cup D(D_{ns}(S_2 \times I, K_2, C_2))}\\
    & \geq \frac{1}{2}v_3 \gnorm{D(D_{ns}(S_1 \times I, K_1, C_1))} + \frac{1}{2}v_3 \gnorm{D(D_{ns}(S_2 \times I, K_2, C_2))}  \\
    & \geq \frac{1}{2}v_3 \gnorm{D(D_{ns}(S_1 \times I, K_1, C_1)) \pmc T} + \frac{1}{2}v_3 \gnorm{D(D_{ns}(S_2 \times I, K_2, C_2))}  \\
    & = \frac{1}{2}vol(D(D_{ns}(S_1 \times I, K_1, C_1)) \pmc T) + \frac{1}{2}vol(D(D_{ns}(S_2 \times I, K_2, C_2))\\
    & = \frac{1}{2}vol(D(D_{ns}(S_1 \times I, K_1, C_1)) \pmc T) + vol(D_{ns}(S_2 \times I, K_2, C_2))
\end{align*}

Dividing by 2 yields the desired inequality. 
\end{proof}

Now we have a volume bound for $(S_1 \times I, K_1, C_1) \#_{ns} (S_2 \times I, K_2, C_2)$ in terms of volumes of other tg-hyperbolic manifolds whose volumes can be computed in SnapPy.

\section{Arbitrarily High Volumes and Alernating Virtual Links}\label{sect: infinite vol limit}

Before proving Theorem \ref{thm: vol limits to infinity}, we need a few definitions.

\begin{definition}
A \emph{twist region} in a projection is a maximal sequence of end-to-end bigons. If no bigons touch a crossing, then that crossing by itself is also a twist region.
\end{definition}

\begin{definition}
Let $K$ be a knot living in a thickened surface $S \times I$. Call the projection surface $P = S \times \{\frac{1}{2}\}$, and let $G$ be the 4-valent graph with crossing information at vertices obtained by projecting $K$ onto $P$. Then we can \emph{augment} $K$ by adding a single trivial component around a single crossing of $K$, where the component is perpendicular to $P$ and punctures $P$ exactly twice in two distinct non-adjacent regions of $G$ that touch each other just at the crossing. This new component is called an \emph{augmenting component}. We say that $K$ has been \emph{augmented} if at least one augmenting component has been added. Augmenting a twist region means that we add one augmenting component around any one of the crossings in the twist region such that it does not intersect any of the bigons.
\end{definition}

\begin{theorem} \label{thm: vol limits to infinity} 
Given two tg-hyperbolic alternating virtual knots or links $K_1, K_2$ with genera $g_1, g_2 > 0$, there exists a sequence of tg-hyperbolic compositions $W_i = (S_1 \times I, K_1, C_1^i) \#_{ns} (S_2 \times I, K_2, C_2)$ such that $$\lim_{i \to \infty}(vol(W_i)) = \infty.$$
\end{theorem}

Note that a virtual link is alternating if either the classical crossings alternate when we ignore the virtual crossings in a projection of the knot, or equivalently, the knot is realized by a knot in a thickened surface with projection that alternates on the surface. Since  the link is virtual, it will be realized with a cellular alternating link projection  on the surface. The link is  weakly  prime (sometimes called obviously prime) if the Gauss code has no alternating classical subcode (see \cite{tg hyperbolicity in S cross I})  or equivalently, if on the resultant projection on the surface, there are no disks containing crossings with boundary circle that  intersect the projection transversely twice. In the case that the virtual link is alternating and weakly prime,  it  will be tg-hyperbolic by  Theorem 1 of \cite{small18}.
\begin{proof}

We will prove this theorem by showing that we can create a sequence of projections $P_1, P_2, \dots$ of $K_1$ on $S_1$ such that we can choose a cork $C_1^i$ for $P_i$  and a cork $C_2$ for $K_2$ such that $(S_1 \times I, K_1, C_1^i) \#_{ns} (S_2 \times I, K_2, C_2)$ is tg-hyperbolic with volume increasing without bound.   Letting $i \xrightarrow[]{} \infty$ will yield  compositions with volume tending to $\infty$. 

Given a projection of $K_1$, see Figure \ref{fig:increasing twist number} for how to create twist regions with arbitrarily many crossings, and see Figure \ref{fig:increasing twist regions} for how to increase the number of twist regions. Note that both of these processes are done by repeating Type II Reidemeister moves. The red circle in each figure shows the choice of cork. Combining these two processes yields  a projection $P_i$ of $K_1$ with $2i$ many twist regions, each containing $2j-1$ crossings, as shown in Figure \ref{fig:increasing twist number and regions}. Because the new crossings occur in a disk in the projection plane that does not involve any virtual crossings, these same crossings and the same local diagram will appear on the surface $S_1$ when we lift the projection to the surface.

\begin{figure}[htbp]
    \centering
    \includegraphics[width=0.5\textwidth]{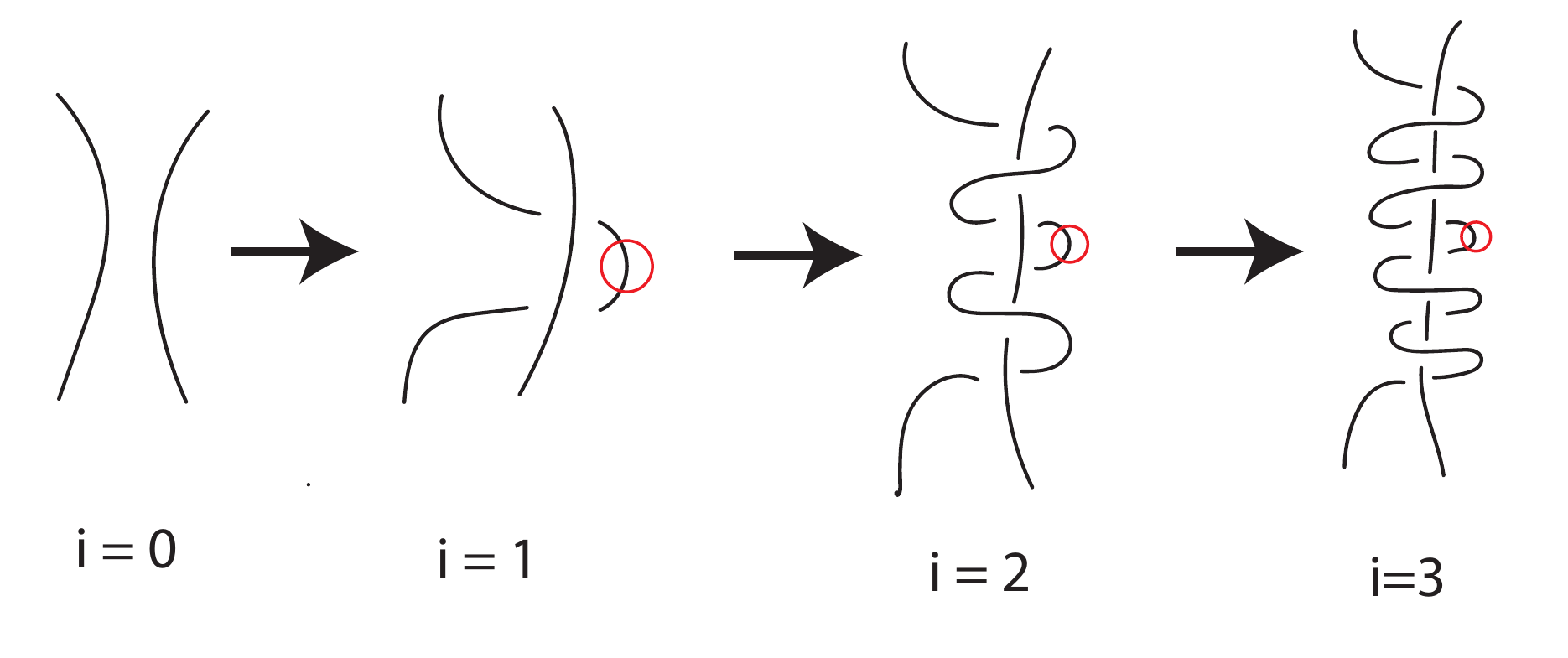}
    \caption{Increasing the number of crossings in each twist region.}
    \label{fig:increasing twist number}
\end{figure}

\begin{figure}[htbp]
    \centering
    \includegraphics[width=0.5\textwidth]{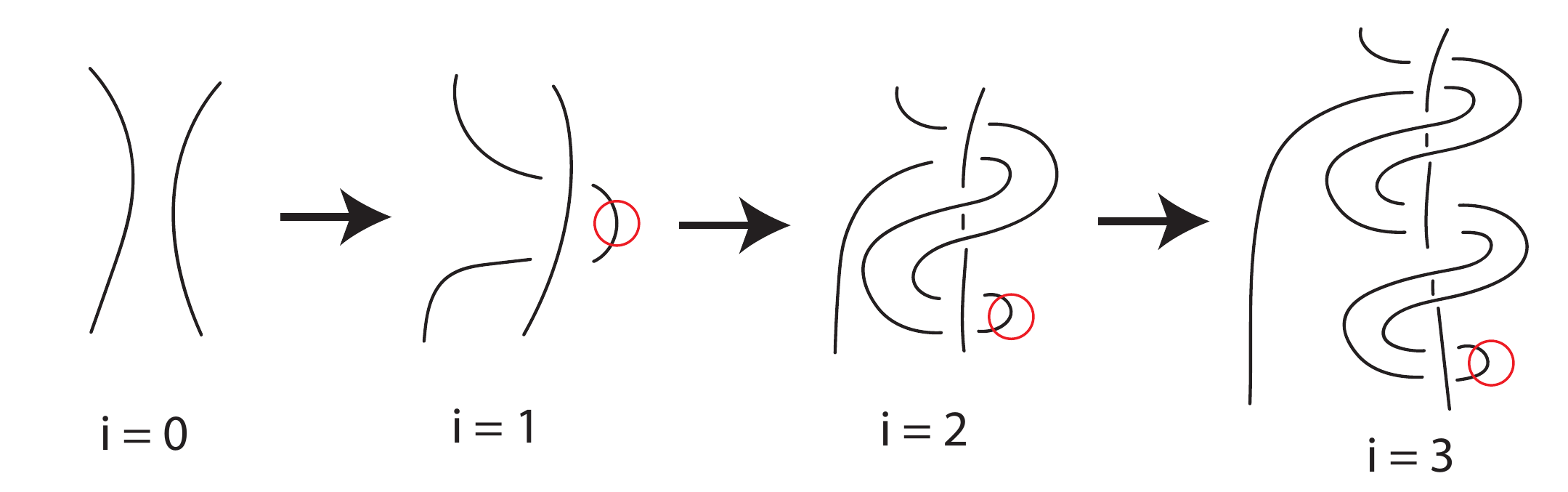}
    \caption{Increasing the number of twist regions.}
    \label{fig:increasing twist regions}
\end{figure}

\begin{figure}[htbp]
    \centering
    \includegraphics[width=0.3\textwidth]{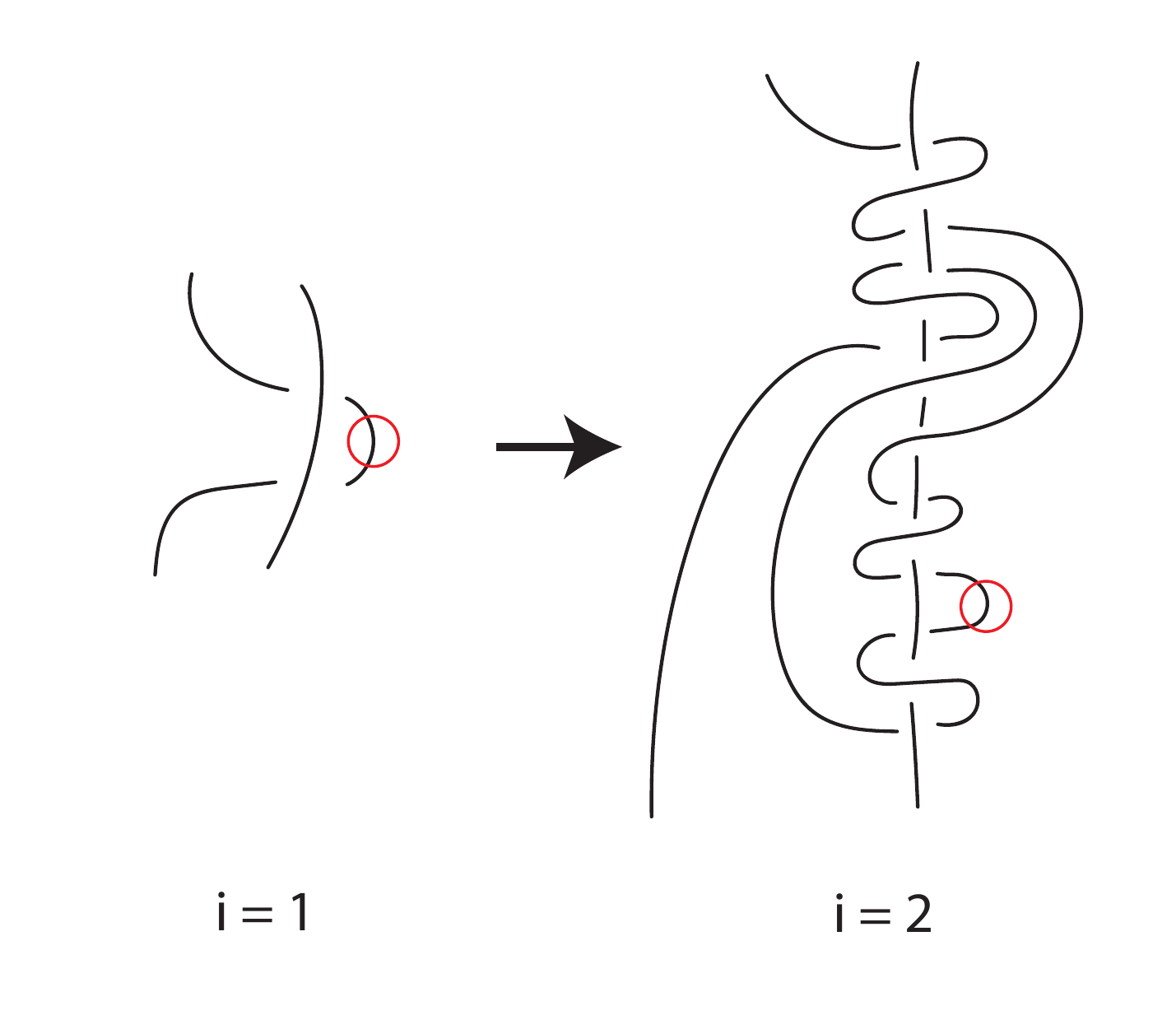}
    \caption{Increasing the number of twists per twist region, and the number of twist regions simultaneously.}
    \label{fig:increasing twist number and regions}
\end{figure}


Next, we augment each of the $2i$ twist regions exactly once, meaning we add a trivial component that wraps around it, obtaining  a new link $L_i$. Figure \ref{fig: augmenting twist region} shows the process of augmenting a twist region in parts (i) and (ii). 

We can cut along a twice-punctured disk bounded by an augmented component and twist full twists before regluing  to add or remove any even number of twists and preserve the link complement up to homeomorphism. This process can take us from $L_i$ to a link $L_i'$ that has zero or one one half twist at each augmenting component. Also, it was proved in \cite{Adams thrice punctured spheres} that in a hyperbolic 3-manifold, cutting along a twice-punctured disk and twisting a half twist before regluing preserves tg-hyperbolicity and does not change volume. Therefore, since $K_1$ was alternating,  we can set half twists of the augmenting components to have one crossing at each augmented component so that for the resulting link $L_i''$, if we remove the augmenting components to obtain $K_1'$, we have an alternating link. Because the original projection of $K_1$ on $S_1$ yielded a hyperbolic manifold, it must be that $K_1'$ is weakly prime.  

. 


\begin{figure}[htbp]
    \centering
    \includegraphics[width=0.3\textwidth]{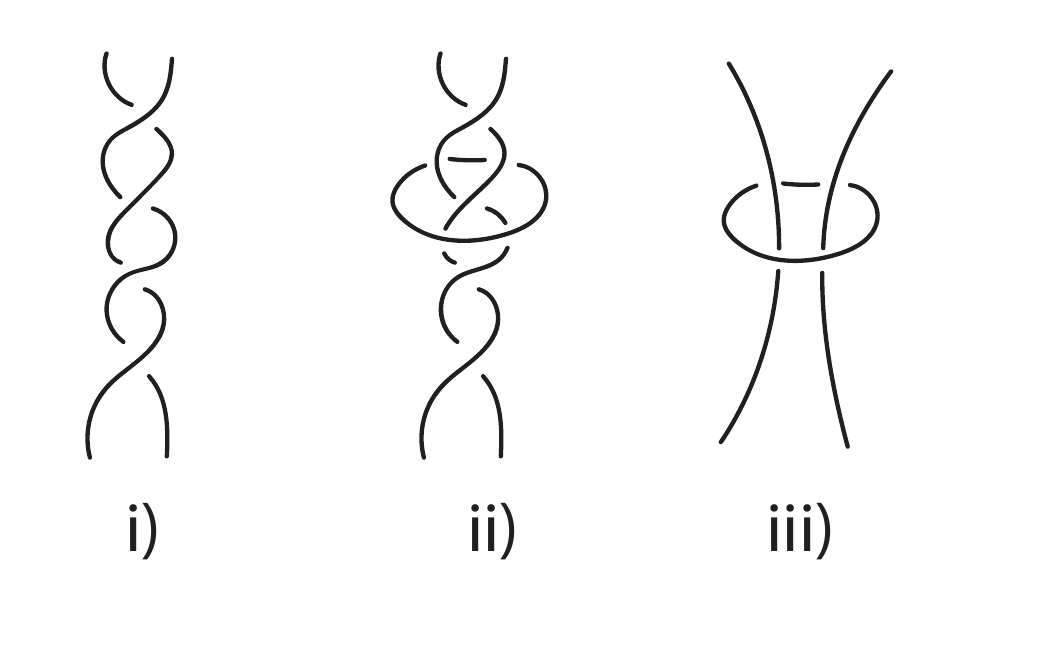}
    \caption{i) shows a twist region. ii) shows an augmented twist region, and iii) shows an augmenting component after untwisting.}
    \label{fig: augmenting twist region}
\end{figure}

Since $K_2$ is alternating, we can choose the cork $C_2$ so that the resultant link $L_i'''$ obtained from the composition $(S_1 \times I, L_i'', C_i) \#_{ns} (S_2 \times I, K_2, C_2)$  is cellular alternating and weakly prime once we drop the augmented components.  By results of \cite{Augmented alternating}, $L_i'''$ is  tg-hyperbolic. 

We know that each component of a link is a cusp of the manifold in hyperbolic 3-space. Furthermore, it was proved in \cite{Adams cusp vol} that if $M$ is an $n$-cusped hyperbolic 3-manifold of finite volume, then $vol(M) \geq n v_0$, where $v_0$ is the volume of an ideal regular tetrahedron in $H^3$. Since we can let $n = 2i$, which is the number of augmenting circles, the volumes of the complements of $L_i'''$ must be approaching infinity as $i$ grows.

For each value of $i$, we can do half-twists along the requisite augmenting components of $L_i'''$  so that there is a surgery on the augmenting components that yields the composition of ($S_1 \times I, K_1, C_1^i)$. As the longitudinal surgery coefficients grow on each, the manifolds obtained will be tg-hyperbolic and their volumes approach the volume of the complement of $L_i'''$. By letting $j$ grow large, we can make sure the necessary longitudinal surgeries will both yield tg-hyperbolic manifolds and make the volume arbitrarily close for every $i$. Therefore the compositions are tg-hyperbolic with volumes limiting to infinity.

\end{proof}

Note that we expect this theorem to hold even for non-alternating tg-hyperbolic knots and links. If so, then it shows that for any two tg-hyperbolic virtual knots, there are an infinite number of different compositions with volumes approaching infinity.

\begin{conjecture}\label{conj: nonalternating vol to infinity}
Given two tg-hyperbolic virtual knots or links $K_1, K_2$ with genera $g_1, g_2 > 0$, there exists a sequence of compositions $W_i = (S_1 \times I, K_1, C_1^i) \#_{ns} (S_2 \times I, K_2, C_2)$ such that $$\lim_{i \to \infty}(vol(W_i)) = \infty.$$
\end{conjecture}

We now consider virtual knots and links such that they are alternating in the sense that if we ignore the virtual crossings, the remaining classical crossing alternate.

\begin{lemma} \label{alternatingnonsingular}
Let $K$ be an alternating virtual knot with genus $g \geq 1$. Then $K$ has no singular curves for any projection to that surface, and all possible corks for that surface are nonsingular.
\end{lemma}

\begin{proof}
In its minimal genus thickened surface, the knot must have all regions simply connected, or otherwise we could find a reducing annulus. Given a connected alternating projection on an orientable surface with all regions topological disks, it can always be checkerboard colored. In particular, as in Figure \ref{fig: orientation}, each region has either a clockwise rotation or a counterclockwise rotation, and two regions that share an edge must have opposite rotations. So if we color regions with a clockwise rotation to be white and regions with a counterclockwise region to be black, we obtain a checkerboard coloring. Now, as we perform Reidemeister moves to the link projection in the surface, no longer using clockwise or counterclockwise rotations to determine coloring of regions,  we can maintain the fact we have a checkerboard coloring.  

But if there were a singular curve, it would intersect only one edge, meaning a region would be adjacent to itself, contradicting the fact adjacent regions must have distinct colors.

\begin{figure}[htbp]
    \centering
    \includegraphics[width=0.6\textwidth]{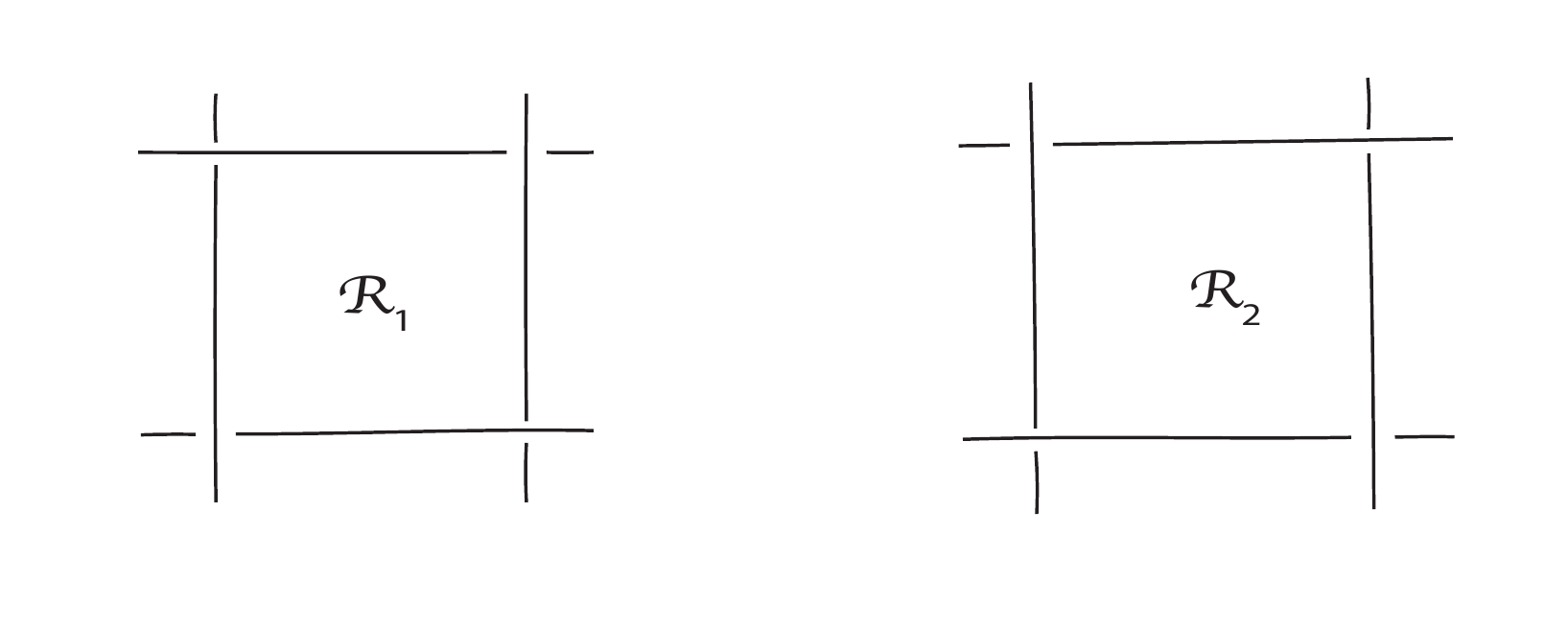}
    \caption{$R_1$ is has a clockwise orientation, and $R_2$ has a counterclockwise orientation.}
    \label{fig: orientation}
\end{figure} 

\end{proof}

\begin{theorem}\label{thm: alternatingcork} Let $K$ be a non-classical alternating virtual knot that has a weakly prime projection on $S$. Then for any choice of alternating cork, $(S \times I, K, C)$ is hyperbolically composable.
    \end{theorem}


    \begin{proof} The knot $K$ appears on $S$ as a cellular alternating link. Thus, we know $(S \times I, K, C)$ is tg-hyperbolic by Theorem 1 of \cite{small18} or Theorem 1.1 of  \cite{HP}.  We show that composition of $(S \times I,K, C)$ with itself will be tg-hyperbolic.  Note that the cork hits a single strand of the link such that to one side of the cork, the outgoing strand goes under a crossing and to the other side, the outgoing strand goes over a crossing. When we compose the knot with itself, we rotate a copy of the initial projection by $180^\circ$ and therefore each of the two connections we make matches a strand that goes over a crossing to a strand that goes under a crossing, so the resulting projection is alternating. Thus, the composition is a reduced cellular alternating projection of a knot $K'$. The fact $K$ is weakly prime implies that $K'$ is weakly prime as well. Then again, by Theorem 1 of \cite{small18} or Theorem 1.1 of \cite{HP}, the  resulting knot $K'$ is tg-hyperbolic, and $(S \times I,K, C)$ is hyperbolically composable.
    \end{proof}

\section{Explicit Examples} \label{sect: explicit examples}

We include a few explicit examples here. We also include Table \ref{table: examples} of volumes of virtual knots with corks removed to provide volumes for the examples and to allow the reader to play with different compositions without having to calculate volumes of manifolds. In the first column, if the cork is nonsingular, the nonsingular volume $vol_{ns}(S \times I, K, C)$ is listed and the second volume has N/A for not applicable.  All of these examples are hyperbolically composable. 

If the cork is singular, the volume $\frac{1}{4}vol(D(D_{ns}(S_1 \times I, K_1, C_1)) \pmc T)$ is listed in the first column and the singular volume $vol_{s}(S \times I, K, C)$ is listed in the second column. 

\begin{example}
The first example we include is a nonsingular composition of the virtual trefoil knot (the knot 2.1 in Jeremy Green's knot table \cite{Knot Table}) with the virtual figure-eight knot (the knot 3.2 in Jeremy Green's knot table \cite{Knot Table}). The composition is shown in Figure \ref{fig: ns tref with ns fig 8 example}. It is easy to see that composing with diagrams of virtual knots is equivalent to composing in thickened surfaces. 

\begin{figure}[htbp]
    \centering
    \includegraphics[width=0.7\textwidth]{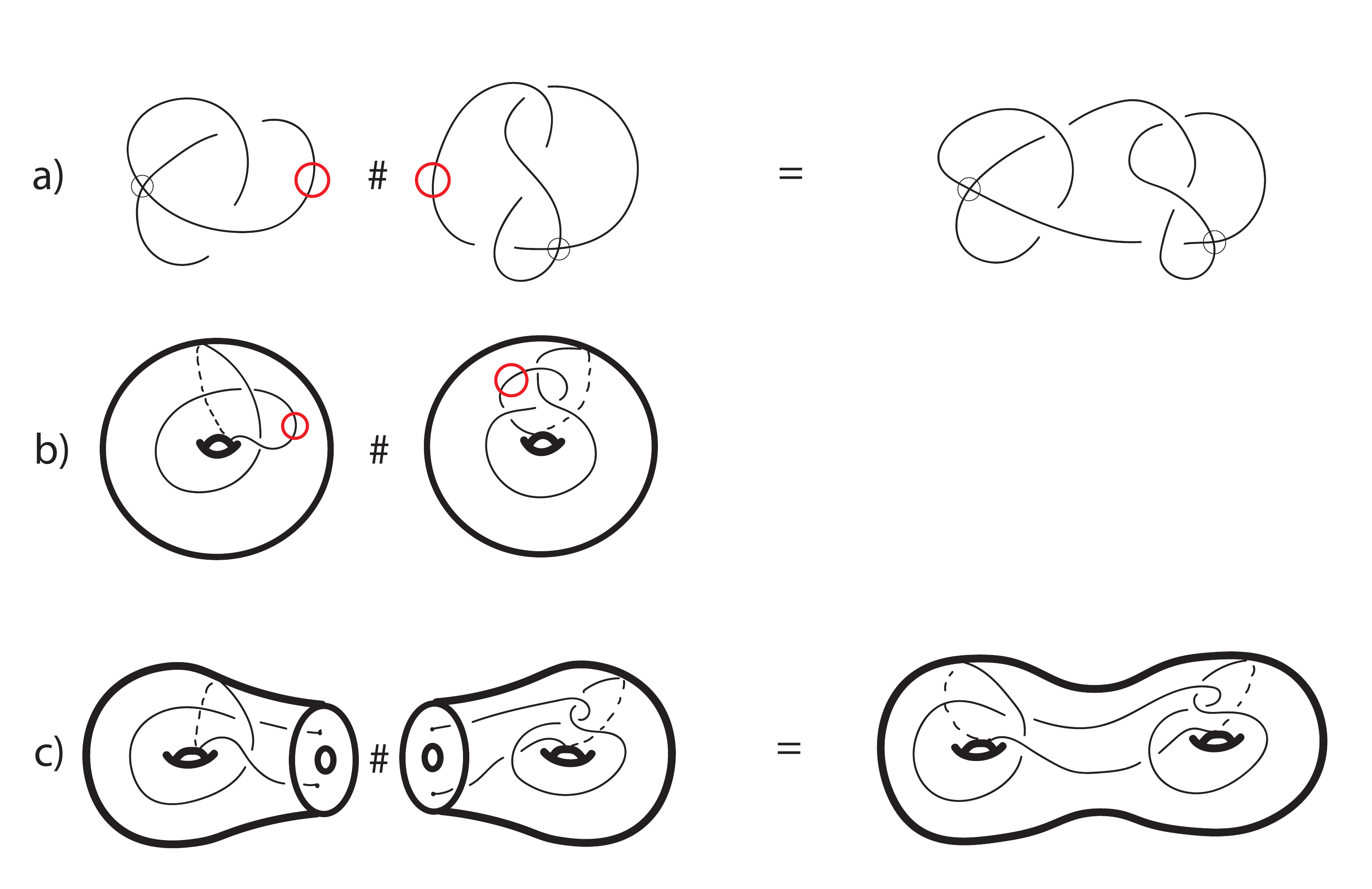}
    \caption{The composition in diagrams appears in a).  The knots appear in their thickened surfaces in b), and the composition of the knot complements appears in c). The corks are indicated by red circles.}
    \label{fig: ns tref with ns fig 8 example}
\end{figure}

Setting $K_1$ to be the virtual trefoil and $K_2$ to be the virtual figure-eight, Theorem \ref{thm: nonsing compose nonsing vol bound} says that the volume of the composition should be bounded below by the sum of half the volumes of the respective cork doubles. And indeed, we see that $K_1$ contributes a bound of 9.4158416835 and $K_2$ contributes a bound of 13.5043855968, for a total lower bound on the volume of the composition of 22.92022727. Computing the actual volume of the composition in SnapPy yields a volume of 26.236005. 
\end{example}




\begin{example}
Here, we give an example of singular composition by again composing the virtual trefoil  with the virtual figure-eight knot, but this time with singular corks as in  Figure \ref{fig: trefoil compose singular fig 8}.  Again, the composition of diagrams is equivalent to the composition in thickened surfaces. 

 Since both knots are genus one and the composition is singular, we can use Corollary \ref{cor: genus 1 sing comp} to say the sum of the volumes of the original knots should be the resulting volume. The volume of the virtual trefoil is 5.3334895 and the volume of the virtual figure-eight is 7.7069118. Computing the volume in SnapPy of the composition  does give a volume  of  13.0404013, which is their sum, as expected. 

\begin{figure}[htbp]
    \centering
    \includegraphics[width=0.6\textwidth]{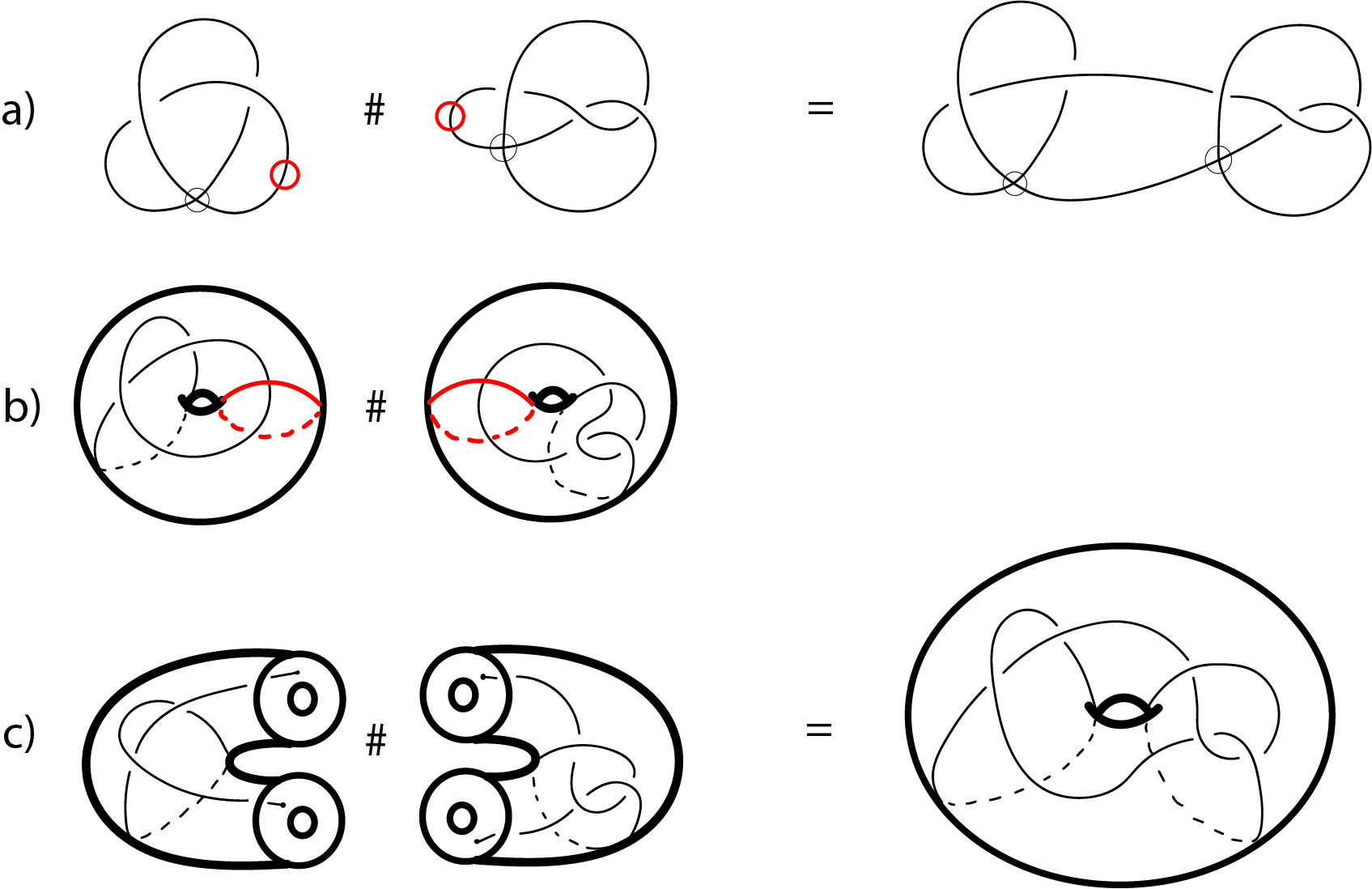}
    \caption{An example of singular composition of the virtual trefoil with the virtual figure-eight knot.}
    \label{fig: trefoil compose singular fig 8}
\end{figure}

\end{example}


\begin{example}
Next, we compose the virtual trefoil $K_1$ with the virtual figure-eight knot $K_2$, only this time we take the cork for the virtual trefoil knot to be singular, while the cork for the virtual figure-eight knot is nonsingular.  Thus the  composition is nonsingular. We show it in Figure \ref{fig: trefoil compose nonsingular fig 8}.

\begin{figure}[htbp]
    \centering
    \includegraphics[width=0.6\textwidth]{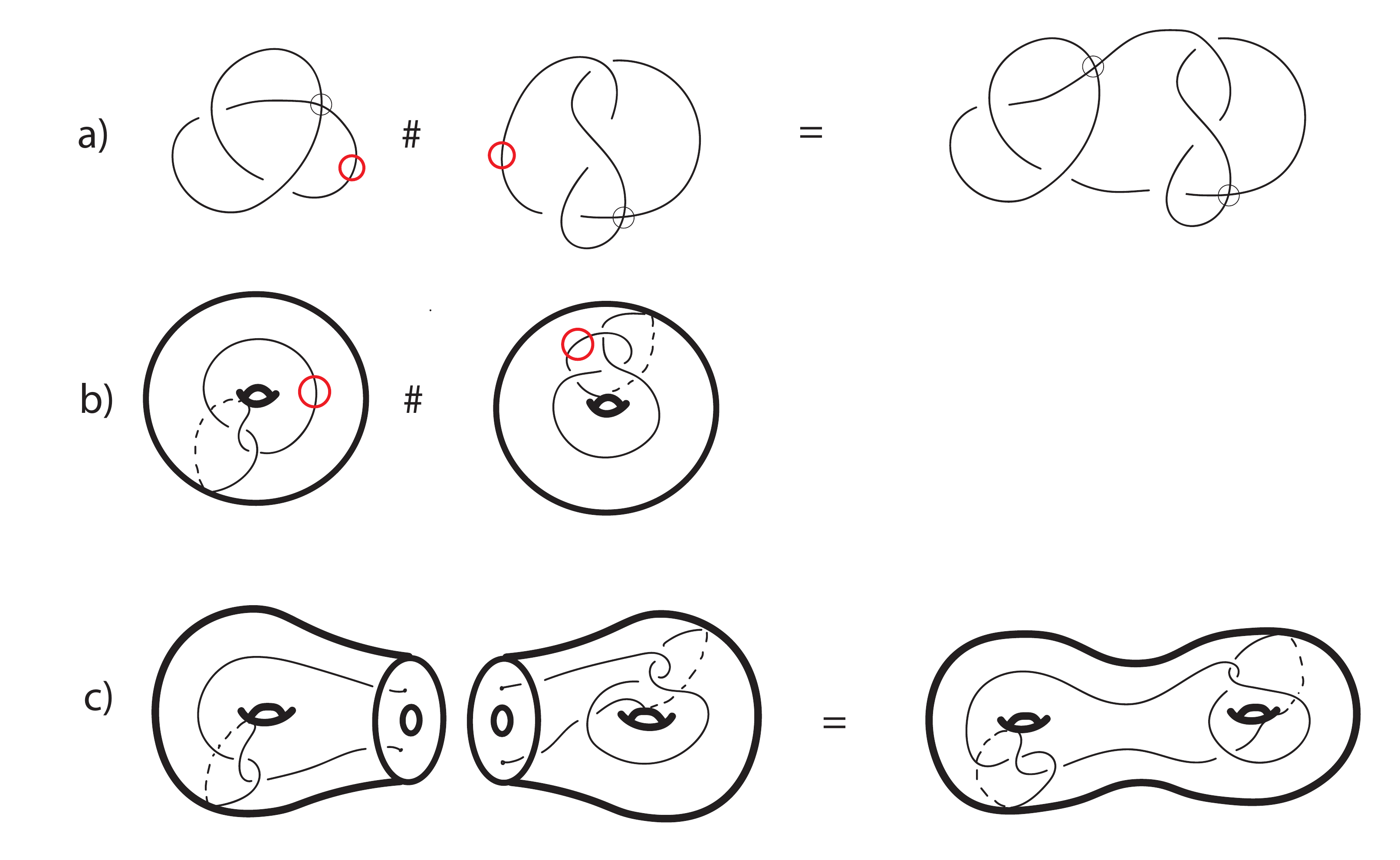}
    \caption{An example of nonsingular composition between the virtual trefoil with a singular cork and the virtual figure-eight knot with a nonsingular cork.}
    \label{fig: trefoil compose nonsingular fig 8}
\end{figure}


We obtain the following volume bounds from Theorem \ref{thm: nonsing compose sing vol bound}: 
\begin{align*}
    vol((S_1 \times I, K_1, C_1) \#_{ns} (S_2 \times I, K_2, C_2)) & \geq \frac{1}{4}vol(DM_1' \pmc T) + \frac{1}{2}vol(D_{ns}(S_2 \times I, K_2, C_2)) \\
    & \geq 10.149416064 + 13.5043855968\\
    & \geq 23.6538016608 
\end{align*}
Computing the volume of $(S_1 \times I, K_1, C_1) \#_{ns} (S_2 \times I, K_2, C_2)$ in SnapPy, we see that the computed volume is greater than the sum of the bounds: $26.3735 \geq 23.6538016608$. 
\end{example}

\newpage
\begin{table}[htbp] 
\begin{center}
 \begin{tabular}{||c | c | c | c ||} 
 \hline
 Diagram of $K$ & $(S \times I) \setminus K$ & $vol_{ns}(S \times I K, C)$ & $vol_s(S \times I, K, C)$ \\ [0.5ex] 
 \hline 
 \includegraphics[scale=0.4]{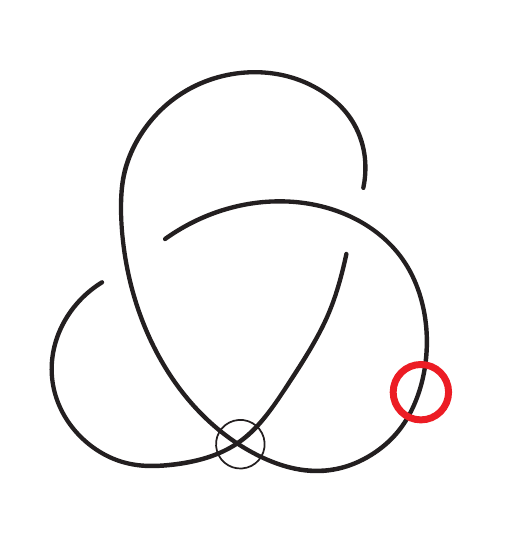} & \includegraphics[scale=0.4]{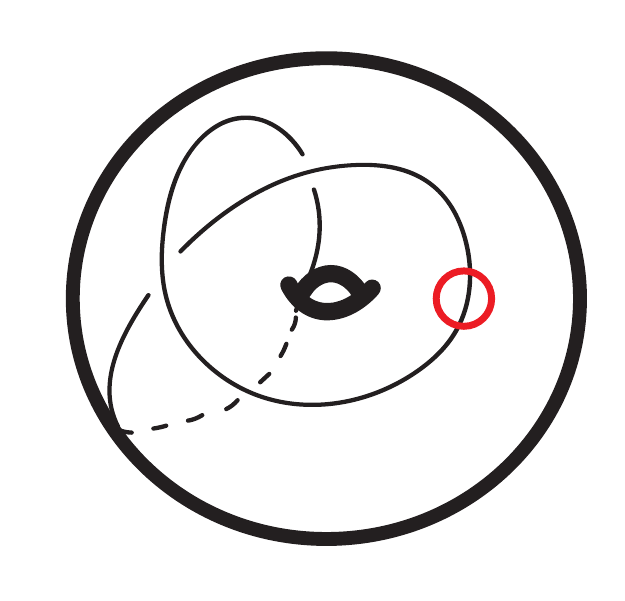} & 10.149416064 
 & 5.33348956690 \\
 \hline 
 \includegraphics[scale=0.4]{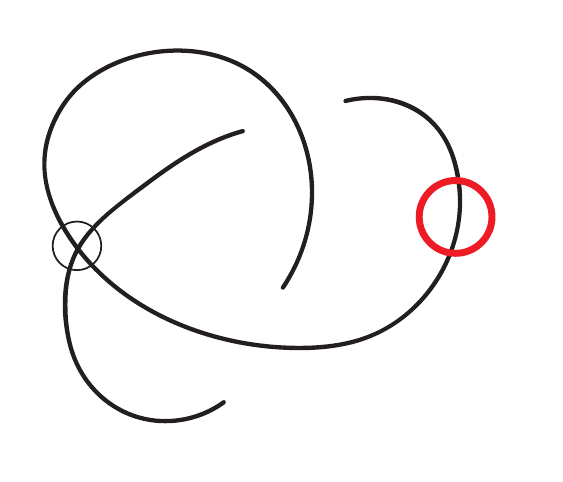} & \includegraphics[scale=0.4]{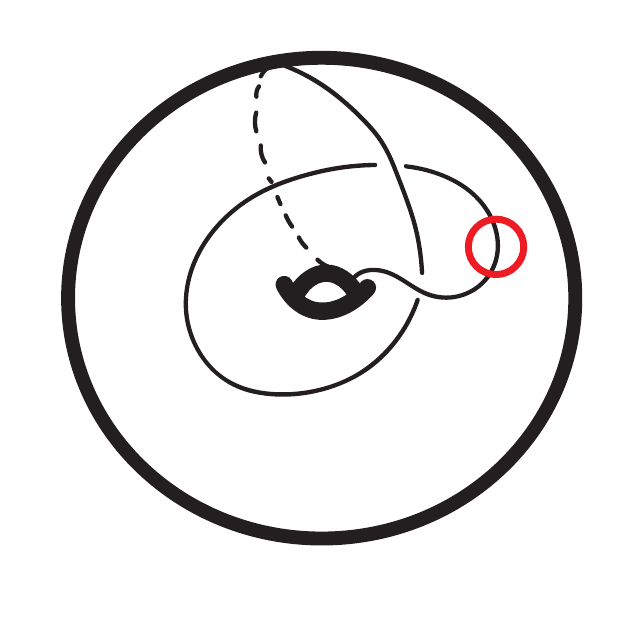} & 9.4158416835 & N/A \\
 \hline 
 \includegraphics[scale=0.4]{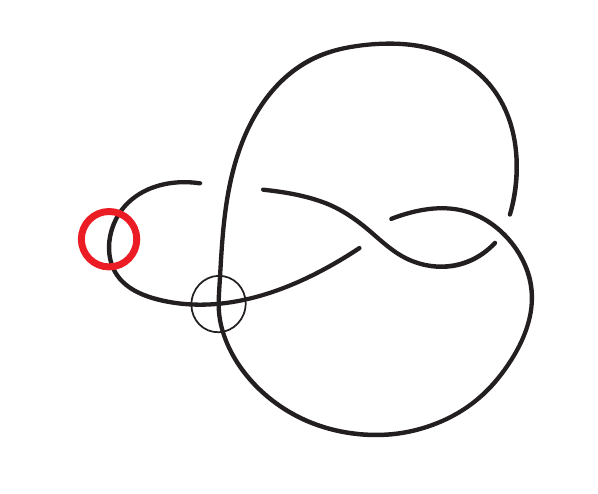} & \includegraphics[scale=0.4]{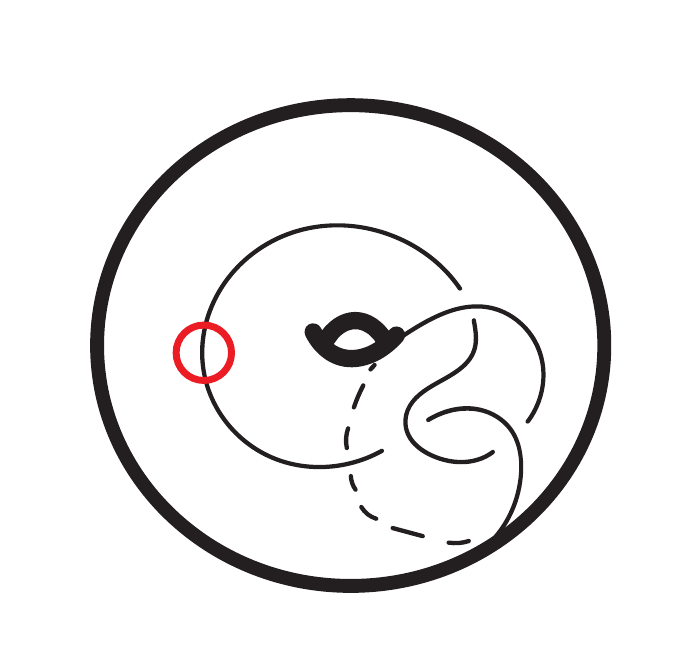} & 12.8448530045 
 & 7.70691180281 \\
 \hline
 \includegraphics[scale=0.4]{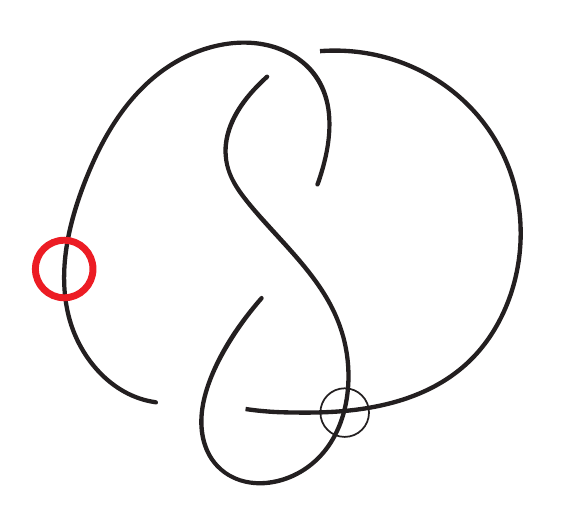} & \includegraphics[scale=0.4]{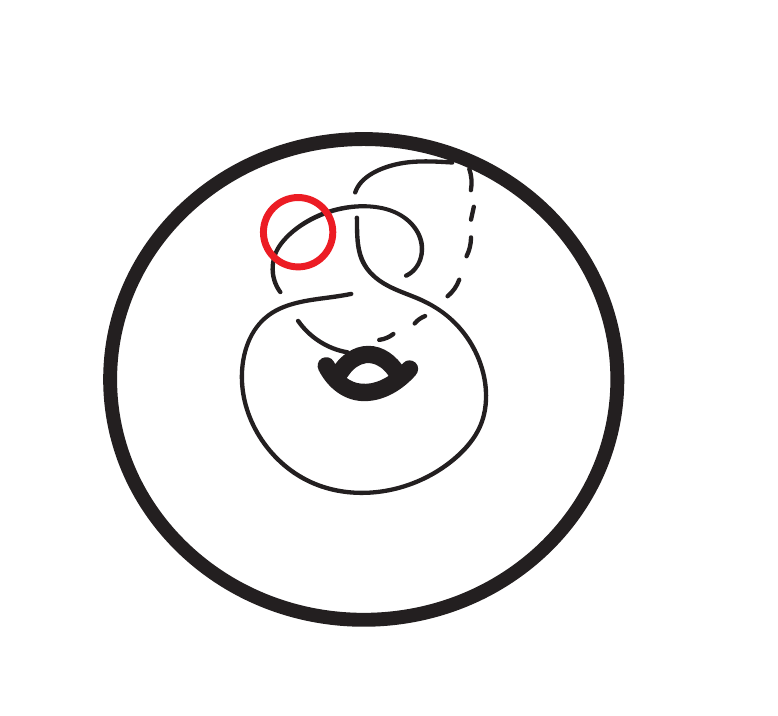} & 13.5043855968 & N/A \\
 \hline 
  \includegraphics[scale=0.4]{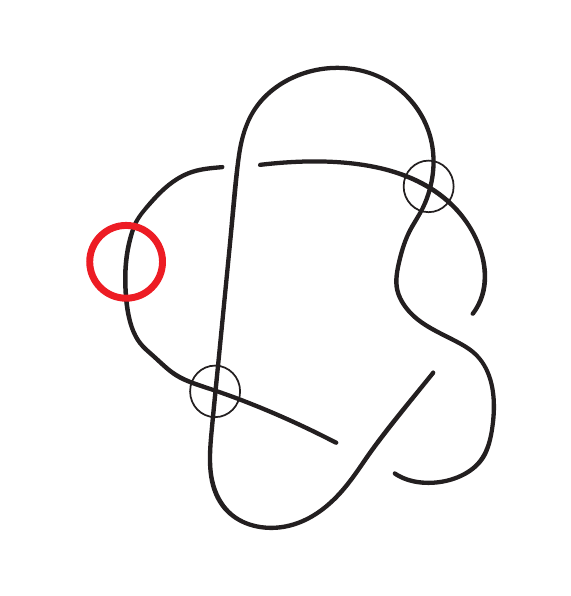} & \includegraphics[scale=0.4]{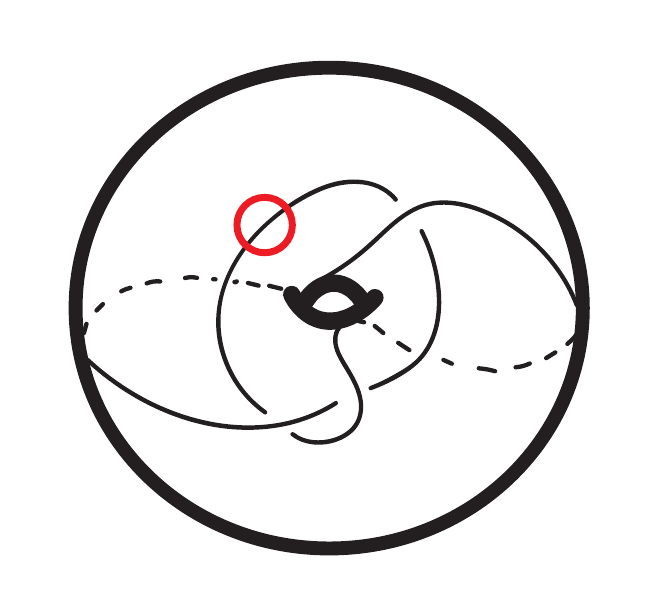} & 12.9446980685 & N/A \\
 \hline 
   \includegraphics[scale=0.4]{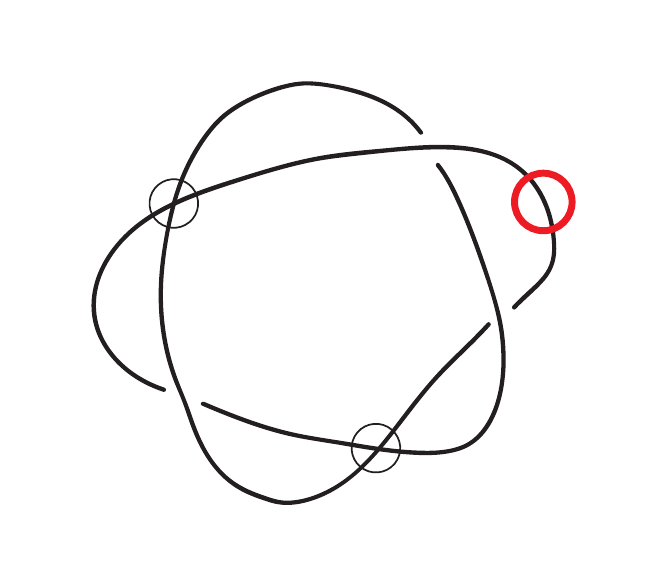} & \includegraphics[scale=0.4]{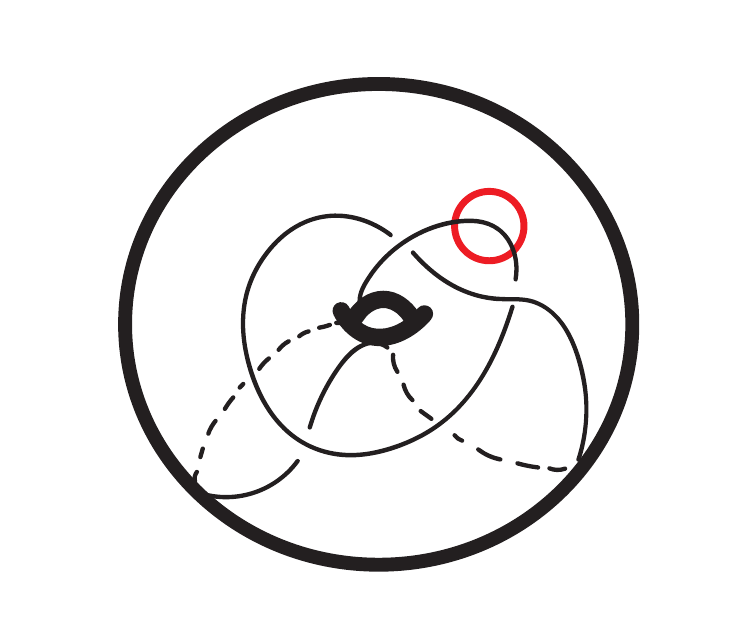}  & 15.8327412531 & N/A \\
 \hline  
\end{tabular}
\end{center}
\caption{$vol_{ns}(S \times I, K,C)$ and $vol_{s}(S \times I, K, C)$ for some examples.}
\label{table: examples}
\end{table}

\newpage 

\newpage

\newpage


\end{document}